\pdfoutput=1

\documentclass[
	bigMargin = false,
	font = palatino,
	final
	]{MyClass}

\usepackage{myMathHeader}
\let\ConnSpace\relax
\let\GauGroup\relax
\let\GauAlgebra\relax
\MyNewMathOperator{\ConnSpace}		{command={\mathbrush{C}}, sort={C}, display={$\ConnSpace(P)$}, description={Space of connections of the principal bundle $P$}}
\MyNewMathOperator{\GauGroup}			{command={\mathbrush{G}}, sort={G}, display={$\GauGroup(P)$}, description={Group of gauge transformation of the (principal) bundle $P$}}
\MyNewMathOperator{\GauAlgebra}		{command={\mathrm{L}\mathbrush{G}}, sort={LG}, display={$\GauAlgebra(P)$}, description={Lie algebra of infinitesimal gauge transformation of the principal bundle $P$}}

\NewDocumentCommand { \normalForm }{ m }{
	{#1}_{\textnormal{NF}}
}

\NewDocumentCommand { \normalFormGroup }{ m m }{
	{#2}_{{#1} \textnormal{-NF}}
}

\Title{Normal form of equivariant maps \\ in infinite dimensions}

\Abstract{
	Local normal form theorems for smooth equivariant maps between infinite-dimensional manifolds are established.
	These normal form results are new even in finite dimensions.
	The proof is inspired by the Lyapunov--Schmidt reduction for dynamical systems and by the Kuranishi method for moduli spaces.
	It uses a Slice Theorem for Fréchet manifolds as the main technical tool.
	As a consequence, the abstract moduli space obtained by factorizing a level set of the equivariant map with respect to the group action carries the structure of a Kuranishi space, \ie, such moduli spaces are locally modeled on the quotient by a compact group of the zero set of a smooth map.
	The general results are applied to the moduli space of anti-self-dual instantons, the Seiberg--Witten moduli space and the moduli space of pseudoholomorphic curves.
}



\Institution{itp}{
	Institut für Theoretische Physik, 
	Universität Leipzig, 
	04009 Leipzig, Germany}

\Institution{delft}{
	Institute of Applied Mathematics,
	Delft University of Technology,
	2628 XE Delft, Netherlands}

\Author[
	affiliation = delft,
	email = T.Diez@tudelft.nl,
	]{diez}{Tobias Diez}
\Author[
	affiliation = itp,
	email = {rudolph@rz.uni-leipzig.de}
	]{rudolph}{Gerd Rudolph}

\begin{document}
	
\MakeTitle

\tableofcontents


\section{Introduction}

Moduli spaces parametrize solutions of partial differential equations up to some natural notion of equivalence.
They play an essential role in theoretical physics as well as in pure mathematics.
In (gauge) field theory, they represent the reduced phase spaces after the (gauge) symmetry has been divided out, see \parencite{RudolphSchmidtEtAl2002} and references therein.
The topological and inherently non-perturbative aspects of the physical system are often manifested in the geometry of the moduli space.
On the other hand, the structure of the moduli space encodes some astounding topological and geometrical properties.
The topological data extracted from this auxiliary moduli space serve as a powerful nonlinear invariant of the original manifold.
Important examples can be found in many areas: gauge theory (flat connections \parencite{AtiyahBott1983}, anti-self-dual instantons \parencite{Donaldson1983}, Seiberg--Witten monopoles \parencite{SeibergWitten1994}), symplectic geometry (pseudoholomorphic curves \parencite{Gromov1985}, symplectic field theory \parencite{EliashbergGiventalHofer2000}), complex geometry \parencite{Kuranishi1965} and string theory.

For applications in both geometry and physics, a deep understanding of the local structure of the moduli space is essential.
Of particular importance is the formation of singularities due to points having a non-trivial stabilizer under the symmetry group action.
Although relying on similar techniques, so far these fundamental features are analyzed only on a case by case basis for each moduli space separately.

\emph{In this paper, we provide a general framework that gives a unified approach to these differential-geometric moduli spaces.}
Specifically, we establish a convenient normal form for a large class of nonlinear differential equations with symmetries.
Furthermore, we show that the corresponding moduli space of solutions can be endowed with the structure of a Kuranishi space, which roughly speaking means that it can be locally identified with the quotient of the zero set of a smooth map by the linear action of a compact group.
Our approach is inspired by the ideas that underlie the Lyapunov--Schmidt reduction for dynamical systems and the Kuranishi method for moduli spaces in differential geometry.
As we discuss in the last section, our general framework applies to the fundamental moduli spaces of anti-self-dual instantons, Seiberg--Witten monopoles and pseudoholomorphic curves; leading to simplified and unified proofs that these moduli spaces have natural Kuranishi charts.

Our results are phrased in terms of smooth equivariant maps between infinite-dimensional manifolds modeled on locally convex spaces.
In field theory and global analysis, the maps under consideration are usually given by partial differential operators between spaces of sections.
As such they give rise to smooth maps between appropriate Sobolev completions.
On the other hand, the symmetry action often involves compositions of maps and thus fails to be differentiable as a map between spaces of sections of a given Sobolev class.
For example, the group of diffeomorphisms of a fixed Sobolev regularity is a Banach manifold as well as a topological group but not a Lie group, because the group operation is not differentiable.
When working with smooth objects these problems disappear, and the group of smooth diffeomorphisms is a bona fide Lie group modeled on a Fréchet space.
In order to include these important cases, we consider throughout the paper infinite-dimensional manifolds modeled on Fréchet or even more general locally convex spaces.
The approach via Fréchet spaces has also the advantage that the geometric arguments in the applications are simpler, because one does not have to deal with issues originating in the low regularity of the sections or in the loss of differentiability.

Beyond the Banach context, the classical Banach Inverse Function Theorem used in the proof of \cref{prop:normalFormMap:banach} has to be replaced by a different version.
We will use Glöckner's Inverse Function Theorem for maps between Banach spaces with parameters in a locally convex space and the Nash--Moser Theorem in the tame Fréchet setting.
Moreover, the normal form theorems are phrased and proved in a modularized way, leading to a flexible general framework in which other analytical setups can be included in a \textquote{plug and play} fashion based on other Inverse Function Theorems.

As another approach to the issues that composition is not smooth relative to a fixed Sobolev regularity, Hofer, Wysocki and Zehnder have introduced the scale calculus, see, \eg, \parencite{HoferWysockiZehnder2014,HoferWysockiZehnder2017} and references therein.
In this approach, one works with a sequence of Banach spaces while the Nash--Moser approach focuses on the limit space, see \parencite{Gerstenberger2016} for a detailed comparison.
The scale calculus is tailored to the elliptic setting one encounters in symplectic field theory.
We emphasize that the Nash--Moser theorem covers these cases as well, see \cref{sec:normalFormMap:elliptic,sec:pseudoholomorphicCurves}, but additionally allows for applications that go well beyond the elliptic setting.
This will be crucial for problems where classical Banach space methods do not apply and more sophisticated Nash--Moser arguments are needed, such as deformation and normal form problems in the theory of fibrations \parencite{Hamilton1978} and Poisson manifolds \parencite{Conn1985,Marcut2014}.  
Finally, the general framework developed in this paper lays the foundation for the singular symplectic reduction in infinite dimensions, which will be presented elsewhere \parencite{DiezThesis,DiezSingularReduction}.

The paper is structured as follows.
\begin{description}[leftmargin=0cm]
\item[\cref{sec:normalFormLinearMap}]
	We begin by considering the linear setting and determine under which conditions a continuous linear map between locally convex spaces can be brought into a normal form, that is, it factorizes through a topological isomorphism.
	An operator admitting such a factorization is called regular.
	Fredholm operators, and in particular elliptic operators, are important examples of regular operators.
	As a preparation for the nonlinear case, we also discuss regularity of families of linear maps depending continuously on a parameter and of chain complexes.
\item[\cref{sec:normalFormMap}]
	Next, we discuss the local behavior of a smooth map between locally convex manifolds.
	Unifying the concepts of immersion, submersion and subimmersion in one framework, the notion of a normal form of a nonlinear map is introduced.
	Using versions of the Inverse Function Theorem, we establish \cref{prop:normalFormMap:finiteDim,prop:normalFormMap:banachTarget,prop:normalFormMap:banachDomain,prop::normalFormMap:tameFrechet,prop::normalFormMap:elliptic} which show that a given map can be brought into such a normal form in various functional-analytic settings under suitable conditions. 
	These normal form theorems provide a unified approach to the Immersion Theorem, the Level Set Theorem and the Constant Rank Theorem in the setting of locally convex manifolds and tame Fréchet manifolds.
\item[\cref{sec:normalFormEquivariantMap}]
	We introduce the concept of an equivariant normal form and provide suitable conditions which ensure that an equivariant map can be brought into such a normal form, resulting in \cref{prop:normalFormEquivariantMap:abstract} and its variants.
	Besides the normal form results of \cref{sec:normalFormMap}, the main technical tool is the existence of slices for actions on Fréchet manifolds as established in \parencite{DiezSlice}.
	We investigate the local structure of the moduli space obtained by taking the quotient of a level set of the equivariant map by the group action.
	Under the assumption that the map can be brought into an equivariant normal form, the corresponding moduli space has the structure of a Kuranishi space, which roughly speaking means that it can be locally identified with the quotient of the zero set of a smooth map with respect to the linear action of a compact group.
	Moreover, we find additional conditions on the normal form which ensure that the moduli space is stratified by orbit types.
	Finally, to show the utility of this novel framework, we apply the general theory to the example of the moduli space of anti-self-dual Yang–Mills connections, to the Seiberg--Witten moduli space and to the moduli space of pseudoholomorphic curves.
	The upshot is that these moduli spaces have natural Kuranishi charts.
	Since no Sobolev techniques are used, the arguments are streamlined and simplified.
	For examples, we do not face any issues coming from functions of low regularity that complicate the analysis in the usual Banach approach.
\item[Appendix]
	In the appendix, we summarize without proofs the relevant background material concerning the calculus of infinite-dimensional manifolds with a primary focus on the Inverse Function Theorem as well as on Lie group actions.
\end{description}
	
Most of the material first appeared in the first author's thesis \parencite{DiezThesis}.

\paragraph*{Conventions}
Our main references for terminology and notation in the framework of infinite-dimensional differential geometry are~\parencite{Hamilton1982} for the tame Fréchet category and~\parencite{Neeb2006} for the general locally convex case.
As the latter is our standard setting, we use the word \textquote{manifold} to refer to an infinite-dimensional locally convex manifold; further assumptions on the model space are designated by additional qualifier such as \textquote{Fréchet or finite-dimensional manifold}.
For a Lie group action \( \Upsilon: G \times M \to M \), we also use the \textquote{dot notation} and abbreviate \( \Upsilon(g,m) \equiv g \cdot m \) for \( g \in G \) and \( m \in M \).
On the infinitesimal level, for \( A \in \LieA{g} \), we write the fundamental vector field \( A^* \) as \( A^*_m = \tangent_e \Upsilon_m (A) \equiv A \ldot m \). 
Similarly, \( \tangent_m \Upsilon_g (X) \equiv g \ldot X \) for \( X \in \TBundle_m M \).

\paragraph*{Acknowledgments}
T.~Diez was supported by the NWO grant 639.032.734 \textquote{Cohomology and representation theory of infinite dimensional Lie groups}.
We gratefully acknowledge support of the Max Planck Institute for Mathematics in the Sciences in Leipzig and of the University of Leipzig.

\section{Normal form of a linear map}
\label{sec:normalFormLinearMap}
In this section, we discuss the normal form of continuous linear maps between locally convex spaces.
Recall that every \( (m \times n) \)-matrix \( T \) with rank \( r \) can be written in the following normal form
\begin{equation}
	\label{eq:normalFormLinearMap:finiteDimFactorization}
	T = P \Matrix{0 & 0 \\ 0 & \one_{r \times r}} Q,
\end{equation} 
where \( P \) and \( Q \) are invertible matrices of type \( (m \times m) \) and \( (n \times n) \), respectively.
As we will see, a similar factorization is possible for continuous linear maps between locally convex spaces, which are relatively open and whose kernel and image are closed complemented subspaces.
We call such operators regular and their associated representation~\eqref{eq:normalFormLinearMap:finiteDimFactorization} a normal form.
As a preparation for the nonlinear case, we define and study regularity of families of linear maps depending continuously on a parameter.
With a view toward applications, we give a brief overview of the theory of Fredholm operators and of elliptic operators in the locally convex framework and, in particular, show that these operators are regular.

\subsection{Uniform regularity}
\label{sec:familyGeneralizedInverse}
Let \( X \) and \( Y \) be locally convex spaces, and let \( P \) be a neighborhood of \( 0 \) in some locally convex space.
A continuous map \( T: P \times X \to Y \) is called a \emphDef{continuous family of linear maps} if, for all \( p \in P \), the induced map \( T_p \equiv T(p, \cdot): X \to Y \) is linear.

\begin{defn}
	\label{defn:familyGeneralizedInverse:uniformRegular}
	A continuous family \( T: P \times X \to Y \) of linear maps between locally convex spaces \( X \) and \( Y \) is called \emphDef{uniformly regular} (at \( 0 \)) if there exist topological decompositions
	\begin{equation}
		\label{eq::normalFormMap:decompositionLinerSpaceKernelImage}
		X = \ker T_0 \oplus \coimg T_0, \qquad Y = \coker T_0 \oplus \img T_0,
	\end{equation}
	where \( \coimg T_0 \) and \( \coker T_0 \) are closed subspaces\footnotemark{} of \( X \) and \( Y \), and, for every \( p \in P \), the map \( \tilde{T}_p = \pr_{\img T_0} \circ \restr{(T_p)}{\coimg T_0}: \coimg T_0 \to \img T_0 \) is a topological isomorphism and the inverses \( \tilde{T}^{-1}_p \) form a continuous family \( P \times \img T_0 \to \coimg T_0 \) of isomorphisms.
	\footnotetext{The coimage and cokernel of an continuous linear map \( T: X \to Y \) are defined as \( \coimg T = X \slash \ker T \) and \( \coker T = Y \slash \img T \), respectively. There exists, of course, no canonical realization of these quotient spaces as subspaces of \( X \) and \( Y \). Nonetheless, the choice of topological complements \( A \) and \( B \) of \( \ker T \) and \( \img T \), respectively, leads to the identifications \( \coimg T \isomorph A \) and \( \coker T \isomorph B \). It is in this sense and with a slight abuse of notation that we view \( \coimg T \) and \( \coker T \) as subspaces of \( X \) and \( Y \), respectively.}
\end{defn}
In the case \( P = \set{0} \), the notion of uniform regularity reduces to the notion of a relatively open operator \( T: X \to Y \) with closed complemented kernel and image. 
We refer to this situation by saying that \( T \) is a \emphDef{regular} operator\footnotemark{}.
\footnotetext{We use the word \enquote{operator} interchangeably with \enquote{continuous linear map}.}
Every regular operator \( T \) can be written in a normal form similar to~\eqref{eq:normalFormLinearMap:finiteDimFactorization}: 
\begin{equation}
	T  = P \Matrix{0 & 0 \\ 0 & \tilde{T}} Q,
\end{equation}
where \( Q: X \to \ker T \oplus \coimg T \) and \( P: \coker T \oplus \img T \to Y \) are the natural isomorphisms determined by the decompositions~\eqref{eq::normalFormMap:decompositionLinerSpaceKernelImage}, and \( \tilde{T}: \coimg T \to \img T \) is a topological isomorphism.
We call \( \tilde{T} \) (together with the isomorphisms \( Q \) and \( P \)) a normal form of \( T \).

If the space of invertible maps is open in the space of all continuous linear maps, then, for every continuous family \( T: P \times X \to Y \) with \( T_0 \) being regular, one can shrink \( P \) to pass to a uniformly regular family.
This openness property fails when one leaves the Banach realm.
However, when it does hold, uniform regularity reduces to a condition at one point. 
\begin{lemma}
	\label{prop::normalFormLinearMap:generalizedInverseWhenImageFiniteDim}
	\label{prop::normalFormLinearMap:generalizedInverseWhenTargetFiniteDim}
	\label{prop::normalFormLinearMap:generalizedInverseWhenDomainFiniteDim}
	Let \( T: P \times X \to Y \) be a continuous family of linear maps between locally convex spaces \( X \) and \( Y \). 
	If \( \img T_0 \) is finite-dimensional, then \( T \) is uniformly regular after possibly shrinking \( P \).
\end{lemma}
\begin{proof}
	Since \( T_0 \) has a finite-dimensional range, by \parencite[Proposition~20.5.5]{Koethe1983}, \( \img T_0 \) is closed and has a topological complement.
	Moreover, \( \ker T_0 \) has finite codimension in \( X \) and hence is topologically complemented according to \parencite[Proposition~15.8.2]{Koethe1983}.
	The maps \( \tilde{T}_p: \coimg T_0 \to \img T_0 \) are continuous linear maps between finite-dimensional spaces.
	Since \( \tilde{T}_0 \) is a bijection, the openness of the set of invertible operators implies that \( \tilde{T}_p \) is a topological isomorphism for \( p \in P \) close enough to \( 0 \).
\end{proof}

Uniform regularity implies a semi-continuity property of the kernel and the image.
Similar semi-continuity properties are well known for families of Fredholm operators between Banach spaces \parencite[Corollary~19.1.6]{Hoermander2007}.
\begin{lemma}
	Let \( T: P \times X \to Y \) be a continuous family of linear maps between locally convex spaces \( X \) and \( Y \).
	If \( T \) is uniformly regular, then the following holds:
	\begin{thmenumerate}
		\item
			The kernel of \( T \) is upper semi-continuous at \( 0 \) in the sense that \( \ker T_p \subseteq \ker T_0 \) for all \( p \in P \).
		\item
			The image of \( T \) is lower semi-continuous at \( 0 \) in the sense that \( \img T_p \supseteq \img T_0 \) for all \( p \in P \).
			\qedhere
	\end{thmenumerate}
\end{lemma}
\begin{proof}
	The inclusions \( \ker T_p \subseteq \ker T_0 \) and \( \img T_p \supseteq \img T_0 \) need to be valid, because otherwise \( \tilde{T}_p = \pr_{\img T_0} \circ \restr{(T_p)}{\coimg T_0} \) cannot be an isomorphism from \( \coimg T_0 \) to \( \img T_0 \).
\end{proof}
Uniform regularity is tightly connected to the invertibility of an extended operator.
\begin{thm}
	\label{prop:normalFormLinearMap:uniformRegularCharacterization}
	Let \( T: P \times X \to Y \) be a continuous family of linear maps between locally convex spaces \( X \) and \( Y \).
	Then, the following are equivalent:
	\begin{enumerate}
	 	\item
	 		\label{prop:normalFormLinearMap:uniformRegularCharacterization:regular}
	 		\( T \) is uniformly regular.
	 	\item
	 		\label{prop:normalFormLinearMap:uniformRegularCharacterization:external}
	 		There exist locally convex spaces \( Z^\pm \), continuous linear maps \( T^+: Z^+ \to Y \) and \( T^-: X \to Z^- \), and continuous families of linear maps \( S: P \times Y \to X \), \( S^-: P \times Z^- \to X \), \( S^+: P \times Y \to Z^+ \) and \( S^{-+}: P \times Z^- \to Z^+ \) with \( S^{-+}_0 = 0 \) such that
	 		\begin{equation}
	 			\label{prop:normalFormLinearMap:uniformRegularCharacterization:external:inverse}
	 			\Matrix{T_p & T^+ \\ T^- & 0}^{-1} = \Matrix{S_p & S^-_p \\ S^+_p & S^{-+}_p},
	 		\end{equation}
			holds for all \( p \in P \) and such that the operators 
	 		\begin{equation}
	 			\Gamma_p \equiv \Matrix{
					\pr_{\ker T_0} \circ \restr{(S_p)}{\coker T_0}
					&
					\pr_{\ker T_0} \circ S^-_p
					\\
					\restr{(S^+_p)}{\coker T_0}
					&
					S^{-+}_p
				}: \coker T_0 \oplus Z^- \to \ker T_0 \oplus Z^+
	 		\end{equation}
	 		are invertible for all \( p \in P \) and their inverses form a continuous family.
	 		\qedhere
	 \end{enumerate}
\end{thm}
For the proof, we need the following basic result about the invertibility of block matrices in terms of the Schur complement.
\begin{lemma}
	\label{prop:normalFormLinearMap:inverseBlockMatrix}
	Let \( A_{11}: X_1 \to Y_1 \), \( A_{12}: X_2 \to Y_1 \), \( A_{21}: X_1 \to Y_2 \) and \( A_{22}: X_2 \to Y_2 \) be continuous linear maps between locally convex spaces such that
	\begin{equation}
		\Matrix{A_{11} & A_{12} \\ A_{21} & A_{22}}^{-1} = \Matrix{B_{11} & B_{12} \\ B_{21} & B_{22}}
	\end{equation}
	for continuous linear maps \( B_{ij} \) for \( i,j = 1,2 \).
	\begin{thmenumerate}
		\item
			\label{prop:normalFormLinearMap:inverseBlockMatrix:schurComplement}
			If \( B_{22} \) is a topological isomorphism, then so is \( A_{11} \), and its inverse is given by
			\begin{equation}
				A_{11}^{-1} = B_{11}^{} - B_{12}^{} B_{22}^{-1} B_{21}^{}.
			\end{equation}
		\item
			\label{prop:normalFormLinearMap:inverseBlockMatrix:projection}
			If \( B_{22} = 0 \), then \( A_{11} B_{11} \) and \( B_{11} A_{11} \) are idempotent and satisfy
			\begin{equation}
				\img \,(A_{11} B_{11}) = \img A_{11}, \qquad \quad \ker \, (B_{11} A_{11}) = \ker A_{11}.
				\qedhere
			\end{equation}
	\end{thmenumerate}
\end{lemma}
\begin{proof}
	The first statement is \parencite[Lemma~3.1]{SjoestrandZworski2007} and can be verified by a direct calculation.
	The second statement follows from the identity
	\begin{equation}
		A_{11} B_{11} A_{11} = A_{11} (\id_{X_1} - B_{12}A_{21}) = A_{11},
	\end{equation}
	where we used \( B_{11} A_{11} + B_{12} A_{21} = \id_{X_1} \) and \( A_{11} B_{12} = 0 \).
\end{proof}

\begin{proof}[Proof of \cref{prop:normalFormLinearMap:uniformRegularCharacterization}]
	First, suppose that \( T \) is a uniformly regular family of linear maps.
	Then, by definition, we have topological decompositions \( X = \ker T_0 \oplus \coimg T_0 \), \( Y = \coker T_0 \oplus \img T_0 \)
	and, for every \( p \in P \), the map 
	\begin{equation}
		\tilde{T}_p = \pr_{\img T_0} \circ \restr{(T_p)}{\coimg T_0}
	\end{equation}
	is a topological isomorphism.
	Set \( Z^+ = \coker T_0 \) and \( Z^- = \ker T_0 \), and let \( T^+: Z^+ \to Y \) and \( T^-: X \to Z^- \) be the canonical inclusion and projection, respectively.
	Moreover, let \( S_p \defeq \tilde{T}^{-1}_p \circ \pr_{\img T_0}: Y \to X \), and define \( S^{-+}_p: Z^- \to Z^+ \) by
	\begin{equation}
		S^{-+}_p = \pr_{\coker T_0} \circ (T_p \circ \tilde{T}^{-1}_p \circ \pr_{\img T_0} - \id_Y) \circ \restr{(T_p)}{\ker T_0}\,.
	\end{equation}
	Finally, define \( S^\pm_p \) by
	\begin{subequations}\begin{align}
	 	S^+_p &= \pr_{\coker T_0} \circ (\id_Y - T_p \circ S_p): Y \to Z^+,
	 	\\
	 	S^-_p &= \restr{(\id_X - S_p \circ T_p)}{\ker T_0}: Z^- \to X.
	\end{align}\end{subequations}
	Since, by definition, the inverses \( \tilde{T}^{-1}_p \) form a continuous family \( P \times \img T_0 \to \coimg T_0 \), the families \( S, S^\pm, S^{-+} \) are continuous.
	Furthermore, a direct calculation yields
	\begin{subequations}\label{eq:normalFormLinearMap:uniformRegularCharacterization:projection}\begin{align}
		T_p \circ S_p &= \pr_{\img T_0} + \pr_{\coker T_0} \circ T_p \circ \tilde{T}^{-1}_p \circ \pr_{\img T_0},
		\\
		S_p \circ T_p &= \pr_{\coimg T_0} + \tilde{T}^{-1}_p \circ \pr_{\img T_0} \circ \restr{(T_p)}{\ker T_0}.
	\end{align}\end{subequations}
	Using these identities, it is straightforward to check that~\eqref{prop:normalFormLinearMap:uniformRegularCharacterization:external:inverse} holds for every \( p \in P \).
	Moreover, we have
	\begin{equation}
		\Gamma_p = 
		\Matrix{
			\pr_{\ker T_0} \circ \restr{(S_p)}{\coker T_0}
			&
			\pr_{\ker T_0} \circ S^-_p
			\\
			\restr{(S^+_p)}{\coker T_0}
			&
			S^{-+}_p
		}
		=
		\Matrix{
			0
			&
			\id_{\ker T_0}
			\\
			\id_{\coker T_0}
			&
			S^{-+}_p
		},
	\end{equation}
	which is clearly invertible with a continuous family of inverses given by
	\begin{equation}
		\Gamma^{\, -1}_p = \smallMatrix{
			- S^{-+}_p
			&
			\id_{\coker T_0}
			\\
			\id_{\ker T_0}
			&
			0
		}
		.
	\end{equation}

	Conversely, let \( T^\pm, S, S^\pm \) and \( S^{-+} \) satisfying the assumptions of the second statement of \cref{prop:normalFormLinearMap:uniformRegularCharacterization}.
	Since \( S^{-+}_0 = 0 \), \cref{prop:normalFormLinearMap:inverseBlockMatrix:projection} implies that \( T_0 \circ S_0 \) and \( S_0 \circ T_0 \) are idempotent with \( \img T_0 \circ S_0 = \img T_0 \) and \( \ker S_0 \circ T_0 = \ker T_0 \).
	Hence, \( \ker T_0 \) and \( \img T_0 \) are images of continuous idempotent operators, and as such they are closed and topologically complemented according to \parencite[Proposition~15.8.1]{Koethe1983}.
	As above, denote the complements by \( \coimg T_0 \) and \( \coker T_0 \), respectively.
	It remains to show that \( \tilde{T}_p = \pr_{\img T_0} \circ \restr{(T_p)}{\coimg T_0} \) is a topological isomorphism for all \( p \in P \), and that \( \tilde{T}_p^{-1} \) form a continuous family.
	For this purpose, we write all operators in block form with respect to the decompositions \( X = \coimg T_0 \oplus \ker T_0 \) and \( Y = \img T_0 \oplus \coker T_0 \) (note the different order of the summands).
	Using this convention, the identity~\eqref{prop:normalFormLinearMap:uniformRegularCharacterization:external:inverse} becomes
	\begin{equation}\begin{split}
		\Matrix{
			\tilde{T}_p
			&
			\pr_{\img T_0} \circ \restr{(T_p)}{\ker T_0}
			&
			\pr_{\img T_0} \circ T^+
			\\
			\pr_{\coker T_0} \circ \restr{(T_p)}{\coimg T_0}
			&
			\pr_{\coker T_0} \circ \restr{(T_p)}{\ker T_0}
			&
			\pr_{\coker T_0} \circ T^+
			\\
			\restr{(T^-)}{\coimg T_0}
			&
			\restr{(T^-)}{\ker T_0}
			&
			0
		}^{-1}
		=
		\\
		\qquad
		\Matrix{
			\pr_{\coimg T_0} \circ \restr{(S_p)}{\img T_0}
			&
			\pr_{\coimg T_0} \circ \restr{(S_p)}{\coker T_0}
			&
			\pr_{\coimg T_0} \circ S^-_p
			\\
			\pr_{\ker T_0} \circ \restr{(S_p)}{\img T_0}
			&
			\pr_{\ker T_0} \circ \restr{(S_p)}{\coker T_0}
			&
			\pr_{\ker T_0} \circ S^-_p
			\\
			\restr{(S^+_p)}{\img T_0}
			&
			\restr{(S^+_p)}{\coker T_0}
			&
			S^{-+}_p
		}.
	\end{split}\end{equation}
	These matrices should be read as operators from \( \img T_0 \oplus \coker T_0 \oplus Z^- \) to \( \coimg T_0 \oplus \ker T_0 \oplus Z^+ \).
	Note that the lower right two-times-two block of the right-hand side coincides with the operator \( \Gamma_p \).
	Since \( \Gamma_{p} \) is invertible by assumption, \cref{prop:normalFormLinearMap:inverseBlockMatrix:schurComplement} shows that \( \tilde{T}_p \) is invertible, too.
	Moreover, the inverses are given by 
	\begin{equation}\begin{split}
		\tilde{T}_p^{-1} 
			&= \pr_{\coimg T_0} \circ \restr{(S_p)}{\img T_0} 
			\\
			&\qquad- 
				\Vector{\pr_{\coimg T_0} \circ \restr{(S_p)}{\coker T_0} \\ \pr_{\coimg T_0} \circ S^-_p}
				\Gamma^{\, -1}_{p} 
				\Vector{\pr_{\ker T_0} \circ \restr{(S_p)}{\img T_0} \\ \restr{(S^+_p)}{\img T_0}}
	\end{split}\end{equation}
	and thus form a continuous family \( P \times \img T_0 \to \coimg T_0 \).
	Hence, \( T \) is uniformly regular.
\end{proof}
For the special case of a single operator, we obtain the following.
\begin{coro}
	\label{prop:normalFormLinearMap:regularCharacterization}
	A continuous linear map \( T: X \to Y \) between locally convex spaces is regular if and only if there exist locally convex spaces \( Z^\pm \), continuous linear maps \( T^+: Z^+ \to Y \), \( T^-: X \to Z^- \) and \( S: Y \to X \), \( S^-: Z^- \to X \), \( S^+: Y \to Z^+ \) such that
	\begin{equation}
		\Matrix{T & T^+ \\ T^- & 0}^{-1} = \Matrix{S & S^- \\ S^+ & 0}.
		\qedhere
	\end{equation}
\end{coro}
\begin{remark}
	In the setting of \cref{prop:normalFormLinearMap:regularCharacterization}, it is straightforward to verify that \( T \circ S \circ T = T \) holds.
	An operator \( S \) satisfying such a relation is called a generalized inverse of \( T \), \cf \parencite{Rao1962,Moore1920,Penrose1955}.
	In fact, one can show that regularity of an operator is equivalent to the existence of a generalized inverse.
	Since we do not need this point of view in the remainder, we refer to \parencite{DiezThesis,Harte1987} for further details.
\end{remark}
\begin{remark}[Uniform regularity in the tame Fréchet category]
	It is clear that a version of \cref{prop:normalFormLinearMap:uniformRegularCharacterization} holds in the tame Fréchet category if the word \enquote{tame} is inserted in the right places (see \cref{sec:tameFrechet} for a brief overview of the main concepts of tame Fréchet spaces).
	Let us spell out the details.
	
	Let \( X \) and \( Y \) be tame Fréchet spaces and let \( T: P \times X \to Y \) be a tame smooth family of linear maps.
	Then, \( T \) is called \emphDef{uniformly tame regular} if there exist tame decompositions \( X = \ker T_0 \oplus \coimg T_0 \) and \( Y = \coker T_0 \oplus \img T_0 \), and, for every \( p \in P \), the map \( \tilde{T}_p = \pr_{\img T_0} \circ \restr{(T_p)}{\coimg T_0}: \coimg T_0 \to \img T_0 \) is a tame isomorphism such that the inverses form a tame smooth family \( P \times \img T_0 \to \coimg T_0 \).
	Then, the equivalence of \cref{prop:normalFormLinearMap:uniformRegularCharacterization} holds with \( Z^\pm \) being tame Fréchet spaces, \( T^\pm \) being tame maps and \( S, S^\pm, S^{-+}, \Gamma^{-1} \) being tame smooth families. 
\end{remark}

\subsection{Fredholm operators}
An important class of examples of regular operators is given by Fredholm operators.
Fredholm operators are usually studied as maps between Banach spaces (or Hilbert spaces), but most results extend to the locally convex setting, \cf \parencite{Schaefer1956, Schaefer1959, Edwards1965}.
\begin{defn}
	\label{def:fredholmOperators:fredholmOperator}
	A continuous linear map \( T: X \to Y \) between locally convex spaces is called a \emphDef{Fredholm operator} if \( T \) is relatively open, the kernel of \( T \) is a finite-dimensional subspace of \( X \), and the image of \( T \) is a finite-codimensional closed subspace of \( Y \).
	The \emphDef{index} \( \ind T \) of a Fredholm operator \( T \) is defined by
	\begin{equation}
		\ind T = \dim \ker T - \dim \coker T.
		\qedhere
	\end{equation}
\end{defn}
Since finite-dimensional subspaces and finite-codimensional closed subspaces of a locally convex space are always topologically complemented according to \parencite[Propositions~15.8.2 and~20.5.5]{Koethe1983}, every Fredholm operator is regular.

For a continuous family \( T \) of linear maps with \( T_0 \) being a Fredholm operator, the invertibility of the family \( \Gamma \) in \cref{prop:normalFormLinearMap:uniformRegularCharacterization} is automatic.
\begin{coro}
	\label{prop:normalFormLinearMap:uniformRegularCharacterization:fredholm}
	Let \( T: P \times X \to Y \) be a continuous family of linear maps between locally convex spaces \( X \) and \( Y \) such that \( T_0 \) is a Fredholm operator.
	Then, \( T \) is uniformly regular if and only if there exist finite-dimensional spaces \( Z^\pm \) and continuous linear maps \( T^+: Z^+ \to Y \) and \( T^-: X \to Z^- \) such that, after possibly shrinking \( P \), 
	\begin{equation}
		\Matrix{T_p & T^+ \\ T^- & 0}^{-1} = \Matrix{S_p & S^-_p \\ S^+_p & S^{-+}_p}
	\end{equation}
	holds for all \( p \in P \), where \( S: P \times Y \to X \), \( S^-: P \times Z^- \to X \), \( S^+: P \times Y \to Z^+ \) and \( S^{-+}: P \times Z^- \to Z^+ \) are continuous families of linear maps with \( S^{-+}_0 = 0 \).
\end{coro}
\begin{proof}
	If \( T \) is uniformly regular, then the proof of \cref{prop:normalFormLinearMap:uniformRegularCharacterization} shows that one can choose \( Z^+ = \coker T_0 \) and \( Z^- = \ker T_0 \).
	Both spaces are finite-dimensional, because \( T_0 \) is a Fredholm operator.
	This establishes one direction.

	Conversely, let \( T^\pm, S, S^\pm, S^{-+} \) be given as stated above.
	By \cref{prop:normalFormLinearMap:uniformRegularCharacterization}, it suffices to show that the operator
	\begin{equation}
		\Gamma_p = \Matrix{
			\pr_{\ker T_0} \circ \restr{(S_p)}{\coker T_0}
			&
			\pr_{\ker T_0} \circ S^-_p
			\\
			\restr{(S^+_p)}{\coker T_0}
			&
			S^{-+}_p
		}: \coker T_0 \oplus Z^- \to \ker T_0 \oplus Z^+
	\end{equation}
	is invertible and that the inverses form a continuous family.
	Since \( S^{-+}_0 = 0 \), \cref{prop:normalFormLinearMap:inverseBlockMatrix:projection} implies \( S^-_0 \circ T^- = \pr_{\ker T_0} \) and \( T^+ \circ S^+_0 = \pr_{\coker T_0} \).
	A straightforward calculation using these identities shows that we have
	\begin{equation}
		\Gamma_0 = \Matrix{
			0
			&
			\pr_{\ker T_0} \circ S^-_0
			\\
			\restr{(S^+_0)}{\coker T_0}
			&
			0
		},
	\end{equation}
	and
	\begin{equation}
		\Gamma_0^{\, -1} = \Matrix{
			0
			&
			\pr_{\coker T_0} \circ T^+
			\\
			\restr{(T^-)}{\ker T_0}
			&
			0
		}.
	\end{equation}
	Since, for every \( p \in P \), \( \Gamma_p \) is an operator between finite-dimensional spaces and \( \Gamma_0 \) is invertible, we can shrink \( P \) in such a way that \( \Gamma_p \) is invertible for all \( p \in P \) and that the inverses form a continuous family.
	Thus, \cref{prop:normalFormLinearMap:uniformRegularCharacterization} implies that \( T \) is uniformly regular.
\end{proof}
In the Banach setting, the set of Fredholm operators is open in the space of bounded linear operators and the index does not change under a continuous deformation~\parencite[Corollary~19.1.6]{Hoermander2007}.
This statement relies on the openness of the set of invertible operators and thus does not carry over to the locally convex setting.
The following proposition shows that the notion of uniform regularity is an adequate substitute.
\begin{prop}
	Let \( T: P \times X \to Y \) be a uniformly regular family of linear maps between locally convex spaces \( X \) and \( Y \).
	If \( T_0 \) is  Fredholm, then \( T_p \) is Fredholm and \( \ind T_p = \ind T_0 \) for all \( p \in P \).
\end{prop}
\begin{proof}
	By \cref{prop:normalFormLinearMap:uniformRegularCharacterization:fredholm}, there exist finite-dimensional spaces \( Z^\pm \) and continuous linear maps \( T^+: Z^+ \to Y \) and \( T^-: X \to Z^- \) such that 
	\begin{equation}
		\Matrix{T_p & T^+ \\ T^- & 0}
	\end{equation}
	is invertible for all \( p \in P \).
	This is only possible if \( \ker T_p \) and \( \coker T_p \) are finite-dimensional so that \( T_p \) has to be Fredholm.
	Moreover, using the invariance of the index under finite-rank perturbations, we have
	\begin{equation}
		0 
			= \ind \Matrix{ T_p & T^+ \\ T^- & 0} 
			= \ind T_p + \dim Z^+ - \dim Z^-
			= \ind T_p - \ind T_0,
	\end{equation}
	which establishes the formula for the index.
\end{proof}

As we discuss now, families of elliptic operators constitute an important class of examples of uniformly regular Fredholm operators.
Let \( E \to M \) and \( F \to M \) be finite-dimensional vector bundles over a compact manifold \( M \) without boundary.
Endow the spaces \( \SectionSpaceAbb{E} \) and \( \SectionSpaceAbb{F} \) of smooth sections of \( E \) and \( F \), respectively, with the compact-open \( \sFunctionSpace \)-topology.
With respect to this topology, these section spaces are tame Fréchet spaces, see \parencite[Theorem~II.2.3.1]{Hamilton1982}.
A continuous linear map \( L: \SectionSpaceAbb{E} \to \SectionSpaceAbb{F} \) is a partial differential operator of degree \( r \) if and only if there exists a vertical vector bundle morphism \( f: \JetBundle^r E \to F \) such that \( L \) factors through the jet bundle \( \JetBundle^r E \) as follows:
\begin{equationcd}[label=eq:normalFormLinearMap:familiesEllipticOps:factorizationDiffOp]
	L\tikzcolon \SectionSpaceAbb{E}
	\to[r, "\jet^r"]
	&\sSectionSpace(\JetBundle^r E)
	\to[r, "f_*"]
	&\SectionSpaceAbb{F},
\end{equationcd}
where \( \jet^r \) denotes the \( r \)-th jet prolongation and \(  f_* \) is the push-forward by \( f \).
We refer to \( f \) as the coefficients of \( L \) and will sometimes write \( L_f \) instead of \( L \) to emphasize this relation.
Recall that the principal symbol \( \principalSymbol_f \) of \( L_f \) is a homogeneous polynomial of degree \( r \) on \( \CotBundle M \) with values in the bundle \( \LinMapBundle(E, F) \) of fiberwise linear maps \( E \to F \).
A differential operator \( L_f \) with coefficients \( f \) is called \emphDef{elliptic} if its symbol is invertible; that is, for each nonzero \( p \in \CotBundle M \), the bundle map \( \principalSymbol_f (p, \dotsc, p) \in \LinMapBundle(E, F) \) is invertible.

It is a standard result in elliptic theory that every elliptic differential operator over a compact manifold is a Fredholm operator between appropriate Sobolev spaces \parencite[Theorem~19.2.1]{Hoermander2007}.
The same holds true in the tame Fréchet category.
In fact, more is true: elliptic operators are regularly parametrized by their coefficients. 
\begin{thm}
	\label{prop:normalFormLinearMap:familiesEllipticOps:areUniformlyTameRegular}
	Let \( E \to M \) and \( F \to M \) be finite-dimensional vector bundles over a compact manifold \( M \) without boundary, and denote the space of smooth sections of \( E \) and \( F \) by \( \SectionSpaceAbb{E} \) and \( \SectionSpaceAbb{F} \), respectively.
	The parametrization of a partial differential operator by their coefficients yields a tame smooth family
	\begin{equation}
		\label{eq:normalFormLinearMap:familiesEllipticOps:diffOpFamily}
		L: {\sSectionSpace\bigl(\LinMapBundle(\JetBundle^r E, F)\bigl)} \times \SectionSpaceAbb{E} \to \SectionSpaceAbb{F},
		\quad
		(f, \phi) \mapsto L_f (\phi)
	\end{equation}
	of linear operators which is uniformly tame regular in a neighborhood of every \( f_0 \in \sSectionSpace\bigl(\LinMapBundle(\JetBundle^r E, F)\bigl) \) for which \( L_{f_0} \) is an elliptic differential operator.
\end{thm}
\begin{proof}
	Let \( f_0 \in \sSectionSpace\bigl(\LinMapBundle(\JetBundle^r E, F)\bigl) \) be such that \( L_{f_0} \) is an elliptic differential operator.
	By \parencite[Theorem~II.3.3.3]{Hamilton1982}, there exist an open neighborhood \( \SectionSpaceAbb{U} \) of \( f_0 \) in \( \sSectionSpace\bigl(\LinMapBundle(\JetBundle^r E, F)\bigl) \), finite-dimensional vector spaces \( Z^\pm \) and continuous linear maps \( L^+: Z^+ \to Y \) and \( L^-: X \to Z^- \) such that 
	\begin{equation}
		\Matrix{L_f & L^+ \\ L^- & 0}: \SectionSpaceAbb{E} \times Z^+ \to \SectionSpaceAbb{F} \times Z^-  
	\end{equation}
	is invertible for all \( f \in \SectionSpaceAbb{U} \).
	Moreover, the inverses form a tame smooth family \( \SectionSpaceAbb{U} \times \SectionSpaceAbb{F} \times Z^- \to \SectionSpaceAbb{E} \times Z^+ \) of linear operators.
	Hence, by \cref{prop:normalFormLinearMap:uniformRegularCharacterization:fredholm}, \( \restr{L}{\SectionSpaceAbb{U} \times \SectionSpaceAbb{E} } \) is uniformly tame regular at \( f_0 \).
\end{proof}

\subsection{Elliptic complexes}
\label{sec:normalFormLinearMap:ellipticComplexes}
In this section, the notion of uniform regularity is extended to linear chain complexes.
The main application we have in mind is elliptic complexes.

Let \( P \) be a neighborhood of \( 0 \) in some locally convex space, let \( X_i \) be a sequence of locally convex spaces and let \( T_i: P \times X_i \to X_{i+1} \) be a sequence of continuous families of linear maps such that \( (X_i, T_{i, 0}) \) is a complex, \ie, \( T_{i+1, 0} \circ T_{i, 0} = 0 \) for all \( i \in \Z \).
We say that \( (P, X_i, T_{i}) \) is a \emphDef{continuous family of chains}.
Simple examples (\cf \cref{prop:normalFormLinearMap:familiesEllipticOps:covariantDerivUniformRegular} below) show that a deformation of a chain complex is in general not a complex; this is why we require \( T_{i, p} \) to form a complex only at \( p = 0 \).
The following notion is a natural generalization of uniform regularity to chains.
\begin{defn}
	\label{defn:normalFormLinearMap:ellipticComplexes:uniformRegular}
	A continuous family of chains \( (P, X_i, T_{i}) \) is called \emphDef{uniformly regular} (at \( 0 \)) if the following holds for every \( i \in \Z \):
	\begin{enumerate}
		\item
			The image of \( T_{i-1, 0} \) is closed in \( X_i \), and there exist closed subspaces \( H_i \) and \( \coimg T_{i, 0} \) of \( X_i \) such that
			\begin{equation}
				\label{eq:normalFormLinearMap:chainComplex:decomposition}
				X_i = \img T_{i-1, 0} \oplus \coimg T_{i, 0} \oplus H_i
			\end{equation}
			is a topological decomposition and \( H_i \subseteq \ker T_{i,0} \).
		\item
			For every \( p \in P \), the map 
			\begin{equation}
				\tilde{T}_{i, p} = \pr_{\img T_{i, 0}} \circ \restr{(T_{i, p})}{\coimg T_{i, 0}}: \coimg T_{i, 0} \to \img T_{i, 0}
			\end{equation}
			is a topological isomorphism such that the inverses form a continuous family \( P \times \img T_{i, 0} \to \coimg T_{i,0} \).
	\end{enumerate}
	If additionally, for every \( i \in \Z \), \( X_i \) is a tame Fréchet space, \( T_i \) is a tame smooth family, the decomposition~\eqref{eq:normalFormLinearMap:chainComplex:decomposition} of \( X_i \) is tame and \( \tilde{T}_{i, p} \) are a tame isomorphisms such that the inverses form a tame smooth family, then \( (P, X_i, T_i) \) is called \emphDef{uniformly tame regular}.
\end{defn}
By definition, for an uniformly regular family \( (P, X_i, T_{i}) \) of chains, we have 
\begin{equation}
	\ker T_{i, 0} = \img T_{i-1,0} \oplus H_i \, ,
\end{equation}
which justifies the notion \( \coimg T_{i,0} \) for the subspace in the decomposition~\eqref{eq:normalFormLinearMap:chainComplex:decomposition}.
The subspaces \( H_i \) are identified with the \emphDef{homology groups} for the complex at \( p = 0 \), that is,
\begin{equation}
	H_i \isomorph \ker T_{i, 0} \slash \img T_{i-1, 0}.
\end{equation}

For the applications we have in mind, the following characterization of uniform regularity of chains turns out to be more convenient.
It entails that, roughly speaking, a family of chains \( (P, X_i, T_{i}) \) is uniformly regular if each family \( T_i \) of linear maps is uniformly regular after factoring-out the image of the direct predecessor \( T_{i-1, 0} \).
\begin{prop}
	\label{prop:normalFormLinearMap:ellipticComplexes:uniformRegularEquiv}
 	A continuous family of chains \( (P, X_i, T_{i}) \) is uniformly regular if and only if, for every \( i \in \Z \), the subspace \( \img T_{i-1, 0} \) of \( X_i \) is closed and topologically complemented, say \( X_i = \img T_{i-1, 0} \oplus \coker T_{i-1, 0} \), and the continuous family \( p \mapsto \restr{(T_{i, p})}{\coker T_{i-1, 0}} \) of linear maps is uniformly regular.
\end{prop}
\begin{proof}
	The claim is a simple consequence of the observation that the image of \( \restr{(T_{i, 0})}{\coker T_{i-1, 0}} \) coincides with the image of \( T_{i, 0} \) and that 
	\begin{equation}
		H_i \isomorph \ker{\restr{(T_{i, 0})}{\coker T_{i-1, 0}}}
	\end{equation}
	holds, because \( T_{i,0} \) is a complex.
\end{proof}
Let us now turn to deformations of elliptic complexes.
Let \( E_0, E_1, \dotsc, E_N \) be a sequence of finite-dimensional vector bundles over a compact manifold \( M \), and let \( \SectionSpaceAbb{E}_i \) be the tame Fréchet space of smooth sections of \( E_i \). 
Moreover, let \( P \) be an open neighborhood of \( 0 \) in some tame Fréchet space, and let \( L_{i}: P \times \SectionSpaceAbb{E}_i \to \SectionSpaceAbb{E}_{i+1} \) be a sequence of differential operators parametrized by points of \( P \).
We assume that, for every \( i \in \Z \), the parametrization factors through the space of coefficients as follows:
\begin{equationcd}
	P \times \SectionSpaceAbb{E}_i
		\to[rr, "{\hat{L}_i} \, \times \, {\id_{\SectionSpaceAbb{E}_i}}"]
		&& {\sSectionSpace\bigl(\LinMapBundle(\JetBundle^{r_i} E_i, E_{i+1})\bigl)} \times {\SectionSpaceAbb{E}_i}
		\to[rr]
		&& \SectionSpaceAbb{E}_{i+1},
\end{equationcd}
where \( \hat{L}_i: P \mapsto \sSectionSpace\bigl(\LinMapBundle(\JetBundle^{r_i} E_i, E_{i+1})\bigl) \) is a tame smooth map and the second map is the parametrization of differential operators by their coefficients as defined in~\eqref{eq:normalFormLinearMap:familiesEllipticOps:diffOpFamily}.
For simplicity, let us assume that the degree \( r_i \) of the differential operator \( L_{i, p}: \SectionSpaceAbb{E}_i \to \SectionSpaceAbb{E}_{i+1} \) is the same for all \( p \in P \) and \( i \in \Z \).  
We will refer to this setting by saying that \( (P, \SectionSpaceAbb{E}_i, L_{i}) \) is a \emphDef{tame family of chains of differential operators}.
A chain complex \( L_i: \SectionSpaceAbb{E}_i \to \SectionSpaceAbb{E}_{i+1} \) of differential operators is called \emphDef{elliptic} if the sequence of principal symbols
\begin{equationcd}
	\dotsb \to[r] & \CotBundleProj^* E_{i} \to[r, "\principalSymbol(L_i)"] & \CotBundleProj^* E_{i+1} \to[r] & \dotsb
\end{equationcd}
is exact outside of the zero section of the cotangent bundle \( \CotBundleProj: \CotBundle M \to M \).

As a generalization of \cref{prop:normalFormLinearMap:familiesEllipticOps:areUniformlyTameRegular}, we have the following result concerning deformations of elliptic complexes.
\begin{thm}
	\label{prop:normalFormLinearMap:familiesEllipticOps:ellipticComplexUniformlyTameRegular}
	Let \( E_0, \dotsc, E_N \) be a sequence of finite-dimensional vector bundles over a compact manifold \( M \), and denote the space of smooth sections of \( E_i \) by \( \SectionSpaceAbb{E}_i \).
	Let \( (P, \SectionSpaceAbb{E}_i, L_{i}) \) be a tame family of chains of differential operators.
	If \( (\SectionSpaceAbb{E}_i, L_{i, 0}) \) is an elliptic complex, then \( (P, \SectionSpaceAbb{E}_i, L_{i}) \) is uniformly tame regular (after possibly shrinking \( P \)).
\end{thm}
\begin{proof}
	The proof is inspired by the proof of \parencite[Proposition~6.1]{AtiyahBott1967}, where a parametrix of an elliptic complex is constructed by using the parametrix of an elliptic operator.
	Similarly, we will reduce the question of the uniform tame regularity of the chain to the uniform regularity of a deformation of differential operators, for which we can employ \cref{prop:normalFormLinearMap:familiesEllipticOps:areUniformlyTameRegular}.

	For this purpose, fix a Riemannian metric on \( M \) and a fiber Riemannian metric on every vector bundle \( E_i \).
	These data define a natural \( \LTwoFunctionSpace \)-inner product on \( \SectionSpaceAbb{E}_i \).
	By partial integration, we see that the adjoints \( L^*_{i, p}: \SectionSpaceAbb{E}_{i+1} \to \SectionSpaceAbb{E}_i \) of \( L_{i,p} \) with respect to these inner products yield a tame family of chains of differential operators.
	For every \( i \in \Z \), define the tame family \( \laplace_{i}: P \times \SectionSpaceAbb{E}_i \to \SectionSpaceAbb{E}_i \) by
	\begin{equation}
		\label{eq:normalFormLinearMap:familiesEllipticOps:ellipticComplexUniformlyTameRegular:laplace}
		\laplace_{i, p} = L_{i,0}^* \circ L_{i,p} + L_{i-1,p} \circ L_{i-1,0}^*.
	\end{equation}
	Clearly, \( \laplace_{i} \) is a family of differential operators of order \( 2 r \).
	Moreover, \( \laplace_{i, 0} \) is an elliptic operator, because \((\SectionSpaceAbb{E}_i, L_{i, 0}) \) is an elliptic complex by assumption.
	Thus, \cref{prop:normalFormLinearMap:familiesEllipticOps:areUniformlyTameRegular} implies that the family \( \laplace_{i} \) is uniformly tame regular.
	In particular, \( \laplace_{i, 0} \) is regular and self-adjoint so that we get the following topological decomposition
	\begin{equation}
		\label{eq:normalFormLinearMap:familiesEllipticOps:ellipticComplexUniformlyTameRegular:decompositionLaplacian}
		\SectionSpaceAbb{E}_i = \ker \laplace_{i, 0} \oplus \img \laplace_{i, 0} \, .
	\end{equation}
	Moreover, \( \tilde{\laplace}_{i, p} = \pr_{\img \laplace_{i, 0}} \circ \restr{(\laplace_{i, p})}{\img \laplace_{i, 0}} \) is a tame automorphism of \( \img \laplace_{i, 0} \) for every \( p \in P \) (after possibly shrinking \( P \)) in such a way that the inverses form a tame smooth family.
	The decomposition~\eqref{eq:normalFormLinearMap:familiesEllipticOps:ellipticComplexUniformlyTameRegular:decompositionLaplacian} implies that the images of \( L_{i-1,0} \) and \( L^*_{i,0} \) are closed and that they fit into the topological decomposition
	\begin{equation}
		\label{eq:normalFormLinearMap:familiesEllipticOps:ellipticComplexUniformlyTameRegular:decomposition}
		\SectionSpaceAbb{E}_i = \img L_{i-1,0} \oplus \img L^*_{i,0} \oplus H_i \, ,
	\end{equation}
	where \( H_i \equiv \ker \laplace_{i, 0} = \ker L^{}_{i,0} \intersect \ker L^*_{i-1, 0} \).
	Finally, a direct calculation shows that, for every \( i \in \Z \), the tame smooth family \( G_i \) defined by
	\begin{equation}
		G_{i,p} = L^*_{i, 0} \circ \restr{(\tilde{\laplace}_{i+1,p}^{-1})}{\img L_{i, 0}}: \img L_{i, 0} \to \img L^*_{i, 0}
	\end{equation}
	is an inverse of the family
	\begin{equation}
		\tilde{L}_{i, p} = \pr_{\img L_{i, 0}} \circ \restr{(L_{i, p})}{\img L^*_{i, 0}}: \img L^*_{i, 0} \to \img L_{i, 0}.
	\end{equation}
	This shows that \( (P, \SectionSpaceAbb{E}_i, L_{i}) \) is uniformly tame regular.
\end{proof}

\begin{example}
	\label{prop:normalFormLinearMap:familiesEllipticOps:covariantDerivUniformRegular}
	Let \( P \to M \) be a finite-dimensional principal \( G \)-bundle over a compact manifold \( M \), and let \( E \) be an associated vector bundle.
	The space \( \ConnSpace(P) \) of connections on \( P \) is an affine tame Fréchet space.
	Every connection \( A \in \ConnSpace(P) \) yields via the covariant exterior differential on \( E \)-valued forms a chain
	\begin{equationcd}
		\dotsb \to[r] & \DiffFormSpace^i(M, E) \to[r, "\dif_A"] & \DiffFormSpace^{i+1}(M, E) \to[r] & \dotsb \,.
	\end{equationcd}
	As this chain is an elliptic complex if the connection is flat, \cref{prop:normalFormLinearMap:familiesEllipticOps:ellipticComplexUniformlyTameRegular} entails that the family of chains \( \bigl(\ConnSpace(P), \DiffFormSpace^k(M, E), \dif_A \bigr) \) is uniformly tame regular in a neighborhood of every flat connection \( A_0 \in \ConnSpace(P) \).
	Moreover, by \cref{prop:normalFormLinearMap:ellipticComplexes:uniformRegularEquiv}, the family
	\begin{equation}
		\ConnSpace(P) \times \DiffFormSpace^0(M, E) \to \DiffFormSpace^{1}(M, E),
		\qquad
		(A, \alpha) \mapsto \dif_A \alpha
	\end{equation}
	is uniformly regular at every \( A_0 \in \ConnSpace(P) \) (we do not need flatness of \( A_0 \) for this in \( 0 \)-degree).
	The operator defined in~\eqref{eq:normalFormLinearMap:familiesEllipticOps:ellipticComplexUniformlyTameRegular:laplace} takes here the following form
	\begin{equation}
		\laplace_{A_0 A}^{} = \dif^*_{A_0} \dif^{}_A  + \dif^*_A \dif^{}_{A_0}: \DiffFormSpace^{k}(M, E) \to \DiffFormSpace^{k}(M, E)
	\end{equation}
	and is a natural extension of the Faddeev--Popov operator to forms of higher degree, \cf \parencite[eq.~(8.4.8)]{RudolphSchmidt2014}.
	A similar operator played a central role in \parencite[p.~405]{DiezHuebschmann2017} for the study of the curvature map \( F: \ConnSpace(P) \to \DiffFormSpace^2(M, \AdBundle P) \) near a flat connection.
\end{example}

\section{Normal form of a nonlinear map}
\label{sec:normalFormMap}

In this section, we study the local behavior of a smooth map \( f: M \to N \) between (infinite-dimensional) manifolds.
We introduce the concept of a normal form and find suitable conditions that ensure that \( f \) can be brought into such a normal form. 

In the linear setting, we have seen that every regular operator factorizes through a linear isomorphism.
Similarly, every nonlinear map can be represented locally by a linear isomorphism up to some higher order error term.
In fact, for a given smooth map \( f: M \to N \) and \( m \in M \), consider its local representative \( \tilde{f} \defeq \rho \circ f \circ \kappa^{-1}: X \supseteq U \to Y \) with respect to charts\footnotemark{} \( \kappa: M \supseteq U' \to U \subseteq X \) at \( m \) and \( \rho: N \supseteq V' \to V \subseteq Y \) at \( f(m) \) modeling \( M \) on \( X \isomorph \TBundle_m M \) and \( N \) on \( Y \isomorph \TBundle_{f(m)} N \), respectively.
\footnotetext{Throughout this work, we follow the convention that a chart \( \kappa: M \supseteq U' \to U \subseteq X \) at a point \( m \in M \) satisfies \( \kappa(m) = 0 \). Moreover, \( U' \) and \( U \) are understood to be open neighborhoods of \( m \) in \(  M \) and of \( 0 \) in \(  X \), respectively.}
Suppose \( \tangent_0 \tilde{f}: X \to Y \) is a regular operator with decompositions \( X = \ker \tangent_0 \tilde{f} \oplus \coimg \tangent_0 \tilde{f} \), \( Y = \coker \tangent_0 \tilde{f} \oplus \img \tangent_0 \tilde{f} \) and core \( \hat{f}: \coimg \tangent_0 \tilde{f} \to \img \tangent_0 \tilde{f} \).
Then defining \( f_\singularPart \defeq \tilde{f} - \hat{f}: X \supseteq U \to Y \) we get the following local representation of \( f \):
\begin{equation}
	\rho \circ f \circ \kappa^{-1} = \hat{f} + f_\singularPart.
\end{equation}
Thus, the singular part \( f_\singularPart \) obstructs \( f \) from being locally represented by the isomorphism \( \hat{f} \).
The concrete form of \( f_\singularPart \) depends, of course, on the chosen charts.
The aim is to construct charts such that the singular part satisfies additional properties, which are formalized in the following definition.
\begin{defn}
	\label{defn:normalFormMap:normalFormAbstract}
	An \emphDef{abstract normal form} consists of a tuple \( (X, Y, \hat{f}, f_\singularPart) \), where
	\begin{enumerate}
		\item
			\( X \) and \( Y \) are locally convex spaces with topological decompositions\footnotemark{} \( X = \ker \oplus \coimg \) and \( Y = \coker \oplus \img \),
			\footnotetext{In these decompositions \( \ker \), \( \coimg \), \etc denote abstract spaces. Below, we will identify them with the kernel, coimage, \etc of the tangent map \( \tangent_m f \), respectively.}
		\item
			\( \hat{f}: \coimg \to \img \) is a linear topological isomorphism,
		\item
			\( f_\singularPart: X \supseteq U \to \coker \) is a smooth map defined on an open neighborhood \( U \) of \( 0 \) in \( X \) such that \( f_\singularPart(0, x_2) = 0 \) holds for all \( x_2 \in U \intersect \coimg \) and such that the derivative \( \tangent_{(0, 0)} f_\singularPart: X \to \coker \) of \( f_\singularPart \) at \( (0, 0) \) vanishes.
		\end{enumerate}
		Given an abstract normal form \( (X, Y, \hat{f}, f_\singularPart) \), set \( \normalForm{f} = \hat{f} + f_\singularPart: X \supseteq U \to Y \).

		A normal form \( (X, Y, \hat{f}, f_\singularPart) \) is called \emphDef{tame} if \( X, Y \) are tame Fréchet spaces with tame decompositions \( X = \ker \oplus \coimg \), \( Y = \coker \oplus \img \), \( \hat{f} \) is a tame isomorphism and \( f_\singularPart \) is a tame smooth map.
\end{defn}
The \( 0 \)-level set of \( \normalForm{f} \) is given by
\begin{equation}
	\label{eq:normalFormMap:zeroSetNF}
	\normalForm{f}^{-1}(0) = \set[big]{(x_1, 0) \in U \given f_\singularPart(x_1, 0) = 0}.
\end{equation}
Since \( \tangent_{(0, 0)} f_\singularPart = 0 \), the level set \( \normalForm{f}^{-1}(0) \) is in general not a smooth manifold.
Its singular structure is completely determined by \( f_\singularPart \).
For this reason, we refer to \( f_\singularPart \) as the \emphDef{singular part} of \( \normalForm{f} \).

\begin{defn}
	\label{defn:normalFormMap:bringIntoNormalForm}
	We say that a smooth map \( f: M \to N \) between manifolds can \emphDef{be brought into the normal form} \( (X, Y, \hat{f}, f_\singularPart) \) at the point \( m \in M \) if there exist charts \( \kappa: M \supseteq U' \to U \subseteq X \) at \( m \) and \( \rho: N \supseteq V' \to V \subseteq Y \) at \( f(m) \) such that \( f(U') \subseteq V' \), \( f_\singularPart \) is defined on \( U \) and 
	\begin{equation}
		\label{eq:normalFormMap:bringIntoNormalForm:commutative}
		\rho \circ \restr{f}{U'} \circ \kappa^{-1} = \normalForm{f}
	\end{equation}
	holds.
	For short, we say that \emphDef{\( f \) is locally equivalent to \( \normalForm{f} \)}.
\end{defn}
Assume that the smooth map \( f: M \to N \) can be brought into the normal form \( (X, Y, \hat{f}, f_\singularPart) \) at the point \( m \in M \) using diffeomorphisms \( \kappa: U' \to U \) and \( \rho: V' \to V \).
The isomorphisms \( \tangent_m \kappa: \TBundle_m M \to X \) and \( \tangent_{f(m)} \rho: \TBundle_{f(m)} N \to Y \) identify the abstract spaces \( X \) and \( Y \) with the tangent spaces of \( M \) and \( N \), respectively.
Under these identifications, the spaces \( \ker \) and \( \img \) in the decomposition of \( X \) and \( Y \) coincide with the kernel and the image of \( \tangent_m f \).

In certain cases, a normal form amounts to a linearization of the map under consideration.
We say that a smooth map \( f: M \to N \) is a \emphDef{submersion} at \( m \in M \) if it is equivalent to a linear projection in a neighborhood of \( m \).
Similarly, \( f \) is called an \emphDef{immersion} at \( m \) if it is equivalent to a linear injection in a neighborhood of \( m \).
More generally, \( f \) is a \emphDef{subimmersion} at \( m \) if it is equivalent to a linear map in a neighborhood of \( m \).
The following proposition connects these definitions to the perhaps more classical characterizations known from finite-dimensional geometry.
\begin{prop}
	\label{prop:normalFormMap:submersionImmersionConstantRank}
	Let \( f: M \to N \) be a smooth map.
	Assume that \( f \) can be brought into a normal form in a neighborhood \( U' \) of \( m \in M \).
	Then, the following holds:
	\begin{thmenumerate}
		\item (Submersion)
			\label{prop:normalFormMap:submersion}
			\( f \) is a submersion at \( m \) if and only if \( \tangent_m f \) is surjective.
		\item (Immersion)
			\label{prop:normalFormMap:immersion}
			\( f \) is an immersion at \( m \) if and only if \( \tangent_m f \) is injective.
		\item (Constant rank)
			\label{prop:normalFormMap:constantRank}
			 \( f \) is a subimmersion at \( m \) if \( \tangent_p f \) is a finite-rank operator\sideNoteMark{} satisfying \( \rank \tangent_p f = \rank \tangent_m f \) for all \( p \in U' \).
			\qedhere
			\sideNoteText{An operator \( T: X \to Y \) between locally convex spaces \( X \) and \( Y \) is said to be of \emphDef{finite-rank} if its image is finite-dimensional.
			In this case, the dimension of \( \img T \) is called the rank of \( T \) and it is denoted by \( \rank T \).}
	\end{thmenumerate}
\end{prop}
\Cref{prop:normalFormMap:constantRank} generalizes to maps whose derivatives have a constant but not necessarily finite-dimensional image, \cf \parencite[Theorem~2.5.15]{MarsdenRatiuEtAl2002} for a Banach version.
\begin{proof}
	As the claim is of local nature, it suffices to consider the case where \( f = \normalForm{f} = \hat{f} + f_\singularPart \) for an abstract normal form \( (X, Y, \hat{f}, f_\singularPart) \), and \( m = 0 \).
	If \( \tangent_{0} \normalForm{f} = \hat{f} \circ \pr_{\coimg} \) is surjective, then \( \coker \) is trivial and hence \( f_\singularPart = 0 \).
	Similarly, if \( \tangent_{0} \normalForm{f} \) is injective, then \( \coimg = X \) and thus \( f_\singularPart = 0 \), because \( f_\singularPart(0, x_2) = 0 \) for all \( x_2 \in U \intersect \coimg \) by assumption.
	Moreover, we have
	\begin{equation}
		\tangent_x \normalForm{f} \, (v_1, v_2) = \bigl(\tangent_x f_\singularPart (v_1, v_2), \hat{f} (v_2)\bigr) \in \coker \oplus \img = Y
	\end{equation}
	for all \( x \in U \), \( v_1 \in \ker \) and \( v_2 \in \coimg \).
	The constant rank condition translates to \( \rank \tangent_x \normalForm{f} = \rank \tangent_0 \normalForm{f} = \dim \img \), and thus implies that \( \tangent_{(x_1, x_2)} f_\singularPart (v_1, 0) = 0 \) for all \( (x_1, x_2) \in U \) and \( v_1 \in \ker \).
	Hence, \( f_\singularPart(x_1, x_2) \) does not depend on \( x_1 \) and so \( f_\singularPart(x_1, x_2) = f_\singularPart(0, x_2) = 0 \).

	The converse directions are clear.
\end{proof}
Using the local structure of submersions and immersions, it is straightforward to verify that the usual statements about the submanifold structure of the naturally induced subsets translate to the locally convex setting.
Indeed, according to \parencite[Theorem~C]{Gloeckner2015}, the level set \( f^{-1}(\mu) \) is a submanifold of \( M \) if \( f \) is a submersion at every \( m \in f^{-1}(\mu) \).
If instead \( f: M \to N \) is an immersion and a topological embedding, then \( f(M) \) is a submanifold of \( N \), see \parencite[Lemma~I.13]{Gloeckner2015}.

The upshot of \cref{prop:normalFormMap:submersionImmersionConstantRank} is that one obtains the submersion, regular value, immersion and constant rank theorem practically for free as soon as one knows that the map under study can be brought into a normal form.
\emph{A normal form theorem thus unifies these fundamental theorems under one umbrella.}
In other words, we avoid the renewed construction of special normal forms for each of these results separately, which would be the standard approach in textbooks about (finite-dimensional) differential geometry \parencite{MarsdenRatiuEtAl2002,Lang1999}.
The aim of the remainder of the section is to find suitable conditions on the derivative of a map \( f \) which ensure that \( f \) can be brought into a normal form.
\begin{remark}[Finite-dimensional reduction and Sard--Smale theorem]
	As we have seen in~\eqref{eq:normalFormMap:zeroSetNF}, the singularities of a level set \( f^{-1}(\mu) \) are, in terms of a local normal form, completely encoded in the map 
	\begin{equation}
		f_\singularPart(\cdot, 0): \ker \to \coker .
	\end{equation}
	If \( \tangent_m f \) is a Fredholm operator, then \( \ker \) and \( \coker \) are finite-dimensional spaces.
	Thus, in this case, the study of the singular structure of \( f^{-1}(\mu) \) is reduced to a question in finite dimensions.
	Exploiting this reduction to finite dimensions, we may mimic the usual proof of the Sard--Smale theorem \parencite{Smale1965} to obtain a generalization of this theorem to Fredholm maps between \emph{locally convex} manifolds.
	We leave the details to the reader.
\end{remark}

\subsection{Banach version}
The main idea of the proof of our general normal form theorem is to deform some initially chosen charts in such a way that the singular part of the local representative satisfies the conditions of \cref{defn:normalFormMap:normalFormAbstract}.
This deformation will be accomplished using the Inverse Function Theorem.
To somewhat reduce the functional analytic complexity, we first restrict attention to the Banach setting.
\begin{thm}[Normal form --- Banach]
	\label{prop:normalFormMap:banach}
	\label{prop:normalFormMap:finiteDim}
	A smooth map \( f: M \to N \) between Banach manifolds can be brought into a normal form at a point \( m \in M \) if and only if \( \tangent_m f: \TBundle_m M \to \TBundle_{f(m)} N \) is a regular operator.
	In particular, every smooth map between finite-dimensional manifolds can be brought into a normal form around every point.
\end{thm}
\begin{proof}
	If \( f \) can be brought into the normal form \( (X, Y, \hat{f}, f_\singularPart) \), then \( \tangent_m f \) coincides up to conjugation by topological isomorphisms with the differential at \( 0 \) of \( \normalForm{f} = \hat{f} + f_\singularPart: X \supseteq U \to Y \).
	Now \( \tangent_0 \normalForm{f} = \hat{f} \circ \pr_{\coimg} \) shows that \( \tangent_m f \) is regular.

	In the other direction, since the claim is of local nature, we can use charts \( \tilde{\kappa}: M \supseteq U' \to U \subseteq X \) at \( m \) and \( \tilde{\rho}: N \supseteq V' \to V \subseteq Y \) at \( f(m) \) to replace \( f \) by its local representative \( \tilde{f} \defeq \tilde{\rho} \circ f \circ \tilde{\kappa}^{-1}: X \supseteq U \to Y \), and \( \tangent_m f \) by \( T \defeq \tangent_0 \tilde{f}: X \to Y \).
	Since \( T \) is regular by assumption, there exist topological decompositions \( X = \ker T \oplus \coimg T \) and \( Y = \coker T \oplus \img T \).
	Moreover, the core \( \hat{T}: \coimg T \to \img T \) of \( T \) is a topological isomorphism.

	Define the smooth map \( \psi: X \supseteq U \to X \) by 
	\begin{equation}
		\label{eq:normalFormMap:banach:deformDomain}
	 	\psi(x_1, x_2) = \bigl(x_1, \hat{T}^{-1} \circ \pr_{\img T} \circ \tilde{f}(x_1, x_2)\bigr),
	\end{equation}
	with \( x_1 \in \ker T \) and \( x_2 \in \coimg T \).
	Note that \( \psi(0) = 0 \).
	Since \( \tangent_0 \psi = \id_X \), it follows from the Inverse Function \cref{prop:inverseFunctionTheorem:banach} that we can shrink \( U \) in such a way that \( \psi(U) \) is an open neighborhood of \( 0 \) in \( X \) and \( \psi: U \to \psi(U) \) is a diffeomorphism.
	By possibly shrinking \( V \), we may assume that \( V \intersect \img T \subseteq \hat{T} \circ \psi(U \intersect \coimg T) \).
	Define the smooth map \( \phi: Y \supseteq V \to Y \) by
	\begin{equation}
		\label{eq:normalFormMap:banach:deformTarget}
		\phi(y_1, y_2) = \bigl(y_1 + \pr_{\coker T} \circ \tilde{f} \circ \psi^{-1}(0, \hat{T}^{-1} y_2), y_2\bigl)
	\end{equation}
	with \( y_1 \in \coker T \) and \( y_2 \in \img T \).
	A direct calculation shows \( \phi(0) = 0 \) and \( \tangent_0 \phi = \id_Y \).
	Thus, the Inverse Function \cref{prop:inverseFunctionTheorem:banach} implies that we can shrink \( V \) so that \( \phi(V) \) is an open neighborhood of \( 0 \) in \( Y \) and \( \phi: V \to \phi(V) \) is a diffeomorphism.
	By possibly shrinking \( U \), we may assume \( \tilde{f}(U) \subseteq V \) and \( \tilde{f}(U) \subseteq \phi(V) \).

	Set \( \hat{f} \defeq \hat{T}: \coimg T \to \img T \) and define the smooth map \( f_\singularPart: X \supseteq \psi(U) \to \coker T \) by
	\begin{equation}
		f_\singularPart \circ \psi(x_1, x_2) = \pr_{\coker T} \left(\tilde{f}(x_1, x_2) - \tilde{f} \circ \psi^{-1}\bigl(0, \hat{T}^{-1} \circ \pr_{\img T} \circ \tilde{f}(x_1, x_2)\bigr)\right).
	\end{equation}
	A straightforward computation shows that the following diagram commutes:
	\begin{equationcd}
		X \supseteq U	\to[d, "\psi", swap] \to[r, "\tilde{f}"]
			& \phi(V) \subseteq Y
			\\
		X \supseteq \psi(U) \to[r, "\hat{f} + f_\singularPart", swap]	
			& V \subseteq Y \to[u, "\phi", swap].
	\end{equationcd}
	Thus, in the charts \( \kappa = \psi \circ \tilde{\kappa} \) and \( \rho = \phi^{-1} \circ \tilde{\rho} \), the map \( f \) coincides with \( \normalForm{f} = \hat{f} + f_\singularPart \).
	This proofs~\eqref{eq:normalFormMap:bringIntoNormalForm:commutative}.
	Finally, let us verify the asserted properties of the singular part \( f_\singularPart \).
	We clearly have \( \tangent_0 f_\singularPart = 0 \).
	Moreover, for all \( x_2 \in U \intersect \coimg T \), we get
	\begin{equation}\begin{split}
		f_\singularPart \circ \psi(0, x_2)
			&= \pr_{\coker T} \left(\tilde{f}(0, x_2) - \tilde{f} \circ \psi^{-1}(0, \hat{T}^{-1} \circ \pr_{\img T} \circ \tilde{f}(0, x_2))\right)
			\\
			&= \pr_{\coker T} \left(\tilde{f}(0, x_2) - \tilde{f} \circ \psi^{-1} \circ \psi (0, x_2) \right) 
			\\
			&= 0.
	\end{split}\end{equation}
	From \( \psi(0, x_2) \in \coimg T \) it follows that \( \tilde{f}(0, x'_2) = 0 \) holds for all \( x'_2 \in \psi(U) \intersect \coimg T \).
\end{proof}
The first part of the proof of \cref{prop:normalFormMap:banach} is inspired by the Lyapunov--Schmidt reduction procedure.
To establish the link, let us give a brief outline of this procedure, mostly ignoring the peculiarities of the infinite-dimensional setting, see, \eg, \parencite[Section~1.3]{Chang2005}.
Given Banach spaces \( X \) and \( Y \), and a smooth map \( f: X \supseteq U \to Y \) defined on an open neighborhood \( U \) of \( 0 \) in \( X \) with \( f(0) = 0 \), we are interested in solutions of the nonlinear equation
\begin{equation}
	\label{eq:normalFormMap:LyapunovSchmidtStart}
	f(x) = 0
\end{equation}
near the solution \( x = 0 \).
The Lyapunov--Schmidt scheme consists of the following steps:
\begin{enumerate}
	\item
		Split \( X \) and \( Y \) into direct sums \( X = \ker T \oplus \coimg T \) and \( Y = \coker T \oplus \img T \), where \( T = \tangent_0 f: X \to Y \) as above.
		The equation~\eqref{eq:normalFormMap:LyapunovSchmidtStart} is then equivalent to the system
		\begin{equation}\label{eq:normalFormMap:LyapunovSchmidtDecomposed}\begin{split}
			\pr_{\coker T} \circ f (x_1, x_2) &= 0,
			\\
			\pr_{\img T} \circ f (x_1, x_2) &= 0,
		\end{split}\end{equation}
		with \( x_1 \in \ker T \) and \( x_2 \in \coimg T \).
	\item
		The Implicit Function Theorem shows that, after possibly shrinking \( U \), the second equation in~\eqref{eq:normalFormMap:LyapunovSchmidtDecomposed} has a unique solution \( x_2 = x_2 (x_1) \in U \intersect \coimg T \) as a function of \( x_1 \in U \intersect \ker T \).
	\item
		Substituting this solution of the second equation into the first equation of~\eqref{eq:normalFormMap:LyapunovSchmidtDecomposed} yields the reduced equation
		\begin{equation}
			\label{eq:normalFormMap:LyapunovSchmidtReduced}
			\pr_{\coker T} \circ f (x_1, x_2 (x_1)) = 0 .
		\end{equation}
		for the unknown \( x_1 \in U \intersect \ker T \).
\end{enumerate}
In this way, the nonlinear equation~\eqref{eq:normalFormMap:LyapunovSchmidtStart} is reduced to the nonlinear equation~\eqref{eq:normalFormMap:LyapunovSchmidtReduced}, which often happens to be a set of finitely many equations for a finite number of unknowns.
When comparing this reduction scheme with our construction of the chart deformation \( \psi \) in the proof of \cref{prop:normalFormMap:banach}, the only conceptual difference is our usage of the Inverse Function Theorem in place of the Implicit Function Theorem employed in the Lyapunov--Schmidt procedure.
Both methods rely fundamentally on the fact that the map \( \pr_{\img T} \circ f: X \supseteq U \to \img T \) has a surjective derivative at \( 0 \).
In our language, the reduced equation~\eqref{eq:normalFormMap:LyapunovSchmidtReduced} takes the form
\begin{equation}
	f_\singularPart(x_1, 0) = 0
\end{equation}
for \( x_1 \in U \intersect \ker T \).

Similar ideas are also used in the study of deformations of geometric objects, see, \eg, \parencite{Kuranishi1965} concerning deformations of complex structures and \parencites[Section~II.2]{Donaldson1983}[Section~6]{Taubes1982} in the gauge theoretic setting; see also \parencite[Section~4.2.5]{DonaldsonKronheimer1997}.
In this context, the counterpart of the local diffeomorphism \( \psi \) is usually referred to as the Kuranishi map.
\begin{remarks}
	\item
		A weaker version of \cref{prop:normalFormMap:banach} can be found in \parencites[Theorem~5.1.8]{MargalefRoigDominguez1992}[Theorem~2.5.14]{MarsdenRatiuEtAl2002}.
		There, the chart on \( N \) is not modified and hence the additional property \( f_\singularPart(0, \cdot) = 0 \) of the singular part is not deduced.
		Note that this property was crucial in the proof of \cref{prop:normalFormMap:immersion} to show that a smooth map with injective differential is an immersion.
	\item
		By \cref{prop:normalFormMap:immersion}, we get a corresponding normal form theorem for immersions.
		Note that in our construction of the normal form the chart on the domain is always deformed, which is in contrast to the classical immersion theorem (\eg, \parencite[Theorem~2.5.12]{MarsdenRatiuEtAl2002}).
		Hence, the constructed submanifold charts differ from the usual ones.
\end{remarks}

\subsection{Banach target or domain}
In the following, we give generalizations of the Banach normal form theorem to different analytic settings.
Let us start with the following extension of \cref{prop:normalFormMap:banach} to more general domains.
\begin{thm}[Normal form --- Banach target]
	\label{prop:normalFormMap:banachTarget}
	Let \( f: M \to N \) be a smooth map, where \( M \) is a locally convex manifold and \( N \) is a Banach manifold.
	Then \( f \) can be brought into a normal form at the point \( m \in M \) if and only if the differential \( \tangent_m f: \TBundle_m M \to \TBundle_{f(m)} N \) is a regular operator.
	In particular, every smooth map \( f: M \to N \) with finite-dimensional target \( N \) can be brought into a normal form at every point. 
\end{thm}
\begin{proof}
	The proof of \cref{prop:normalFormMap:banach} carries over word by word except for the part where the Inverse Function Theorem has been used to show that the map \( \psi \) defined in~\eqref{eq:normalFormMap:banach:deformDomain} is a local diffeomorphism.
	The idea here is to use the Inverse Function \cref{prop:inverseFunctionTheorem:banachWithParameters} due to \textcite{Gloeckner2006a} instead.
	For this purpose, define the smooth map \( \bar{\psi}: X \supseteq U \to \coimg T \) by
	\begin{equation}
		\bar{\psi}(x_1, x_2) = \hat{T}^{-1} \circ \pr_{\img T} \circ f(x_1, x_2).
	\end{equation}
	The partial derivative of \( \bar{\psi} \) at \( 0 \) with respect to the second component is given by \( \tangent^2_0 \bar{\psi} = \id_{\coimg T} \).
	Considering \( x_1 \in \ker \) as a parameter, the Inverse Function \cref{prop:inverseFunctionTheorem:banachWithParameters} shows that the map
	\begin{equation}
		\psi\bigl(x_1, x_2\bigr) = \bigl(x_1, \bar{\psi}(x_1, x_2)\bigr) 	
	\end{equation}
	is a local diffeomorphism.
	Note that \( \psi \) coincides with the map defined in~\eqref{eq:normalFormMap:banach:deformDomain}, so that the rest of the proof of \cref{prop:normalFormMap:banach} goes through without modification.
	The second part of the claim follows from \cref{prop::normalFormLinearMap:generalizedInverseWhenDomainFiniteDim}, which shows that every linear continuous map with finite-dimensional target is a regular operator.
\end{proof}
In combination with \cref{prop:normalFormMap:submersionImmersionConstantRank}, we recover the submersion theorem and the constant rank theorem \parencite[Theorem~A and~F]{Gloeckner2015} for maps with values in a finite-dimensional manifold as a special case of \cref{prop:normalFormMap:banachTarget}.

We also have the following version of the normal form theorem where the domain is a Banach manifold.
\begin{thm}[Normal form --- Banach domain]
	\label{prop:normalFormMap:banachDomain}
	Let \( f: M \to N \) be a smooth map between manifolds, where \( M \) is a Banach manifold and \( N \) is a locally convex manifold.
	Then \( f \) can be brought into a normal form at the point \( m \in M \) if and only if the differential \( \tangent_m f: \TBundle_m M \to \TBundle_{f(m)} N \) is a regular operator.
	In particular, every smooth map \( f: M \to N \) with \( M \) being finite-dimensional can be brought into a normal form around every point. 
\end{thm}
\begin{proof}
	The proof proceeds similarly to \cref{prop:normalFormMap:banach} with the modification that Glöckner's Inverse Function \cref{prop:inverseFunctionTheorem:banachWithParameters} is used to show that \( \phi \) defined in~\eqref{eq:normalFormMap:banach:deformTarget} is a local diffeomorphism
	 (using a similar strategy as in the proof of \cref{prop:normalFormMap:banachTarget}).
	Details are left to the reader.
\end{proof}
In conjunction with \cref{prop:normalFormMap:immersion}, we recover the immersion theorem \parencite[Theorem~H]{Gloeckner2015} for maps from a Banach manifold into a locally convex manifold.
\subsection{Nash--Moser version}
We now establish a normal form theorem in the tame Fréchet category using the Nash--Moser Inverse Function Theorem.
For the convenience of the reader, the basic concepts are briefly summarized in \cref{sec:tameFrechet}.

In contrast to the Banach Inverse Function Theorem, the Nash--Moser Theorem requires the derivative to be invertible in a \emph{whole neighborhood} of a given point.
This additional condition is due to the fact that the subset of invertible operators is no longer open in the space of all operators.
In \cref{sec:familyGeneralizedInverse}, we have introduced the notion of uniform regularity to address similar problems.
Uniform regularity plays a major role in the context of normal forms, too.
In fact, the following result has been the main inspiration for this concept.
\begin{prop}
	\label{prop:normalFormMap:normalFormThenDerivativeUniformRegular}
	For every normal form \( (X, Y, \hat{f}, f_\singularPart) \), the family \( \tangent_x \normalForm{f}: X \to Y \) of linear maps parametrized by \( x \in U \subseteq X \) is uniformly regular at \( 0 \).
\end{prop}
\begin{proof}
	With respect to the decompositions \( X = \ker \oplus \coimg \) and \( Y = \coker \oplus \img \), the derivative of \( \normalForm{f} \) at \( x \in U \) is given in block form as
	\begin{equation}
		\label{eq:normalFormMap:derivativeOfNormalForm}
		\tangent_x \normalForm{f} = 
			\Matrix{
				\restr{(\tangent_x f_\singularPart)}{\ker}
				&
				\restr{(\tangent_x f_\singularPart)}{\coimg}
				\\
				0
				&
				\hat{f}
			}.
	\end{equation}
	In particular, \( \tangent_0 \normalForm{f} = \hat{f} \circ \pr_{\coimg} \), because \( \tangent_0 f_\singularPart = 0 \).
	This shows that the abstract spaces \( \coimg \subseteq X \) and \( \coker \subseteq Y \) are topological complements of the kernel and image of \( \tangent_0 \normalForm{f} \), respectively.
	Moreover, the map
	\begin{equation}
		\pr_{\img} \circ \restr{(\tangent_x \normalForm{f})}{\coimg}: \coimg \to \img
	\end{equation}
	coincides with the isomorphism \( \hat{f} \), which confirms that \( \tangent_x \normalForm{f} \) is uniformly regular at \( 0 \).
\end{proof}

In order to extend the notion of uniform regularity to the setting of manifolds, consider a morphism \( T: E \to F \) of vector bundles over a manifold \( M \).
Suppose that \( E \) and \( F \) are trivialized over an open subset \( U \subseteq M \), that is, we are given vector bundle isomorphism \( \restr{E}{U} \to U \times \FibreBundleModel{E} \) and \( \restr{F}{U} \to U \times \FibreBundleModel{F} \), where \( \FibreBundleModel{E} \) and \( \FibreBundleModel{F} \) are locally convex spaces.
With respect to these trivializations, we identify \( T \) with a family \( \restr{T}{U}: U \times \FibreBundleModel{E} \to \FibreBundleModel{F} \) of linear maps.
\begin{defn}
	\label{defn::normalFormMap:uniformRegularBundle}
	A morphism \( T: E \to F \) of vector bundles over \( M \) is called \emphDef{uniformly regular} at \( m \in M \) if there exist local trivializations of \( E \) and \( F \) on an open neighborhood \( U \) of \( m \) such that the induced family \(  \restr{T}{U}: U \times \FibreBundleModel{E} \to \FibreBundleModel{F} \) of linear maps is uniformly regular at \( m \) in the sense of \cref{defn:familyGeneralizedInverse:uniformRegular}.
	Similarly, a tame morphism \( T: E \to F \) between tame Fréchet vector bundles is called \emphDef{uniformly tame regular} at \( m \), if there exist tame local trivializations such that the induced family \(  \restr{T}{U}: U \times \FibreBundleModel{E} \to \FibreBundleModel{F} \) is uniformly tame regular at \( m \).
\end{defn}

Phrased in this language, \cref{prop:normalFormMap:normalFormThenDerivativeUniformRegular} entails that the derivative of \( f: M \to N \), viewed as a vector bundle map \( \tangent f: \TBundle M \to f^* \TBundle N \) over \( M \), is uniformly regular at \( m \in M \) if \( f \) can be brought into a normal form at \( m \).
The following theorem shows that, in the tame category, uniform regularity of the derivative is also a sufficient condition for the existence of a normal form.
\begin{thm}[Normal form --- Tame Fréchet]
	\label{prop::normalFormMap:tameFrechet}
	A tame smooth map \( f: M \to N \) between tame Fréchet manifolds can be brought into a tame normal form at the point \( m \in M \) if and only if \( \tangent f: \TBundle M \to f^* \TBundle N \) is uniformly tame regular at \( m \).
\end{thm}
\begin{proof}
	If \( f \) can be brought into a tame normal form, then \cref{prop:normalFormMap:normalFormThenDerivativeUniformRegular} (phrased in the tame category) shows that \( \tangent f: \TBundle M \to f^* \TBundle N \) is uniformly tame regular at \( m \).
	
	The proof of the converse direction follows the same line of arguments as the proof of \cref{prop:normalFormMap:banach} except that we will use the Nash--Moser \cref{prop:inverseFunctionTheorem:nashMoser} to show that the chart deformations~\eqref{eq:normalFormMap:banach:deformDomain} and~\eqref{eq:normalFormMap:banach:deformTarget} are local diffeomorphisms.
	Continuing in the notation of the proof of \cref{prop:normalFormMap:banach}, abbreviate \( T_x \equiv \tangent_x f: X \to Y \) for every \( x \in U \subseteq X \).
	The assumption of uniform tame regularity of \( \tangent f \) implies that the family \( T: U \times X \to Y \) is uniformly tame regular at \( 0 \).

	The derivative at \( x \in U \) of the map \( \psi \) defined in~\eqref{eq:normalFormMap:banach:deformDomain} evaluates to
	\begin{equation}
		\tangent_{x} \psi 
			= 
			\begin{pmatrix}
				\id_{\ker T_0}	& 0 \\
				\hat{T}_0^{-1} \circ \pr_{\img T_0} \circ \restr{(T_{x})}{\ker T_0}	& \hat{T}_0^{-1} \circ \tilde{T}^{}_x
			\end{pmatrix},
	\end{equation}
	where \( \tilde{T}_x = \pr_{\img T_0} \circ \restr{(T_x)}{\coimg T_0}: \coimg T_0 \to \img T_0 \).
	Since \( T \) is uniformly tame regular, \( \tilde{T}_x \) is invertible for all \( x \in U \) and the inverses form a tame smooth family.
	Hence, \( \tangent_x \psi \) has a tame smooth family of inverses given by
	\begin{equation}
		\label{eq:normalFormMap:tameFrechet:inverseDeformDomain}
		\left(\tangent_{x} \psi\right)^{-1} 
			= 
			\begin{pmatrix}
				\id_{\ker T_0}	& 0 \\
				- \tilde{T}_x^{-1} \circ \pr_{\img T_0} \circ \restr{(T_{x})}{\ker T_0}	& \tilde{T}_x^{-1} \circ \hat{T}^{}_0 
			\end{pmatrix}.
	\end{equation}
	Thus, we can apply the Nash--Moser \cref{prop:inverseFunctionTheorem:nashMoser} to conclude that \( \psi \) is a local diffeomorphism at \( 0 \).

	Using~\eqref{eq:normalFormMap:tameFrechet:inverseDeformDomain}, the derivative of the map \( \phi \) defined in~\eqref{eq:normalFormMap:banach:deformTarget} can be written in block form as
	\begin{equation}
		\tangent_y \phi
			= 
			\begin{pmatrix}
				\id_{\coker T_0} 	& \pr_{\coker T_0} \circ T^{}_x \circ \tilde{T}_x^{-1} \\
				0			& \id_{\img T_0}
			\end{pmatrix},
	\end{equation}
	where \( y \in V \subseteq Y \) and \( x = \psi^{-1}\bigl(0, \hat{T}^{-1} \circ \pr_{\img T_0}(y)\bigr) \).
	A direct calculation verifies that \( \tangent_y \phi \) is invertible with inverse given by
	\begin{equation}
		\left(\tangent_y \phi\right)^{-1} = 
			\begin{pmatrix}
				\id_{\coker T_0} 	& - \pr_{\coker T_0} \circ T^{}_x \circ \tilde{T}_x^{-1} \\
				0			& \id_{\img T_0}
			\end{pmatrix}.
	\end{equation}
	Since the inverses \( \left(\tangent_y \phi\right)^{-1} \) parametrized by \( y \in V \) form a tame smooth family, the Nash--Moser \cref{prop:inverseFunctionTheorem:nashMoser} implies that \( \phi \) is a local diffeomorphism.
	The remainder of the proof of \cref{prop:normalFormMap:banach} goes through without modification.
\end{proof}

\subsection{Elliptic version}
\label{sec:normalFormMap:elliptic}

In applications to geometry and physics, one is usually interested in differential operators between spaces of geometric objects.
Let \( E \to M \) and \( F \to M \) be finite-dimensional fiber bundles over the compact manifold \( M \).
Denote the space of smooth sections of \( E \) and \( F \) by \( \SectionSpaceAbb{E} \) and \( \SectionSpaceAbb{F} \), respectively.
By \parencite[Theorem~II.2.3.1]{Hamilton1982}, the spaces \( \SectionSpaceAbb{E} \) and \( \SectionSpaceAbb{F} \) are tame Fréchet manifolds.
Following \parencite[Definition~15.3]{Palais1968}, a \emphDef{nonlinear differential operator} of degree \( r \) is a map \( \SectionSpaceAbb{E} \to \SectionSpaceAbb{F} \) factorizing through the \( r \)-jet bundle \( \JetBundle^r E \).
That is, for every vertical morphism \( f: \JetBundle^r E \to F \) of fiber bundles the associated nonlinear differential operator is the composition
\begin{equationcd}
	L\tikzcolon \SectionSpaceAbb{E}
	\to[r, "\jet^r"]
	&\sSectionSpace(\JetBundle^r E)
	\to[r, "f_*"]
	&\SectionSpaceAbb{F},
\end{equationcd}
where \( \jet^r \) denotes the \( r \)-th jet prolongation.
This is similar to the factorization~\eqref{eq:normalFormLinearMap:familiesEllipticOps:factorizationDiffOp} in the linear case, with the important difference that \( f \) is now a nonlinear fiber bundle morphism.
Every nonlinear differential operator \( L_f: \SectionSpaceAbb{E} \to \SectionSpaceAbb{F} \) is a tame smooth map according to \parencite[Corollary~II.2.2.7]{Hamilton1982}.

As one would expect, the linearization of a nonlinear differential operator is a linear differential operator.
In fact, the differential \( \tangent_\phi L_f \) is given by
\begin{equationcd}
	\tangent_\phi L_f \tikzcolon \sSectionSpace(\phi^* \VBundle E)
		\to[r, "\jet^r"]
		&\sSectionSpace(\JetBundle^r (\phi^* \VBundle E))
		\to[d, "\isomorph"]
		&
	\\
		&\sSectionSpace((\jet^r \phi)^* \VBundle (\JetBundle^r E))
		\to[r, "(\vtangent f)_*"]
		&\sSectionSpace((L_f (\phi))^* \VBundle F),
\end{equationcd}
where \( \VBundle E \) denotes the vertical subbundle of \( \TBundle E \) and \( \vtangent f: \VBundle (\JetBundle^r E) \to f^* \VBundle F \) is the vertical derivative of \( f \).
The isomorphism at the center is induced by the natural isomorphism of \( (\jet^r \phi)^* \VBundle (\JetBundle^r E) \) and \( \JetBundle^r (\phi^* \VBundle E) \), see \parencite[Theorem~17.1]{Palais1968}.
Hence, \( \tangent_\phi L_f \) is a linear differential operator with coefficients \( \vtangent f \).

We say that \( L_f \) is \emphDef{elliptic} if the linear differential operator \( \tangent_\phi L_f \) is elliptic for all \( \phi \in \SectionSpaceAbb{E} \).
\begin{thm}[Normal form --- Elliptic]
	\label{prop::normalFormMap:elliptic}
	Let \( E \to M \) and \( F \to M \) be finite-dimensional fiber bundles over the compact manifold \( M \).
	Every nonlinear elliptic differential operator \( L_f: \SectionSpaceAbb{E} \to \SectionSpaceAbb{F} \) can be brought into a tame normal form at every \( \phi \in \SectionSpaceAbb{E} \).
\end{thm}
\begin{proof}
	According to \cref{prop::normalFormMap:tameFrechet}, we have to show that the bundle map
	\begin{equation}
		\tangent L_f: \TBundle \SectionSpaceAbb{E} \to (L_f)^* \TBundle \SectionSpaceAbb{F}
	\end{equation}
	is uniformly tame regular at \( \phi \in \SectionSpaceAbb{E} \).
	After having chosen tubular neighborhoods of \( \img \phi \subseteq E \) and \( \img L_f(\phi) \subseteq F \), it suffices to consider the case where \( E \) and \( F \) are vector bundles.
	In this linear setting, the vertical tangent bundle of \( \JetBundle^r E \) is identified with \( \VBundle(\JetBundle^r E) \isomorph \JetBundle^r E \times_M \JetBundle^r E \).
	Accordingly, the vertical derivative of \( f \) is a map \( \vtangent f: \JetBundle^r E \times_M \JetBundle^r E \to F \).
	Define the tame smooth family \( T: \SectionSpaceAbb{E} \times \SectionSpaceAbb{E} \to \SectionSpaceAbb{F} \) of linear maps by
	\begin{equation}
		T(\varphi, \sigma) \equiv T_\varphi (\sigma) = \vtangent f (\jet^r \varphi, \jet^r \sigma).
	\end{equation}
	We have to show that the family \( \varphi \mapsto T_\varphi \) is uniformly tame regular at \( \varphi = \phi \).
	For this purpose, note that \( T \) factorizes as the composition of the tame smooth map 
	\begin{equation}
		\SectionSpaceAbb{E} \times \SectionSpaceAbb{E} \ni (\varphi, \sigma) \mapsto \bigl((\vtangent f)_* (\jet^r \varphi, \cdot), \sigma) \in \sSectionSpace(\LinMapBundle(\JetBundle^r E, F)) \times \SectionSpaceAbb{E}
	\end{equation}
	and the family of differential operators
	\begin{equation}
		D: \sSectionSpace\bigl(\LinMapBundle(\JetBundle^r E, F)\bigr) \times \SectionSpaceAbb{E} \to \SectionSpaceAbb{F},
		\qquad
		(\Lambda, \sigma) \mapsto \Lambda (\jet^r \sigma).
	\end{equation}
	By assumption, the differential operator with coefficients \( \Lambda_\phi = (\vtangent f)_* (\jet^r \phi, \cdot) \) is elliptic.
	Thus, \cref{prop:normalFormLinearMap:familiesEllipticOps:areUniformlyTameRegular} implies that \( D \) is uniformly tame regular at \( \Lambda_\phi \), and hence that the family \( T \) is uniformly tame regular at \( \varphi = \phi \).
\end{proof}

In applications, one often encounters infinite-dimensional manifolds which are not quite spaces of sections of a fiber bundle.
A prime example is the Fréchet Lie group of diffeomorphisms of a compact manifold.
Note that the diffeomorphism group is locally modeled on the space of vector fields.
Thus, from the perspective of local normal forms, it has a very similar structure to the space of sections.
In order to extend the normal form theorem to such situations, we follow \parencite[p.~57]{Subramaniam1984} and introduce the following subclass of Fréchet manifolds.
\begin{defn}
	\label{defn:normalFormMap:geometric}
	A tame Fréchet manifold \( M \) is said to be \emphDef{geometric} if it is locally modeled on the space of smooth sections of some vector bundle over a compact manifold.
	
	A tame smooth map \( f: M \to N \) is called \emphDef{geometric} if for every point \( m \in M \) there exist vector bundles \( E \) and \( F \) over the same compact manifold, and tame local trivializations \( \restr{(\TBundle M)}{U} \isomorph U \times \sSectionSpace(E) \) and \( \restr{(f^* \TBundle N)}{U} \isomorph U \times \sSectionSpace(F) \) in a neighborhood \( U \subseteq M \) of \( m \) such that in this trivialization the derivative \( \tangent f: \TBundle M \to f^* \TBundle N \) factorizes as the composition of a tame smooth map \( U \to \sSectionSpace(\LinMapBundle(\JetBundle^r E, F)) \) and the family of differential operators
	\begin{equation}
		\label{eq:normalFormMap:geometric:diffOps}
		D: \sSectionSpace(\LinMapBundle(\JetBundle^r E, F)) \times \sSectionSpace(E) \to \sSectionSpace(F),
		\qquad
		(\Lambda, \sigma) \mapsto \Lambda (\jet^r \sigma).
	\end{equation}
	The notions of a geometric vector bundle and a geometric vector bundle morphism are defined in a similar way.
\end{defn}
Roughly speaking, a map \( f: M \to N \) is geometric if its linearization \( \tangent_m f \) is a family of linear differential operators whose coefficients depend tamely on \( m \in M \).
As we have seen in the proof of \cref{prop::normalFormMap:elliptic}, every nonlinear differential operator is a geometric map.
Moreover, a slight reformulation of the proof of \cref{prop::normalFormMap:elliptic} gives the following more general normal form result.
\begin{thm}[Normal form --- Elliptic]
	\label{prop::normalFormMap:ellipticGeometric}
	Let \( f: M \to N \) be a geometric map between geometric Fréchet manifolds.
	If \( \tangent_m f: \TBundle_m M \to \TBundle_{f(m)} N \) is an elliptic differential operator for some \( m \in M \), then \( f \) can be brought into a tame normal form at \( m \).
\end{thm}
\begin{proof}
	By \cref{prop::normalFormMap:tameFrechet}, we have to show that the bundle map \( \tangent f: \TBundle M \to f^* \TBundle N \) is uniformly tame regular at \( m \).
	Since \( f \) is geometric, there exist tame local trivializations \( \restr{(\TBundle M)}{U} \isomorph U \times \sSectionSpace(E) \) and \( \restr{(f^* \TBundle N)}{U} \isomorph U \times \sSectionSpace(F) \) in a neighborhood \( U \subseteq M \) of \( m \) such that \( \tangent f \) factorizes as the composition of a tame smooth map \( S: U \to \sSectionSpace(\LinMapBundle(\JetBundle^r E, F)) \) and the family of differential operators \( D: \sSectionSpace(\LinMapBundle(\JetBundle^r E, F)) \times \sSectionSpace(E) \to \sSectionSpace(F) \) defined in~\eqref{eq:normalFormMap:geometric:diffOps}.
	By assumption, the differential operator \( D_{S(m)} \) with symbol \( S(m) \) is elliptic.
	Hence, \cref{prop:normalFormLinearMap:familiesEllipticOps:areUniformlyTameRegular} implies that the family \( \tangent f \) is uniformly tame regular at \( m \).
\end{proof}

\subsection{Relative normal form}
We now briefly discuss how to bring a smooth map \( f: M \to N \) into a normal form relative to a given submanifold \( P \) of \( N \). 
This is a direct generalization of the standard transversality theory.
Recall that a submanifold \( P \subseteq N \) is called \emphDef{split} if \( \TBundle_p P \) is topologically complemented in \( \TBundle_p N \) for every point \( p \in P \).
Let \( \pr_p: \TBundle_p N \to \TBundle_p N \slash \TBundle_p P \) denote the canonical projection.
For the special case where \( P \) is a single point, the following relative normal form result recovers \cref{prop:normalFormMap:banach}. 
\begin{thm}[Relative normal form --- Banach]
	\label{prop:normalFormMap:banachRelative}
	Let \( f: M \to N \) be a smooth map between Banach manifolds, and let \( P \) be a split submanifold of \( N \).
	If the operator \( \pr_{f(m)} \circ \tangent_m f: \TBundle_m M \to \TBundle_{f(m)} N \slash \TBundle_{f(m)} P \) is regular at the point \( m \in M \), then there exist charts \( \kappa: M \supseteq U' \to U \subseteq X \) at \( m \) and \( \rho: N \supseteq V' \to V \subseteq Y \) at \( f(m) \), and decompositions \( X = \ker \oplus \coimg \), \( Y = \coker \oplus \img \oplus Z \) such that
	\begin{equation}
		\label{eq:normalFormMap:banachRelative:normalForm}
		\rho \circ f \circ \kappa^{-1} = \hat{f} + f_\singularPart + f_P,
	\end{equation}
	where \( \hat{f}: \coimg \to \img \) is a linear topological isomorphism, \( f_\singularPart: X \supseteq U \to \coker \) is a smooth map satisfying \( f_\singularPart(0, x_2) = 0 \) for all \( x_2 \in U \intersect \coimg \) and \( \tangent_0 f_\singularPart = 0 \), and \( f_P: X \supseteq U \to Z \) is a smooth map.
	The chart \( \rho \) can be chosen to be a submanifold chart for \( P \), \ie, \( \rho(P \intersect V') = Z \intersect V \).
\end{thm}
\begin{proof}
	Since \( P \) is a split submanifold of \( N \), there exist a chart \( \tilde{\rho}: N \supseteq V' \to V \subseteq Y \) at \( f(m) \), and a topological decomposition \( Y = \bar{Y} \oplus Z \) such that \( \tilde{\rho}(P \intersect V') = Z \intersect V \).
	Define \( \bar{f} \defeq \pr_{\bar{Y}} \circ \tilde{\rho} \circ \restr{f}{U'}: M \supseteq U' \to \bar{Y} \) and \( f_P \defeq \pr_Z \circ \tilde{\rho} \circ \restr{f}{U'} : M \supseteq U' \to Z \) where \( U' \) is a sufficiently small neighborhood of \( m \) so that \( f(U') \subseteq V' \).
	Since \( \pr_{f(m)} \circ \tangent_m f \) coincides with \( \tangent_m \bar{f} \) under the isomorphism \( \TBundle_{f(m)} N \slash \TBundle_{f(m)} P \isomorph \bar{Y} \), the operator \( \tangent_m \bar{f} \) is regular.
	The claim now follows from \cref{prop:normalFormMap:banach} applied to \( \bar{f} \).
\end{proof}
Since \( \rho \) is a submanifold chart for \( P \) and \( f_P \) takes values in \( Z \), the identity~\eqref{eq:normalFormMap:banachRelative:normalForm} implies
\begin{equation}
	\kappa (f^{-1}(P)) = \set[big]{(x_1, 0) \in U \given f_\singularPart(x_1, 0) = 0}.
\end{equation}
Note that \( \ker \isomorph (\tangent_m f)^{-1}\bigl(\TBundle_{f(m)} P\bigr) \).
Moreover, \( \coker \) is isomorphic to the quotient of \( \TBundle_{f(m)} N \) by \( \img \tangent_m f + \TBundle_{f(m)} P \).
We hence recover the classical result that \( f^{-1}(P) \) is a submanifold of \( M \) if \( f \) is transversal to \( P \). 

The strategy in the proof of \cref{prop:normalFormMap:banachRelative} generalizes to the other analytic settings considered above, leading to relative versions of \cref{prop:normalFormMap:banachTarget,prop:normalFormMap:banachDomain,prop::normalFormMap:tameFrechet,prop::normalFormMap:elliptic,prop::normalFormMap:ellipticGeometric}.
We leave the details to the reader.

\section{Normal form of an equivariant map}
\label{sec:normalFormEquivariantMap}

In this section, we study the local properties of a smooth equivariant map.
We refer the reader to \cref{sec:calculus:groupActionsSlices} for relevant background material on Lie group actions in infinite dimensions.
Consider the following setup.
Let \( G \) be a Lie group, and let \( M \) and \( N \) be \( G \)-manifolds.
Assume that the \( G \)-action on \( M \) is proper.
Let \( f: M \to N \) be a smooth equivariant map.
Choose a point \( m \in M \) and denote its image under \( f \) by \( \mu = f(m) \in N \).

Let us start by describing the general idea of how to construct an equivariant normal form of \( f \).
Since a slice provides a normal form for the \( G \)-action in a neighborhood of a given orbit, it naturally comes to mind as a tool for studying equivariant maps.
An initial idea would be to split-off the \( G \)-action with the help of slices \( S \) and \( R \) for the \( G \)-action at the points \( m \in M \) and \( \mu \in N \), respectively.
If the slices satisfy \( f(S) \subseteq R \), then the part of \( f \) that does not come from the group action is captured in the map \( f_S^R = \restr{f}{S}: S \to R \) between the slices (see \cref{fig:normalFormEquivariantMap:twoSlices}).
For the reduced map \( f_S^R \), we can use the normal form results of \cref{sec:normalFormMap} to arrive at an equivariant normal form for \( f \).
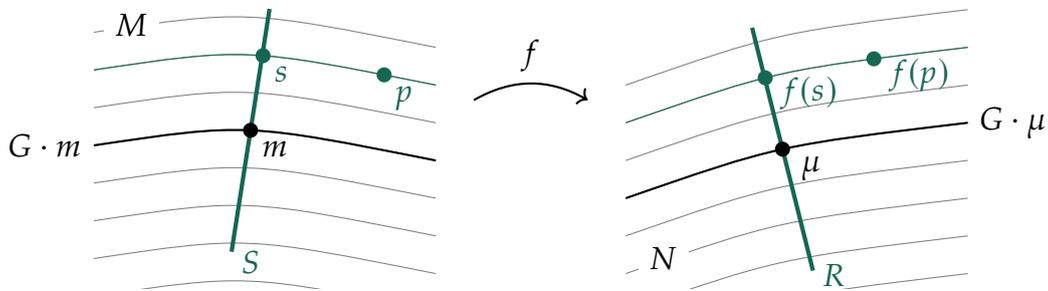
\begin{figure}[b]
	\centering
	\begin{tikzpicture}
		\begin{scope}
			\path[clip] (-2, -1.5) rectangle (2.5, 2.5);

			\foreach \w in {-1.8, -1.3,...,2.2}
				\draw[thin, gray] (-2, 1.2-1.5+\w) .. controls (0, 1.5-1.5+\w) .. (2.5, 1-1.5+\w);

			\node[fill=white] at (-1.5, 2) {$M$};
		\end{scope}

		\begin{scope}
			\path[clip] (5, -1.5) rectangle (9.5, 2.5);

			\foreach \w in {-1.8, -1.3,...,2.2}
				\draw[thin, gray] (5, 0.5-1.5+\w) .. controls (7, 1.2-1.5+\w) .. (9.5, 1.5-1.5+\w);

			\node[fill=white] at (5.5, -1.1) {$N$};
		\end{scope}

		\draw[thick, black]
			(-2, 1.2-1.5+0.7) .. controls (0, 1.5-1.5+0.7) .. (2.5, 1-1.5+0.7) 
			node[midway](m){} 
			node[pos=0.0](orbitM){};

		\draw[thick, black]
			(5, 0.5-1.5+0.7) .. controls (7, 1.2-1.5+0.7) .. (9.5, 1.5-1.5+0.7) 
			node[midway](mu){} 
			node[pos=1.0](orbitN){};
		
		\draw[ultra thick, green4] ($(m) + (0.25, 1.61)$) -- ($(m) + (-0.25, -1.61)$) node[below right=-0.18 and -0.04]{$S$};
		\draw[ultra thick, green4] ($(mu) + (-0.4, 1.61)$) -- ($(mu) + (0.4, -1.61)$) node[below right=-0.26 and -0.04]{$R$};

		\path (m) node[circle, fill=black, inner sep=2pt]{};
		\path (m) node[below right] {$m$};

		\path (mu) node[circle, fill=black, inner sep=2pt]{};
		\path (mu) node[below right=-0.05 and 0.1] {$\mu$};

		\draw[green4]
			(-2, 1.2-1.5+1.7) .. controls (0, 1.5-1.5+1.7) .. (2.5, 1-1.5+1.7) 
			node[pos=0.548, circle, fill=green4, inner sep=2pt](s){} 
			node[pos=0.9, circle, fill=green4, inner sep=2pt](p){}; 
		\draw[green4]
			(5, 0.5-1.5+1.7) .. controls (7, 1.2-1.5+1.7) .. (9.5, 1.5-1.5+1.7) 
			node[pos=0.432, circle, fill=green4, inner sep=2pt](fs){} 
			node[pos=0.8, circle, fill=green4, inner sep=2pt](fp){}; 
		\path (s) node[below right, green4] {$s$};
		\path (fs) node[below right=-0.17 and 0.05, green4] {$f(s)$};
		\path (p) node[below right, green4] {$p$};
		\path (fp) node[below right=-0.15 and 0.05, green4] {$f(p)$};
		
		\path (orbitM) node[left] {$G \cdot m$};
		\path (orbitN) node[right] {$G \cdot \mu$};

		\draw[->, thick] (3, 1) to[out=30,in=150] node[pos=0.5, above]{$f$} (4.5, 1);
	\end{tikzpicture}
	\caption{Illustration of the idea to capture the part of \( f \) that does not come from the group action in a map \( f_S^R: S \to R \) between the slices \( S \) and \( R \).}
	\label{fig:normalFormEquivariantMap:twoSlices}
\end{figure}
In the finite-dimensional setting, a similar strategy has been used in \parencite[Proposition 2.6]{PflaumWilkin2017} to establish an equivariant submersion theorem under the assumption that \( G \) is compact.

However, for the applications we are interested in, it is often too restrictive to assume that the \( G \)-action on the target has a slice \( R \).
For example, momentum maps in symplectic geometry take values in the dual of the Lie algebra; but the coadjoint action of a non-compact group usually does not possess slices, see, \eg, \parencite[Point~5 in Section~2.3.1]{GuilleminLermanEtAl2009}.
As a solution, one may try to use only a slice \( S \) for the \( G \)-action on \( M \) while leaving the action on \( N \) untouched.
In the context of symplectic reduction, this idea is at the heart of the proof of the Marle--Guillemin--Sternberg normal form of the momentum map.
In this case, the slice \( S \) is constructed using an equivariant version of Darboux's theorem.
However, a close inspection of the proof of the singular symplectic reduction theorem reveals that the Marle--Guillemin--Sternberg Theorem gives actually \emph{not} the most convenient normal form for the study of the geometry of the symplectic quotient.
In fact, several issues arise from using a slice for the \( G \)-action but taking the quotient with respect to the subgroup \( G_\mu \).
This asymmetry was counterbalanced in the proof of the reduction theorem \parencite[Proposition~8.1.2]{OrtegaRatiu2003} by deforming the momentum map \( J: M \to \LieA{g}^* \) using a local diffeomorphism of \( \LieA{g}^* \).
Instead of changing coordinates on the target, we will take a different approach and use a slice on \( M \) for the action of \( G_\mu \) instead of \( G \).
This has also the advantage that \( G_\mu \) is often considerably smaller than \( G \) and, thus, a slice for the subgroup action is easier to construct.

Finally, the symplectic setting is special in that the action of \( G \) on \( \LieA{g}^* \) is linear.
When moving beyond the momentum map example, the \( G \)-action on the target \( N \) is usually nonlinear.
As we have explained above, we cannot use a slice to control the entire \( G \)-action on \( N \).
But the stabilizer \( G_m \) of \( m \)	is compact and its action on \( N \) leaves \( \mu \) invariant due to \( G \)-equivariance of \( f \).
Hence, we can at least hope to linearize the induced action of \( G_m \) around the fixed point \( \mu \).

In summary, we will work with a \( G_\mu \)-slice \( S \) at \( m \in M \) and with a \( G_m \)-slice at \( \mu \in N \) and then bring the restriction of \( f \) to the slice \( S \) in a normal form using the results of \cref{sec:normalFormMap}.

\subsection[General version]{General normal form theorem}

Using this strategy, we will show that a wide class of equivariant maps can be brought into the following equivariant normal form.
\begin{defn}
	\label{def:normalFormEquivariantMap:abstractNormalFormEquivariantMap}
	An \emphDef{abstract equivariant normal form} is a tuple \( (H, X, Y, \hat{f}, f_\singularPart) \) consisting of a compact Lie group \( H \) and a normal form \( (X, Y, \hat{f}, f_\singularPart) \) which is \( H \)-equivariant in the sense that \( X, Y \) are endowed with smooth linear \( H \)-actions, the decompositions \( X = \ker \oplus \coimg \), \( Y = \coker \oplus \img \) are \( H \)-invariant, and the maps \( \hat{f}: \coimg \to \img \), \( f_\singularPart: X \supseteq U \to \coker \) are \( H \)-equivariant, where \( U \subseteq X \) is an \( H \)-invariant open neighborhood of \( 0 \).

	For an equivariant normal form \( (H, X, Y, \hat{f}, f_\singularPart) \) and a Lie group \( G \) with \( H \subseteq G \), define the smooth map \( \normalFormGroup{G}{f}: G \times_{H} U \to G \times_{H} Y \) by
	\begin{equation}
		\normalFormGroup{G}{f} \bigl(\equivClass{g, x_1, x_2}\bigr) \defeq \equivClass*{g, \normalForm{f}(x_1, x_2)} = \equivClass*{g, \hat{f}(x_2) + f_\singularPart(x_1, x_2)}
	\end{equation}
	for \( g \in G \), \( x_1 \in U \intersect \ker \) and \( x_2 \in U \intersect \coimg \).
\end{defn}

A slice for the \( G_\mu \)-action on \( M \) plays a fundamental role in connecting an equivariant map with its equivariant normal form.
Recall from \cref{sec:calculus:groupActionsSlices}, that a \( G_\mu \)-slice at \( m \in M \) is a \( G_m \)-invariant submanifold \( S \) of \( M \) which is transverse to the orbit \( G_\mu \cdot m \) and which possesses some additional properties, see \cref{defn:slice:slice}.
A proper action of a finite-dimensional Lie group on a finite-dimensional manifold always admit slices \parencite{Palais1961}, but this may no longer be true in infinite dimensions and additional assumptions have to be made, see \parencite{DiezSlice}.
In the following, we assume the existence of a slice for the \( G_\mu \)-action at every point \( m \in M \). 
In particular, the natural fibration \( G_\mu \to G_\mu \slash G_m \) is a locally trivial principal bundle and it admits a local section \( \chi: G_\mu \slash G_m \supseteq W \to G_\mu \) defined on an open neighborhood \( W \) of the identity coset \( \equivClass{e} \) in such a way that the map
\begin{equation}
	\chi^S: W \times S \to M, \qquad (\equivClass{g}, s) \mapsto \chi(\equivClass{g}) \cdot s
\end{equation}
is a diffeomorphism onto an open neighborhood of \( m \), \ie, the slice yields convenient product coordinates in a neighborhood of \( m \).
This product structure extends to a semi-local normal form of the orbit using the tube map
\begin{equation}
	\chi^\tube: G_\mu \times_{G_m} S \to M, \qquad \equivClass{g, s} \mapsto g \cdot s,
\end{equation}
which is a \( G_\mu \)-equivariant diffeomorphism onto an open, \( G_\mu \)-invariant neighborhood of \( G_\mu \cdot m \) in \( M \), see \cref{prop:slice:sliceToTube}.
\begin{defn}
	\label{def:normalFormEquivariantMap:normalFormEquivariantMap}
	Let \( f: M \to N \) be a smooth \( G \)-equivariant map.
	For \( m \in M \) write \( \mu = f(m) \in N \), and assume that the stabilizer \( G_\mu \) is a Lie subgroup of \( G \).
	We say that \( f \) can be \emphDef{brought into the equivariant normal form \( (H, X, Y, \hat{f}, f_\singularPart) \)} at \( m \) if \( H = G_m \) and there exist
	\begin{enumerate}
		\item 
			a slice \( S \) at \( m \) for the \( G_\mu \)-action,
		\item 
			a \( G_m \)-equivariant diffeomorphism \( \iota_S: X \supseteq U \to S \subseteq M \), 
		\item 
			a \( G_m \)-equivariant chart \( \rho: N \supseteq V' \to V \subseteq Y \) at \( \mu \) with \( f(S) \subseteq V' \)
	\end{enumerate}
	such that the following diagram commutes:
	\begin{equationcd}[label=eq:normalFormEquivariantMap:bringIntoNormalForm]
		M \to[r, "f"] 
			& N
			\\
		\mathllap{G_\mu \times_{G_m} X \supseteq {}} G_\mu \times_{G_m} U
			\to[r, "\normalFormGroup{G_\mu}{f}"]
			\to[u, "{\chi^\tube} \, \circ \, {(\id_{G_\mu} \times \iota_S)}"]
			& G_\mu \times_{G_m} V \mathrlap{{}\subseteq G_\mu \times_{G_m} Y,}
			\to[u, "\rho^\tube" swap] 
	\end{equationcd}
	where \( \chi^\tube: G_\mu \times_{G_m} S \to M \) is the tube diffeomorphism associated with \( S \) and \( \rho^\tube: G_\mu \times_{G_m} V \to N \) is defined by \( \rho^\tube(\equivClass{g, y}) = g \cdot \rho^{-1}(y) \).
	For short, we say that \( f \) is locally \( G \)-equivalent to \( \normalFormGroup{G_\mu}{f} \).
\end{defn}
Assume that the \( G \)-equivariant map \( f: M \to N \) can be brought into the equivariant normal form \( (G_m, X, Y, \hat{f}, f_\singularPart) \) by using diffeomorphisms \( \iota_S \) and \( \rho \) as above.
Then \( \tangent_m \iota_S: X \to \TBundle_m S \) identifies \( X \) with \( \TBundle_m S \isomorph \TBundle_m M \slash \LieA{g}_\mu \ldot m \), as \( G_m \)-representations.
Similarly, \( \tangent_\mu \rho: \TBundle_\mu N \to Y \) is an isomorphism of \( G_m \)-representations.
Under these identifications, the abstract spaces \( \ker \) and \( \coker \) occurring in the decomposition of \( X \) and \( Y \) are identified with
\begin{equation}
	\ker \isomorph \ker \tangent_m f \slash \LieA{g}_\mu \ldot m,
	\qquad
	\coker \isomorph \TBundle_\mu N \slash \img \tangent_m f.
\end{equation}
These are the first and second homology groups of the chain complex
\begin{equationcd}[label=eq:normalFormEquivariantMap:deformationChain]
	0 \to[r] & \LieA{g}_\mu \to[r, "\tangent_e \Upsilon_m"] & \TBundle_m M \to[r, "\tangent_m f"] & \TBundle_\mu N \to[r] & 0 \,,
\end{equationcd}
where \( \Upsilon_m: G_\mu \to M \) is the orbit map of the \( G_\mu \)-action at \( m \in M \).

At a fixed point for the action, an equivariant normal form signifies that there exist charts which linearize the actions and at the same time bring \( f \) into a normal form (in the sense of \cref{defn:normalFormMap:bringIntoNormalForm}).
In fact, the general case can always be reduced to constructing an equivariant normal form at a fixed point of the action.
\begin{prop}
	\label{prop:normalFormEquivariantMap:reductionToFixedPoint}
	Let \( f: M \to N \) be an equivariant map between \( G \)-manifolds.
	Choose \( m \in M \) and set \( \mu = f(m) \).
	Assume that the following conditions hold:
	\begin{thmenumerate}
		\item
			The stabilizer subgroup \( G_\mu \) of \( \mu \) is a Lie subgroup of \( G \).
		\item 
			The induced \( G_\mu \)-action on \( M \) admits a slice \( S \) at \( m \).
		\item
			The \( G_m \)-equivariant map \( f^S \equiv \restr{f}{S}: S \to N \) can be brought into an equivariant normal form at \( m \) relative to the induced \( G_m \)-action, \ie, it is locally \( G_m \)-equivalent to a map \( \normalFormGroup{G_m}{f^S}: X \supseteq U \to Y \).
	\end{thmenumerate}
	Then, \( f \) can be brought into an equivariant normal form at \( m \) relative to the \( G \)-action.
	In fact, \( f \) is locally \( G \)-equivalent to
	\begin{equation}
		\normalFormGroup{G_\mu}{f}: G_\mu \times_{G_m} U \to G_\mu \times_{G_m} Y, \qquad \equivClass{g, x} \mapsto \equivClass*{g, \normalFormGroup{G_m}{f^S}(x)}.
		\qedhere
	\end{equation}
\end{prop}
\begin{proof}
	Let \( S \) be a slice at \( m \) for the induced \( G_\mu \)-action on \( M \), with associated tube map \( \chi^\tube: G_\mu \times_{G_m} S \to M \) defined by \( \chi^\tube(\equivClass{g,s}) = g \cdot s \).
	By \( G \)-equivariance of \( f \), the following diagram commutes:
	\begin{equationcd}
		M \to[r, "f"]
			& N
			\\
		G_\mu \times_{G_m} S \to[r, "f_\chi"] \to[u, "\chi^\tube"]
			& N, \to[u, "\id" swap]
	\end{equationcd}
	where the smooth map \( f_\chi \) is defined by \( f_\chi(\equivClass{g, s}) = g \cdot f(s) \) for \( g \in G_\mu \) and \( s \in S \).
	Thus, the map \( f \) decomposes into the \( G_\mu \)-action and the restriction \( f^S: S \to N \) of \( f \) to the slice.
	Note that \( f^S \) is \( G_m \)-equivariant.
	By assumption, \( f^S \) can be brought into a \( G_m \)-equivariant normal form.
	Since \( m \) is a fixed point for the \( G_m \)-action, this means that there exists charts \( \kappa: S \supseteq U' \to U \subseteq X \) at \( m \) and  \( \rho: N \supseteq V' \to V \subseteq Y \) at \( \mu \) intertwining the \( G_m \)-actions on \( S \) and \( N \) with a linear \( G_m \)-actions on \( X \) and \( Y \), respectively.
	Moreover, \( f(U') \subseteq V' \) and the following diagram commutes
	\begin{equationcd}
		\mathllap{S \supseteq {}} U' \to[r, "f^S"] 
			\to[d, "\kappa" swap]
			& V' \mathrlap{{}\subseteq N}
			\to[d, "\rho"] 
			\\
		\mathllap{X \supseteq {}} U
			\to[r, "\normalFormGroup{G_m}{f^S}"]
			& V \mathrlap{{}\subseteq Y,}
	\end{equationcd}
	where \( \normalFormGroup{G_m}{f^S} = \hat{f}_S + {f^S}_\singularPart \) with maps \( \hat{f}^S \) and \( {f^S}_\singularPart \) as in \cref{def:normalFormEquivariantMap:abstractNormalFormEquivariantMap}.
	By possibly shrinking \( S \), we may assume that \( \kappa \) is defined on the whole of \( S \).
	Define \( \rho^\tube: G_\mu \times_{G_m} V \to N \) by \( \rho^\tube(\equivClass{g, y}) = g \cdot \rho^{-1}(y) \).
	Due to \( G_m \)-equivariance of \( f^S, \kappa \) and \( \rho \), the prescription 
	\begin{equation}
		\normalFormGroup{G_\mu}{f} (\equivClass{g, x}) = \equivClass{g, \rho \circ f^S \circ \kappa^{-1}(x)}, \qquad g \in G_\mu, x \in U
	\end{equation}
	defines a smooth map \( \normalFormGroup{G_\mu}{f}: G_\mu \times_{G_m} U \to G_\mu \times_{G_m} V \) fitting into the following commutative diagram:
	\begin{equationcd}
		M \to[r, "f"]
			& N
		\\
		G_\mu \times_{G_m} S \to[r, "f_\chi"] \to[d, "\id_{G_\mu} \times \kappa", swap]  \to[u, "\chi^\tube"]
			& N \to[u, "\id" swap]
		\\
		G_\mu \times_{G_m} U \to[r, "\normalFormGroup{G_\mu}{f}"]
			& G_\mu \times_{G_m} V. \to[u, "\rho^\tube", swap]
	\end{equationcd}
	Hence \( f \) is brought into the equivariant normal form \( (G_m, X, Y, \hat{f}^S, {f^S}_\singularPart) \).
\end{proof}
At this point, we have completely split off the \( G_\mu \)-action and reduced the problem to constructing an equivariant normal form of the map \( f^{S}: S \to N \) at a \emph{fixed point} of the Lie group \( G_m \).
This is a considerable simplification of the general situation we started with.
In the proposition below we show that equivariant normal forms at fixed points can be constructed using the methods developed in \cref{sec:normalFormMap}.
For this we will need the following notion.
\begin{defn}
We say that a \( G \)-action \emphDef{can be linearized} at a fixed point \(  m \in M \) if there exist a \( G \)-invariant open neighborhood \( U' \) of \( m \) in \( M \), a chart \( \rho: M \supseteq U' \to U \subseteq X \) and a smooth linear \( G \)-action on \( X \) such that \( \rho \) is \( G \)-equivariant.
\end{defn}
\begin{prop}
	\label{prop:normalFormEquivariantMap:aroundFixedPoint}
	Let \( G \) be compact Lie group, and let \( f: M \to N \) be an equivariant map between Fréchet \( G \)-manifolds.
	Choose \( m \in M \) and set \( \mu = f(m) \in N \).
	Assume that \( m \) is a fixed point of the \( G \)-action, \ie \( G_m = G \), and that the following conditions hold:
	\begin{thmenumerate}
		\item
			The \( G \)-actions on \( M \) and \( N \) can be linearized at the fixed points \( m \) and \( \mu \), respectively.
		\item
			{
			\providecommand{\creflastconjunction}{}
			\renewcommand{\creflastconjunction}{~or~}
			The hypotheses of either \cref{prop:normalFormMap:banach,prop:normalFormMap:banachTarget,prop:normalFormMap:banachDomain,prop::normalFormMap:tameFrechet,prop::normalFormMap:elliptic,prop::normalFormMap:ellipticGeometric} are satisfied so that \( f \) can be brought into a normal form (not necessarily equivariant).
			}	
	\end{thmenumerate}
	Then \( f \) can be brought into an equivariant normal form at \( m \).
\end{prop}
\begin{proof}
	By assumption, there exist charts \( \tilde{\kappa}: M \supseteq U' \to U \subseteq X \) at \( m \) and \( \tilde{\rho}: N \supseteq V' \to V \subseteq Y \) at \( \mu \) which are \( G \)-equivariant with respect to linear actions of \( G \) on \( X \) and \( Y \), respectively.
	All the mentioned normal forms established in \cref{sec:normalFormMap}, use the chart deformations \( \psi \) and \( \phi \) introduced in the proof of \cref{prop:normalFormMap:banach} to bring \( f \) into an normal form.
	We claim that in the present setting these local diffeomorphisms can be chosen to be \( G \)-equivariant.
	To see this, consider the topological decompositions \( X = \ker T \oplus \coimg T \) and \( Y = \coker T \oplus \img T \), where \( T \) corresponds to \( \tangent_m f \) under the chart identifications.
	Since \( G \) is compact and \( T \) is \( G \)-equivariant, the complements \( \coimg T \) and \( \coker T \) can be chosen to be \( G \)-invariant according to \parencite[Lemma~3.13]{DiezSlice}.
	By direct inspection, we see that the map \( \psi \) defined in~\eqref{eq:normalFormMap:banach:deformDomain} is a composition of \( G \)-equivariant maps and hence is \( G \)-equivariant.
	The map \( \phi \) defined in~\eqref{eq:normalFormMap:banach:deformTarget} and the singular part \( f_\singularPart \) are \( G \)-equivariant, too, for similar reasons.
\end{proof}
Thus, in summary, we arrive at the following equivariant normal form theorem.
\begin{thm}[Equivariant normal form --- General]
	\label{prop:normalFormEquivariantMap:abstract}
	Let \( f: M \to N \) be an equivariant map between Fréchet \( G \)-manifolds.
	Choose \( m \in M \) and set \( \mu = f(m) \).
	Assume that the following conditions hold:
	\begin{thmenumerate}
		\item
			The stabilizer subgroup \( G_\mu \) of \( \mu \) is a Lie subgroup of \( G \).
		\item 
			The induced \( G_\mu \)-action on \( M \) is proper and admits a slice \( S \) at \( m \).
		\item
			The induced \( G_m \)-action on \( N \) can be linearized at \( \mu \).
		\item
			{
			\providecommand{\creflastconjunction}{}
			\renewcommand{\creflastconjunction}{~or~}
			The restriction \( f^S \equiv \restr{f}{S}: S \to N \) of \( f \) to \( S \) can be brought into a normal form at \( m \) using either \cref{prop:normalFormMap:banach,prop:normalFormMap:banachTarget,prop:normalFormMap:banachDomain,prop::normalFormMap:tameFrechet,prop::normalFormMap:elliptic,prop::normalFormMap:ellipticGeometric}.
			}
	\end{thmenumerate}
	Then, \( f \) can be brought into an equivariant normal form at \( m \) relative to the \( G \)-action.
\end{thm}
\begin{proof}
	Since the \( G \)-action on \( M \) is proper, the stabilizer \( G_m \) is compact.
	By the slice property~\iref{i:slice:linearSlice}, the \( G_m \)-action on the slice \( S \) can be linearized at \( m \).
	Hence, \cref{prop:normalFormEquivariantMap:aroundFixedPoint} implies that \( f^S \) can be brought into an equivariant normal form relative to the \( G_m \)-action.
	The claim now follows from \cref{prop:normalFormEquivariantMap:reductionToFixedPoint}.
\end{proof}
\noindent We now comment on the assumptions of \cref{prop:normalFormEquivariantMap:abstract}.
\begin{enumerate}
	\item
		In finite dimensions, every closed subgroup of a Lie group is a Lie subgroup.
		This shows, in particular, that stabilizer subgroups are Lie subgroups.
		However, it is not known, even for Banach Lie group actions, whether the stabilizer is always a Lie subgroup, see \parencite[Problem IX.3.b]{Neeb2006}.
	\item
		According to \textcite{Palais1961}, proper actions of finite-dimensional Lie groups on finite-dimensional manifolds always admit slices.
		In infinite dimensions, this may no longer be true and additional assumptions have to be made.
		We refer to \parencite{DiezSlice,Subramaniam1986} for general slice theorems in infinite dimensions and \parencite{AbbatiCirelliEtAl1989,Ebin1970,CerveraMascaroEtAl1991} for constructions of slices in concrete examples.
	\item
		By Bochner's Linearization Theorem \parencites[Theorem~1]{Bochner1945}, every action of a compact Lie group on a finite-dimensional manifold can be linearized near a fixed point of the action.
		This result generalizes to actions on Banach manifolds, see \parencite[Proposition~3.1.4]{DiezThesis}.
		However, the proof of Bochner's Linearization Theorem relies on the Inverse Function Theorem and so does not generalize to actions on Fréchet or more general locally convex manifolds.
	\item
		The proof of \cref{prop:normalFormEquivariantMap:aroundFixedPoint} shows that one can use every Normal Form Theorem provided it is equivariant with respect to linear actions of a compact group.
\end{enumerate}
 
\begin{remark}(Slice mapping)
	\label{rem:equivariantNormalForm:relationToSliceMapping}
	Recall from the proof of \cref{prop:normalFormEquivariantMap:reductionToFixedPoint} the diffeomorphism \( \kappa: S \to U \subseteq X \) which brings \( f^S \) into a normal form.
	When \( S \) is viewed as a submanifold of \( M \), then the map \( \Psi = \kappa^{-1} \) satisfies \( \Psi(0) = m \), 
	\begin{equation}
		\TBundle_{\Psi(x)} M = \LieA{g} \ldot \Psi(x) + \img \tangent_{x} \Psi 
	\end{equation}
	for all \( x \in U \), and
	\begin{equation}
		\tangent_{m} f \circ \tangent_{0} \Psi = \hat{f} \circ \pr_{\coimg \tangent_m f^S}.
	\end{equation}
	In the context when \( f \) is the momentum map for a symplectic \( G \)-action, a map with these properties is called a slice map in \parencite[Definition~2.1]{ChossatLewisOrtegaEtAl1999}.
	The construction in \parencite[Proposition~2.2]{ChossatLewisOrtegaEtAl1999} of such a slice map makes no use of slices.
	Thus, our approach has the advantage of bringing the \( G \)-action into a normal form simultaneously.
\end{remark}

\subsection[Finite-dimensional target or domain]{Normal form theorem with finite-dimensional target or domain}
In the remainder of this section, we present variations of the Equivariant Normal Form \cref{prop:normalFormEquivariantMap:abstract} using assumptions that are often easier to verify in applications.
Similar to the discussion in \cref{sec:normalFormMap}, we start with the simplest case and then work through various levels of functional-analytic settings, finishing with the tame Nash--Moser category.
All assumptions of \cref{prop:normalFormEquivariantMap:abstract} are automatically verified in finite dimensions.
\begin{thm}[Equivariant normal form --- finite dimensions]
	\label{prop:normalFormEquivariantMap:finiteDim}
	Let \( G \) be a finite-dimensional Lie group and let \( f: M \to N \) be a smooth \( G \)-equivariant map between finite-dimensional \( G \)-manifolds.
	If the \( G \)-action on \( M \) is proper, then \( f \) can be brought into an equivariant normal form at every point.
\end{thm}
\begin{proof}
	Let \( m \in M \) and \( \mu = f(m) \).
	Since the stabilizer \( G_\mu \) of \( \mu \) is a closed subgroup of a finite-dimensional Lie group, it is a Lie subgroup.
	The induced \( G_\mu \)-action on \( M \) is proper and thus the Slice Theorem of \textcite{Palais1961} implies that there exists a slice \( S \) for the \( G_\mu \)-action at \( m \).
	Properness of the action also implies that \( G_m \) is compact and thus Bochner's Linearization Theorem \parencites[Theorem~1]{Bochner1945}[Theorem~2.2.1]{DuistermaatKolk1999} shows that the \( G_m \)-action on \( N \) can be linearized around the fixed point \( \mu \).
	By \cref{prop:normalFormMap:finiteDim}, \( f^S = \restr{f}{S}: S \to N \) can be brought into a normal form at \( m \).
	Hence, all assumptions of \cref{prop:normalFormEquivariantMap:abstract} are verified and so \( f \) can be brought into an equivariant normal form. 
\end{proof}

This result for the finite-dimensional case can be directly generalized to the case where only one of the manifolds involved is finite-dimensional.
\begin{thm}[Equivariant normal form --- finite-dimensional target]
	\label{prop:normalFormEquivariantMap:finiteDimTarget}
	Let \( G \) be a Lie group and let \( f: M \to N \) be a smooth \( G \)-equivariant map between \( G \)-manifolds, where \( M \) is a Fréchet manifold and \( N \) is a finite-dimensional manifold.
	Let \( m \in M \) and \( \mu = f(m) \).
	If the stabilizer subgroup \( G_\mu \) of \( \mu \) is a Lie subgroup of \( G \) and the induced \( G_\mu \)-action on \( M \) is proper and admits a slice at \( m \), then \( f \) can be brought into an equivariant normal form at \( m \).
\end{thm}
\begin{proof}
	The proof is similar to the one of \cref{prop:normalFormEquivariantMap:finiteDim} using \cref{prop:normalFormMap:banachTarget} to bring \( f^S \) into a normal form.
	We leave the details to the reader.
\end{proof}

\begin{thm}[Equivariant normal form --- finite-dimensional domain]
	\label{prop:normalFormEquivariantMap:finiteDimDomain}
	Let \( G \) be a finite-dimensional Lie group and let \( f: M \to N \) be a smooth \( G \)-equivariant map between \( G \)-manifolds, where \( M \) is finite-dimensional.
	Let \( m \in M \) and \( \mu = f(m) \).
	If the \( G \)-action on \( M \) is proper and the induced \( G_m \)-action on \( N \) can be linearized at \( \mu \), then \( f \) can be brought into an equivariant normal form at \( m \).
\end{thm}
\begin{proof}
	The proof is similar to the one of \cref{prop:normalFormEquivariantMap:finiteDim} using \cref{prop:normalFormMap:banachDomain} to bring \( f^S \) into a normal form.
	We leave the details to the reader.
\end{proof}

\subsection[Tame Fréchet and elliptic version]{Tame Fréchet and elliptic normal form theorem}

Let us now come to a version of \cref{prop:normalFormEquivariantMap:abstract} which lives in the tame Fréchet category. 
\begin{thm}[Equivariant normal form --- Tame Fréchet]
	\label{prop:normalFormEquivariantMap:tame}
	Let \( G \) be a tame Fréchet Lie group and let \( f: M \to N \) be an equivariant map between tame Fréchet \( G \)-manifolds.
	Choose \( m \in M \) and set \( \mu = f(m) \).
	Assume that the following conditions hold:
	\begin{thmenumerate}
		\item
			The stabilizer subgroup \( G_\mu \) of \( \mu \) is a tame Fréchet Lie subgroup of \( G \).
		\item 
			The induced \( G_\mu \)-action on \( M \) is proper and admits a tame slice \( S \) at \( m \).
		\item
			The induced \( G_m \)-action on \( N \) can be linearized at \( \mu \).
		\item
			The chain
			\begin{equationcd}[label=eq:normalFormEquivariantMap:tame:chain]
				0 \to[r] 
					& \LieA{g}_\mu
						\to[r]
					& \TBundle_s M
						\to[r, "\tangent_s f"]
					& \TBundle_{f(s)} N
						\to[r]
					& 0
			\end{equationcd}
			of linear maps parametrized by \( s \in S \) is uniformly\footnotemark{} tame regular at \( m \).
			Here, the first map is the Lie algebra action given by \( \xi \mapsto \xi \ldot s \) for \( \xi \in \LieA{g}_\mu \).
			\footnotetext{To be precise, the chain should be viewed as a chain of vector bundle morphisms over \( S \). Similarly to \cref{defn::normalFormMap:uniformRegularBundle}, uniform regularity is then defined relative to vector bundle trivializations by requiring that the associated family of linear chains is uniformly regular in the sense of \cref{defn:normalFormLinearMap:ellipticComplexes:uniformRegular}.}
	\end{thmenumerate}
	Then, \( f \) can be brought into an equivariant normal form at \( m \). 
\end{thm}
\begin{proof}
	Let \( S \) be a tame \( G_\mu \)-slice at \( m \) modeled on the tame Fréchet space \( X \).
	According to \cref{prop:normalFormEquivariantMap:abstract}, we have to show that the map \( f^S: S \to N \) satisfies the hypotheses of \cref{prop::normalFormMap:tameFrechet}.
	For this purpose, consider slice coordinates as in \iref{i::slice:SliceDefLocallyProduct} of \cref{defn:slice:slice}.
	That is, let \( \chi: W \to G_\mu \) be a local section of \( G_\mu \to G_\mu \slash G_m \) defined on an open neighborhood \( W \) of the identity coset \( \equivClass{e} \) such that the map \( \chi^S: W \times S \to M \) defined by \( \chi^S(\equivClass{g}, s) = \chi(\equivClass{g}) \cdot s \) is a diffeomorphism onto its image.
	Clearly, \( \tangent_{\equivClass{e}} \chi: \LieA{g}_\mu \slash \LieA{g}_m \to \LieA{g}_\mu \) is a continuous right inverse of the projection \( \LieA{g}_\mu \to \LieA{g}_\mu \slash \LieA{g}_m \) and thus induces a topological isomorphism \( \LieA{g}_\mu \isomorph (\LieA{g}_\mu \slash \LieA{g}_m) \times \LieA{g}_m \).
	With respect to this decomposition, we write elements \( \xi \in \LieA{g}_\mu \) as pairs \( (\equivClass{\xi}, \xi_{\LieA{g}_m}) \) with \( \equivClass{\xi} \in \LieA{g}_\mu \slash \LieA{g}_m \) and \( \xi_{\LieA{g}_m} \in \LieA{g}_m \).
	Let \( \iota_S: X \supseteq U \to S \) be a chart of \( S \) which linearizes the \( G_m \)-action as in~\iref{i:slice:linearSlice}, and let \( \rho: N \supseteq V' \to V \subseteq Y \) be a chart at \( \mu \) which linearizes the \( G_m \)-action on \( N \).
	We denote the chart representation of \( f^S: S \to N \) by \( f^S_\rho: X \supseteq U \to V \subseteq Y \).
	With respect to the local trivialization induced by \( \chi^S \circ (\id_W \times \iota_S) \) and \( \rho \), the chain~\eqref{eq:normalFormEquivariantMap:tame:chain} yields a family of chains
	\begin{equationcd}[label=eq:normalFormEquivariantMap:tame:chainLocal]
		0 \to[r] 
			& \LieA{g}_\mu
				\to[r, "\Gamma_x"]
			& \LieA{g}_\mu \slash \LieA{g}_m \times X
				\to[r, "\Xi_x"]
			& Y
				\to[r]
			& 0
	\end{equationcd}
	parametrized by \( x \in U \), where the tame smooth families of linear maps \( \Gamma: U \times \LieA{g}_\mu \to \LieA{g}_\mu \slash \LieA{g}_m \times X \) and \( \Xi: U \times \LieA{g}_\mu \slash \LieA{g}_m \times X \to Y \) are defined by 
	\begin{equation}
		 \Gamma_x (\xi) = (\equivClass{\xi}, \xi_{\LieA{g}_m} \ldot x),
	\end{equation}
	for \( \xi \in \LieA{g}_\mu \), and 
	\begin{equation}
		\Xi_x(\equivClass{\xi}, v) = \tangent_{f(\iota_S(x))} \rho \bigl((\tangent_{\equivClass{e}} \chi (\equivClass{\xi})) \ldot f(\iota_S(x))\bigr) + \tangent_x f^S_\rho (v)
	\end{equation}
	for \( \equivClass{\xi} \in \LieA{g}_\mu \slash \LieA{g}_m \) and \( v \in X \).
	Since the chain~\eqref{eq:normalFormEquivariantMap:tame:chain} is uniformly tame regular at \( m \), we can assume that the chain~\eqref{eq:normalFormEquivariantMap:tame:chainLocal} is uniformly tame regular at \( x = 0 \).
	Note that \( \img \Gamma_0 = \LieA{g}_\mu \slash \LieA{g}_m \times \set{0} \).
	Thus, \cref{prop:normalFormLinearMap:ellipticComplexes:uniformRegularEquiv} implies that the family
	\begin{equation}
		\restr{(\Xi_x)}{\set{\equivClass{0}} \times X} = \tangent_x f^S_\rho: X \mapsto Y
	\end{equation}
	of linear maps parametrized by \( x \in U \) is uniformly tame regular at \( 0 \).
	Thus, by \cref{prop::normalFormMap:tameFrechet}, \( f^S \) can be brought into a normal form, which completes the proof.
\end{proof}
\begin{remark}
	\label{rem:normalFormEquivariantMap:weakUsingNearSlices}
	In \cref{sec:normalFormMap}, we have seen that a smooth map can be brought into a (tame) normal form if and only if its differential is uniformly (tame) regular.
	It is perceivable that a similar equivalence holds in the equivariant setting as well.
	Indeed, uniform regularity of the left arm of~\eqref{eq:normalFormEquivariantMap:tame:chain} should imply that the action admits a near-slice in the sense of \parencite[Definition~2.1.6]{Palais1961}, which is a weaker notion of a slice only requiring the slice coordinates~\iref{i::slice:SliceDefLocallyProduct}.
	The above arguments then should show that the restriction of the map to the near-slice can be brought into a normal form.
	Working out the details of this idea is left for future work.

	We emphasize that the additional topological properties distinguishing a slice from a near-slice will become important in the study of moduli spaces in \cref{sec:moduliSpaces}, see \cref{rem:quotientsLevelSets:kuranishiNotNeedsAllSliceProperties}.
	If one is not interested in these topological features (for example, because the focus lies on rigidity/stability phenomena), then only requiring the tame uniform regularity of the deformation complex is a promising approach.
	We refer the reader to \parencite[Section~III.3.1]{Hamilton1982} for work in this direction which is concerned with the case where the deformation complex has trivial homology.
\end{remark}

For the following elliptic version of the Equivariant Normal Form Theorem, the reader might want to recall the notion of a geometric Fréchet manifold from \cref{defn:normalFormMap:geometric}.
In particular, this theorem applies to nonlinear differential operators that are equivariant under groups of diffeomorphisms or gauge transformations.
\begin{thm}[Equivariant normal form --- elliptic]
	\label{prop:normalFormEquivariantMap:elliptic}
	Let \( G \) be a tame Fréchet Lie group and let \( f: M \to N \) be an equivariant map between tame Fréchet \( G \)-manifolds.
	Let \( m \in M \) and \( \mu = f(m) \).
	Assume that the following conditions hold:
	\begin{thmenumerate}
		\item
			The stabilizer subgroup \( G_\mu \) of \( \mu \) is a geometric tame Fréchet Lie subgroup of \( G \).
		\item 
			The induced \( G_\mu \)-action on \( M \) is proper and admits a geometric slice \( S \) at \( m \).
		\item
			The induced \( G_m \)-action on \( N \) can be linearized at \( \mu \).
		\item
			The chain
			\begin{equationcd}[label=eq:normalFormEquivariantMap:elliptic:chain]
				0 \to[r] 
					& \LieA{g}_\mu
						\to[r]
					& \TBundle_s M
						\to[r, "\tangent_s f"]
					& \TBundle_{f(s)} N
						\to[r]
					& 0
			\end{equationcd}
			is a chain of geometric linear maps parametrized by \( s \in S \), which is an elliptic complex at \( m \).
	\end{thmenumerate}
	Then, \( f \) can be brought into an equivariant normal form at \( m \).
\end{thm}
\begin{proof}
	The claim follows directly as a special case of \cref{prop:normalFormEquivariantMap:tame}, because, according to \cref{prop:normalFormLinearMap:familiesEllipticOps:ellipticComplexUniformlyTameRegular}, a tame family of chains of differential operators is uniformly tame regular in a neighborhood of a point at which the chain is an elliptic complex.
\end{proof}

\section{Moduli Spaces}
\label{sec:moduliSpaces}
In abstract terms, a moduli space is a space whose points parametrize isomorphism classes of geometric objects.
Usually, one is mainly interested in a subclass of geometric objects satisfying an additional condition, which is often phrased in the form of a partial differential equation.
In the following, we will consider moduli spaces fitting into the following general setup.
Let \( f: M \to N \) be an equivariant map between \( G \)-manifolds.
For every \( \mu \in N \), set
\begin{equation}
	\check{M}_\mu \equiv f^{-1}(\mu) \slash G_\mu,
\end{equation}
where \( G_\mu \) is the stabilizer subgroup of \( \mu \) under the \( G \)-action on \( N \).
Here, \( M \) is the space of geometric objects, the equation \( f(m) = \mu \) describes the additional properties of these objects one is interested in, and the \( G \)-action implements the notion of equivalence.
Because of the flexible nature of this general setting, many well-known moduli spaces fit into this framework, see \cref{sec:yangMillsASD,sec:seibergWitten,sec:pseudoholomorphicCurves} for applications.
Moreover, symplectic quotients are examples of these abstract moduli spaces, and the general theory developed in the next sections lays the foundation to generalize singular symplectic reduction to infinite dimensions, which will be presented elsewhere \parencite{DiezThesis,DiezSingularReduction}.

In the simplest case, when \( \mu \) is a regular value of \( f \) and \( G_\mu \) acts freely, the space \( \check{M}_\mu \) is a smooth manifold.
However, in general, the moduli space has a complicated geometry with singularities.
In this section, we investigate the local structure of the moduli space \(\check{M}_\mu = f^{-1}(\mu) \slash G_\mu \) under the assumption that \( f \) can be brought into a normal form.
In this case, \( \check{M}_\mu \) can be locally identified with the quotient of the zero set of a smooth map by the linear action of a compact group, \ie, it has the structure of a Kuranishi space.
In \cref{sec:quotientsLevelSets:orbitTypeStratification}, we find additional conditions on the normal form which ensure that \( \check{M}_\mu \) is stratified by orbit types.

\subsection{Kuranishi structures}
\label{sec:quotientsLevelSets:kuranishiStructures}
Kuranishi spaces were introduced by \textcite[Section~1.5]{FukayaOno1999} in their study of the geometry of moduli spaces of pseudoholomorphic curves.
The notion of a Kuranishi structure builds on ideas of \textcite{Kuranishi1965} for the moduli space of complex structures and of \textcite[Section~II.2]{Donaldson1983} and \textcite[Section~6]{Taubes1982} for moduli problems in gauge theory, see also \parencite[Section~4.2.5]{DonaldsonKronheimer1997}.
\begin{defn}
	\label{defn:quotientsLevelSets:kuranishiStructures:kuranishiChart}
	Let \( \mathscr{X} \) be a topological space.
	A \emphDef{Kuranishi chart} at a point \( x \in \mathscr{X} \) is a tuple \( (V, E, F, H, s, \kappa) \) consisting of the following data:
	\begin{enumerate}
		\item
			an open neighborhood \( V \) of \( 0 \) in a locally convex space \( E \),
		\item
			a locally convex space \( F \),
		\item
			a compact Lie group \( H \) acting linearly and continuously on \( V \) and \( F \),
		\item 
			a smooth \( H \)-equivariant map \( s: V \to F \),
		\item
			a homeomorphism \( \kappa \) from \( s^{-1}(0) \slash H \) to a neighborhood of \( x \) in \( \mathscr{X} \).
	\end{enumerate}
	The bundle \( F \times V \to V \) is called the \emphDef{obstruction bundle} and \( s \) is referred to as the \emphDef{obstruction map}.
	If \( E \) and \( F \) are finite-dimensional vector spaces, then the Kuranishi chart is said to be \emphDef{finite-dimensional}.
\end{defn}
Informally speaking, a space admitting Kuranishi charts is locally modeled on the quotient of the zero set of a smooth map by a compact group.
For a finite-dimensional Kuranishi chart \( (V, E, F, H, s, \kappa) \) at \( x \in \mathscr{X} \), the number
\begin{equation}
	\dim E - \dim F - \dim H
\end{equation}
is called the \emphDef{virtual dimension} of \( \mathscr{X} \) (at \( x \)).
This is motivated by the fact that \( s^{-1}(0) \slash H \) forms a manifold of this dimension at regular points of \( s \) for which the \( H \)-action is free.
\begin{remarks}
	\item
		If the map \( s \) vanishes and \( H \) is a finite group acting faithfully on \( E \), then the Kuranishi chart \( (V, E, F, H, s, \kappa) \) reduces to an infinite-dimensional orbifold chart, \ie, \( \mathscr{X} \) is locally modeled on the quotient of \( E \) by a finite group action.
		If, in addition, the \( H \)-action on \( E \) is trivial, then we obtain an ordinary manifold chart on \( \mathscr{X} \).
	\item
		Our notion of a finite-dimensional Kuranishi chart is a generalization of the one proposed by \textcite[Definition~A.1.1]{OhOhtaOnoEtAl2009}.
		There, \( H \) is assumed to be a finite group (acting effectively on \( V \)).
		Finiteness of \( H \) is a natural assumption in the context of moduli spaces of pseudoholomorphic curves, but it is too restrictive in our more general setting.
		Moreover, we do not require \( \mathscr{X} \) to be compact nor to be endowed with a metric as it is done in \parencite{OhOhtaOnoEtAl2009}.
		
		Usually, only finite-dimensional Kuranishi charts are discussed in the literature.
		However, for general moduli spaces, one cannot expect the Kuranishi chart to be finite-dimensional.
		As we will see below in \cref{rem:quotientsLevelSets:fredholmFiniteDimKranishiStructure}, this amounts to requiring that a certain complex is Fredholm.
	\item
		In order to define a Kuranishi structure on \( \mathscr{X} \), in a similar spirit to the smooth structure of a manifold, one needs to introduce coordinate transition maps in order to glue together different Kuranishi charts.
		A variety of definitions of suitable chart transitions are proposed in the literature, each with their own advantages and functorial properties.
		We refer the reader to \parencite[Appendix~A]{Joyce2014} for a review on this matter.
	\item
		\label{rem:quotientsLevelSets:kuranishiStructures:polyfold}
		Recently, \citeauthor{HoferWysockiZehnder2014} have introduced the polyfold framework as a different approach to deal with the analytic and geometric issues occurring in the study of moduli spaces in symplectic field theory, see \parencite{HoferWysockiZehnder2014,HoferWysockiZehnder2017} and references therein.
		We refer to \parencite{Yang2014} for an extensive discussion of the relation of Kuranishi structures and polyfold theory.
		A detailed comparison of the tame Fréchet category with the so-called scale calculus that underlies the polyfold framework can be found in \parencite{Gerstenberger2016}.
\end{remarks}

The equivariant normal form studied in the previous section is a suitable tool to obtain information about the local structure of the moduli space. 
\begin{thm}
	\label{prop:quotientsLevelSets:kuranishiByNormalForm}
	Let \( f: M \to N \) be an equivariant map between \( G \)-manifolds.
	Choose \( \mu \in N \) such that \( f^{-1}(\mu) \) is not empty.
	Assume that the stabilizer subgroup \( G_\mu \) of \( \mu \) is a Lie subgroup of \( G \).
	If \( f \) can be brought into an equivariant normal form at every point of \( f^{-1}(\mu) \), then there exists a Kuranishi chart on \( \check{M}_\mu \equiv f^{-1}(\mu) \slash G_\mu \) at every point.
\end{thm}
\begin{proof}
	Let \( m \in f^{-1}(\mu) \), and let \( (G_m, X, Y, \hat{f}, f_\singularPart) \) be an abstract equivariant normal form of \( f \) at \( m \) in the sense of \cref{def:normalFormEquivariantMap:abstractNormalFormEquivariantMap}, with associated decompositions \( X = \ker \oplus \coimg \) and \( Y = \coker \oplus \img \).
	Consider a slice \( S \) at \( m \) for the \( G_\mu \)-action, diffeomorphic to the domain \( U \subseteq X \) of \( f_\singularPart \) via a \( G_m \)-equivariant diffeomorphism \( \iota_S: X \supseteq U \to S \).
	In the following, we suppress the slice diffeomorphism \( \iota_S \). 
	Denote the tube diffeomorphism associated with \( S \) by \( \chi^\tube: G_\mu \times_{G_m} U \to M \).
	For a \( G_m \)-equivariant chart \( \rho: N \supseteq V' \to V \subseteq Y \) at \( \mu \) define \( \rho^\tube: G_\mu \times_{G_m} V \to N \) by \( \rho^\tube(\equivClass{g, y}) = g \cdot \rho^{-1}(y) \).
	According to \cref{def:normalFormEquivariantMap:normalFormEquivariantMap}, we can choose \( S \) and \( \rho \) so that the following diagram commutes:
	\begin{equationcd}[label=eq:quotientsLevelSets:kuranishiByNormalForm:normalFormDiag]
		M \to[r, "f"] 
			& N
			\\
		\mathllap{G_\mu \times_{G_m} X \supseteq {}} G_\mu \times_{G_m} U
			\to[r, "\normalFormGroup{G_\mu}{f}"]
			\to[u, "\chi^\tube"]
			& G_\mu \times_{G_m} V \mathrlap{{}\subseteq G_\mu \times_{G_m} Y,}
			\to[u, "\rho^\tube" swap] 
	\end{equationcd}
	where \( \normalFormGroup{G_\mu}{f} \) is defined by \( \normalFormGroup{G_\mu}{f}(\equivClass{g, x_1, x_2}) = \equivClass{g, \hat{f}(x_2) + f_\singularPart(x_1, x_2)} \) for \( g \in G_\mu \), \( x_1 \in U \intersect \ker \) and \( x_2 \in U \intersect \coimg \).
	Note that \( (\rho^\tube)^{-1}(\mu) = G_\mu \times_{G_m} \set{0} \), by the definition of \( \rho^\tube \).
	Using the commutative diagram~\eqref{eq:quotientsLevelSets:kuranishiByNormalForm:normalFormDiag}, we hence obtain
	\begin{equation}\begin{split}
		(\chi^\tube)^{-1} \bigl(f^{-1}(\mu)\bigr)
			&= \normalFormGroup{G_\mu}{f}^{-1} \bigl( (\rho^\tube)^{-1}(\mu)\bigr)
			\\
			&= \normalFormGroup{G_\mu}{f}^{-1} \bigl(G_\mu \times_{G_m} \set{0}\bigr)
			\\
			&= G_\mu \times_{G_m} \set*{(x_1, 0) \in U \given f_\singularPart(x_1, 0) = 0}.
	\end{split}\end{equation}
	By \( G_\mu \)-equivariance, the tube diffeomorphism \( \chi^\tube \) thus induces a local homeomorphism of \( \check{M}_\mu = f^{-1}(\mu) \slash G_\mu \) with
	\begin{equation}
		\set*{(x_1, 0) \in U \given f_\singularPart(x_1, 0) = 0} \slash G_m.
	\end{equation}
	This local homeomorphism provides a Kuranishi chart on \( \check{M}_\mu \) with \( E = \ker \), \( V = U \intersect \ker \), \( F = \coker \), \( H = G_m \) and \( s = f_\singularPart (\cdot, 0): V \to F \) in the notation of \cref{defn:quotientsLevelSets:kuranishiStructures:kuranishiChart}.
\end{proof}

\begin{remark}
	\label{rem:quotientsLevelSets:fredholmFiniteDimKranishiStructure}
	In the setting of \cref{prop:quotientsLevelSets:kuranishiByNormalForm}, consider the chain complex
	\begin{equationcd}[label=eq:quotientsLevelSets:fredholmFiniteDimKranishiStructure:chain]
		0 \to[r] & \LieA{g}_\mu \to[r, "\tangent_e \Upsilon_m"] & \TBundle_m M \to[r, "\tangent_m f"] & \TBundle_\mu N \to[r] & 0 \,,
	\end{equationcd}
	where \( \Upsilon_m: G_\mu \to M \) is the orbit map of the \( G_\mu \)-action at \( m \in M \).
	Recall from~\eqref{eq:normalFormEquivariantMap:deformationChain} that the homology groups of this complex are identified with the abstract spaces
	\begin{equation}
		\LieA{g}_m \isomorph \LieA{h},
		\qquad
		\ker \tangent_m f \slash \LieA{g}_\mu \ldot m \isomorph \ker,
		\qquad
		\TBundle_\mu N \slash \img \tangent_m f \isomorph \coker,
	\end{equation}
	occurring in the equivariant normal form.
	For the Kuranishi charts on \( \check{M}_\mu \) constructed in the proof of \cref{prop:quotientsLevelSets:kuranishiByNormalForm}, we have \( H = G_m \), \( E = \ker \) and \( F = \coker \).
	Thus, these Kuranishi charts are finite-dimensional if and only if the complex~\eqref{eq:quotientsLevelSets:fredholmFiniteDimKranishiStructure:chain} is Fredholm.
	In this case, the Euler characteristic of the complex~\eqref{eq:quotientsLevelSets:fredholmFiniteDimKranishiStructure:chain},
	\begin{equation}
		\dim \LieA{g}_m - \dim \ker + \dim \coker,
	\end{equation}
	is (minus) the virtual dimension of \( \check{M}_\mu \).
\end{remark}
\begin{remark}
	\label{rem:quotientsLevelSets:kuranishiNotNeedsAllSliceProperties}
	The proof of \cref{prop:quotientsLevelSets:kuranishiByNormalForm} relies in an essential way on the fact that the tube map \( \chi^\tube: G_\mu \times_{G_m} S \to M \) associated with the slice \( S \) is a diffeomorphism onto an open \( G_\mu \)-invariant neighborhood of the orbit \( G \cdot m \).
	This is also the first instance where we use injectivity of \( \chi^\tube \), \ie, slice property \iref{i::slice:SliceDefOnlyStabNotMoveSlice}.
	A close inspection of the proof of \cref{prop:quotientsLevelSets:kuranishiByNormalForm} shows that for the Kuranishi structure on \( \check{M}_\mu \) one actually needs only the restriction of \( \chi^\tube \) to \( G_\mu \times_{G_m} \bigl(S \intersect f^{-1}(\mu)\bigr) \) to be injective.
	This is equivalent to the following weaker version of \iref{i::slice:SliceDefOnlyStabNotMoveSlice}:
	\begin{enumerate}[leftmargin=7ex]
		\item[(SL2\(_\mu \))]
			Any \( g \in G \) for which \( g \cdot \bigl(S \intersect f^{-1}(\mu)\bigr) \) intersects non-trivially with \( S \intersect f^{-1}(\mu) \) is necessarily an element of the stabilizer \( G_m \).
	\end{enumerate}
	Recall that the slice property \iref{i::slice:SliceDefOnlyStabNotMoveSlice} is usually established using some form of compactness of the group action, \eg, properness.
	Thus, for (SL2\(_\mu \)), one needs such a compactness property only on the solution set \( f^{-1}(\mu) \).
	This observation enables one to use compactness arguments specific to the partial differential equation under study.
	We will encounter such a case for pseudoholomorphic curves where Gromov's Compactness Theorem ensures that (SL2\(_\mu \)) holds although \iref{i::slice:SliceDefOnlyStabNotMoveSlice} is not satisfied in general, see \cref{sec:pseudoholomorphicCurves}.
	
	Moreover, we so far did not make any use of property \iref{i::slice:SliceDefPartialSliceSubmanifold} which will only play a fundamental role for the orbit type stratification discussed in \cref{sec:quotientsLevelSets:orbitTypeStratification} below.
	Thus, we expect that a version of \cref{prop:quotientsLevelSets:kuranishiByNormalForm} also holds in the more general setting, where one only requires the weak equivariant normal form that uses a near-slice instead of a slice, as discussed in \cref{rem:normalFormEquivariantMap:weakUsingNearSlices}, and the weak slice property (SL2\(_\mu \)). 
\end{remark}

Since every equivariant map between finite-dimensional \( G \)-manifolds can be brought into an equivariant normal form according to \cref{prop:normalFormEquivariantMap:finiteDim}, we obtain the following corollary of \cref{prop:quotientsLevelSets:kuranishiByNormalForm}.
\begin{coro}
	\label{prop:quotientsLevelSets:kuranishiFiniteDim}
	Let \( f: M \to N \) be an equivariant map between finite-dimensional \( G \)-manifolds and let \( \mu \in N \).
	If the \( G_\mu \)-action on \( M \) is proper, then \( \check{M}_\mu \equiv f^{-1}(\mu) \slash G_\mu \) admits a Kuranishi chart at every point.
\end{coro}
\begingroup
\makeatletter
\@cref@compressfalse
\makeatother
Of course, the Equivariant Normal Form \cref{prop:normalFormEquivariantMap:finiteDimTarget,prop:normalFormEquivariantMap:finiteDimDomain,prop:normalFormEquivariantMap:tame,prop:normalFormEquivariantMap:elliptic} give similar corollaries in infinite dimensions.
\endgroup

\subsection{Orbit type stratification}
\label{sec:quotientsLevelSets:orbitTypeStratification}

As we have seen, the abstract moduli space \( \check{M}_\mu = f^{-1}(\mu) \slash G_\mu \) admits Kuranishi charts if \( f \) can be brought into an equivariant normal form.
Under additional conditions on the equivariant normal form, \( \check{M}_\mu \) has an even more pleasant structure.
For the following result, the reader might want to recall the notions of orbit type and stratification from \cref{sec:calculus:groupActionsSlices}.
Roughly speaking, a stratification is a decomposition of a space such that the pieces are manifolds and a certain approximation condition holds when one approaches the boundary of a piece.
These properties have a local nature and, in fact, can be ensured by requiring the equivariant normal form to satisfy additional conditions.
\begin{prop}
	\label{prop:quotientsLevelSets:orbitTypeStratification}
	Let \( f: M \to N \) be a smooth equivariant map between \( G \)-manifolds, and let \( \mu \in N \).
	Assume that the stabilizer subgroup \( G_\mu \) is a Lie subgroup of \( G \) and that the induced action of \( G_\mu \) on \( M \) is proper.
	Moreover, assume that \( f \) can be brought into an equivariant normal form \( (G_m, X, Y, \hat{f}, f_\singularPart) \) at every \( m \in f^{-1}(\mu) \) such that the following holds:
	\begin{thmenumerate}
		\item
			(Submanifold property)
			The set
			\begin{equation}
				\set*{(x_1, 0) \in U_{(G_m)} \given f_\singularPart(x_1, 0) = 0} = f_\singularPart^{-1}(0) \intersect {\ker} \intersect X_{(G_m)}
			\end{equation}
			is a submanifold of \( X_{(G_m)} \).
		\item
			(Approximation property)
			For every orbit type \( (K) \leq (G_m) \) of the \( G_\mu \)-action on \( M \), the point \( 0 \) lies in the closure of \( f_\singularPart^{-1}(0) \intersect {\ker} \intersect X_{(K)} \) in \( X \). 
	\end{thmenumerate}
	Then, the partition of \( f^{-1}(\mu) \) into the orbit type subsets \( f^{-1}(\mu) \intersect M_{(H)} \) is a stratification.
	Moreover, the decomposition of \( \check{M}_\mu = f^{-1}(\mu) \slash G_\mu \) into the sets \( \bigl(f^{-1}(\mu) \intersect M_{(H)}\bigr) \slash G_\mu \) is a stratification, too. 
\end{prop}
\begin{proof}
	Let \( m \in f^{-1}(\mu) \) and let \( (G_m, X, Y, \hat{f}, f_\singularPart) \) be an equivariant normal form of \( f \) at \( m \) satisfying the submanifold and approximation properties.
	In the proof of \cref{prop:quotientsLevelSets:kuranishiByNormalForm}, we have seen that the \( G_\mu \)-equivariant tube diffeomorphism \( \chi^\tube \) identifies the level set \( f^{-1}(\mu) \) locally with
	\begin{equation}
		G_\mu \times_{G_m} \set*{(x_1, 0) \in U \given f_\singularPart(x_1, 0) = 0}.
	\end{equation}
	Accordingly, \( f^{-1}(\mu) \intersect S_{(G_m)} \) is identified with the set
	\begin{equation}
		\set*{(x_1, 0) \in U_{(G_m)} \given f_\singularPart(x_1, 0) = 0}.
	\end{equation}
	Since the latter is a submanifold of \( X_{(G_m)} \) by the submanifold property, we conclude that \( f^{-1}(\mu) \intersect S_{(G_m)} \) is a closed submanifold of \( S_{(G_m)} \).
	Moreover, the approximation property entails that, for every \( G_\mu \)-orbit type \( (K) \leq (G_m) \), the point \( m \) lies in the closure of \( f^{-1}(\mu) \intersect S_{(K)} \).
	Thus, the claims follow from \cref{prop:slice:orbitTypeStratificationSubset}.
\end{proof}

\subsection{Application: Moduli space of anti-self-dual connections}
\label{sec:yangMillsASD}

The local structure of the moduli space of anti-self-dual Yang--Mills connections is well understood, see, \eg, \parencite{DonaldsonKronheimer1997}.
We will show how these results can be rederived with relatively small effort using the general framework developed in the previous sections.

Consider a principal \( G \)-bundle \( \pi: P \to M \) over a compact \( n \)-dimensional Riemannian manifold \( M \) with \( G \) being a compact Lie group.
A connection in \( P \) is a \( G \)-equivariant splitting of the tangent bundle \( \TBundle P = \VBundle P \oplus \HBundle P \) into the canonical vertical distribution \( \VBundle P \) and a horizontal distribution \( \HBundle P \).
Equivalently, a connection \( A \) in \( P \) yields a splitting of the Atiyah sequence
\begin{equationcd}
	0
		\to[r]
	& \AdBundle P
		\to[r]
	& \TBundle P \slash G
		\to[r, "\tangent \pi"]
	& \TBundle M
		\to[r]
	& 0.
\end{equationcd}
Accordingly, the space \( \ConnSpace(P) \) of connections on \( P \) is identified with the space of sections of an affine bundle over \( M \), and in this way \( \ConnSpace(P) \) carries a natural tame Fréchet manifold structure modeled on the vector space \( \DiffFormSpace^1(M, \AdBundle P) \), \cf \parencite{AbbatiCirelliEtAl1986}.
The covariant derivative and the curvature of a connection \( A \) are denoted by \( \dif_A \) and \( F_A \), respectively.

The group \( \GauGroup(P) \) of local gauge transformations on \( P \) is a tame Fréchet Lie group, because it is realized as the space of smooth sections of the group bundle \( P \times_G G \), see \parencite{CirelliMania1985} for details.
The natural left action of \( \GauGroup(P) \) on \( \ConnSpace(P) \) is given by
\begin{equation}
	A \mapsto \AdAction_{\lambda} A + \lambda \dif \lambda^{-1},
\end{equation}
for \( \lambda \in \GauGroup(P) \).
This action is proper, see \parencite{Diez2013,RudolphSchmidtEtAl2002}.
Moreover, it admits a slice \( \SectionSpaceAbb{S}_{A_0} \) at every \( A_0 \in \ConnSpace \) given by the Coulomb gauge condition.
That is\footnote{Here, as usual, \( \dif^*_{A} \alpha \defeq (-1)^k \hodgeStar \dif_{A} \hodgeStar \alpha \) for a \( k \)-form \( \alpha \).},
\begin{equation}
	\SectionSpaceAbb{S}_{A_0} \defeq \set{A \in \SectionSpaceAbb{U} \given \dif_{A_0}^* (A - A_0) = 0},
\end{equation}
where \( \SectionSpaceAbb{U} \) is a suitable open neighborhood of \( A_0 \) in \( \ConnSpace(P) \).
One uses the Nash--Moser Inverse Function Theorem to show that \( \SectionSpaceAbb{S}_{A_0} \) is a slice.
The details can be found in \parencite{Diez2013,AbbatiCirelliEtAl1989}.
In the Banach context, the orbit type decomposition of \( \ConnSpace(P) \) has been extensively studied in \parencite{KondrackiRogulski1986}, see also \parencite{RudolphSchmidtEtAl2002}.
The proof that the decomposition satisfies the frontier condition \parencite[Theorem~4.3.5]{KondrackiRogulski1986} carries over to our Fréchet setting without major changes.
As a consequence of \cref{prop:slice:orbitTypeStratificationSubset}, the decomposition of \( \ConnSpace(P) \) into gauge orbit types is a stratification.
Finally, the gauge orbit types are in bijection with the set of isomorphism classes of holonomy-induced Howe subbundles of \( P \), \cf \parencite[Theorem~3.3]{RudolphSchmidtEtAl2002}.
Using this observation, the classification of orbit types for all classical groups has been accomplished in \parencite{RudolphSchmidtEtAl2002b,RudolphSchmidtEtAl2002a,RudolphSchmidtEtAl2002,HertschRudolphSchmidt2010,HertschRudolphSchmidt2011}.

Coming to the anti-self-dual equation, we specialize to the case where \( M \) is an oriented compact Riemannian manifold of dimension \( 4 \) and \( G \) is a compact, semisimple Lie group.
On a \( 4 \)-dimensional manifold, the Hodge star operator~\( \hodgeStar \) associated with the Riemannian metric satisfies \( \hodgeStar \hodgeStar = \id \) on \( 2 \)-forms and thus determines a decomposition
\begin{equation}
	\DiffFormSpace^2(M) = \DiffFormSpace^2_+(M) \oplus \DiffFormSpace^2_-(M)
\end{equation}
of the space of \( 2 \)-forms into the \( \pm 1 \)-eigenspaces.
A similar decomposition holds for vector-valued \( 2 \)-forms and, in particular, the curvature \( F_A \) of a connection \( A \in \ConnSpace(P) \) can be written as
\begin{equation}
	F_A = F_A^+ + F_A^-
\end{equation}
with \( F_A^\pm \in \DiffFormSpace^2_\pm (M, \AdBundle P) \).
A connection \( A \) with \( F_A^+ = 0 \) is called \emphDef{anti-self-dual (ASD)}.
The Bianchi identity implies that an ASD connection satisfies the Yang--Mills equation.
The self-dual part of the curvature gives a smooth map
\begin{equation}
	F^+: \ConnSpace(P) \to \DiffFormSpace^2_+ (M, \AdBundle P),
\end{equation}
which is equivariant with respect to the natural actions of \( \GauGroup(P) \).
The \emphDef{moduli space of anti-self-dual connections} is, by definition, the space
\begin{equation}
	\SectionSpaceAbb{A} = (F^+)^{-1}(0) \slash \GauGroup(P)
\end{equation}
of anti-self-dual connections on \( P \) modulo gauge equivalence.
This clearly fits into the general framework considered in the previous sections.
Let us verify the conditions of \cref{prop:normalFormEquivariantMap:elliptic}:
\begin{enumerate}
	\item 
		The stabilizer of \( \mu = 0 \in \DiffFormSpace^2_+ (M, \AdBundle P) \) is the whole group \( \GauGroup(P) \), which is a geometric tame Fréchet Lie group with Lie algebra \( \GauAlgebra(P) = \sSectionSpace(\AdBundle P) \).
	\item
		The natural action of \( \GauGroup(P) \) on \( \ConnSpace(P) \) is proper and admits slices as discussed above.
	\item
		The action of \( \GauGroup(P) \) on \( \DiffFormSpace_+^2(M, \AdBundle P) \) is clearly linear.
	\item
		Let \( A \) be an ASD connection and let \( \SectionSpaceAbb{S} \) be a slice at \( A \). 
		The chain~\eqref{eq:normalFormEquivariantMap:tame:chain} here takes the form
		\begin{equationcd}[label=eq:yangMillsASD:chain]
			0 \to[r] 
				& \DiffFormSpace^0(M, \AdBundle P) \to[r, "\dif_B"]
				& \DiffFormSpace^1(M, \AdBundle P) \to[r, "\dif^+_B"]
				& \DiffFormSpace^2_+(M, \AdBundle P) \to[r]
				& 0,
		\end{equationcd}
		where \( B \in \SectionSpaceAbb{S} \) and \( \dif^+_B \) denotes the self-dual part of the covariant derivative \( \dif_B \).
		This chain is clearly a chain of linear differential operators tamely parametrized by the connection \( B \).
		The ASD condition for \( A \) asserts that \( \dif^+_A \circ \dif_A = 0 \) and so, at \(  B = A \), the chain~\eqref{eq:yangMillsASD:chain} is a complex, called the \emphDef{Yang--Mills complex}.
		This complex has been studied first by \textcite{AtiyahHitchinSinger1978}.
		Ellipticity of the Yang--Mills complex is well known and follows from a straightforward computation in linear algebra, see \eg \parencite[Lemma~6.5.2]{RudolphSchmidt2014}.
		Moreover, applying the Atiyah--Singer Index Theorem one finds the following for its Euler characteristic:
		\begin{equation}
			- 2 \pontryaginClass_1 (\AdBundle P) + \frac{1}{2} (\chi_M - \sigma_M) \cdot \dim G ,
		\end{equation}
		where \( \pontryaginClass_1 (\AdBundle P) \) is the Pontryagin index of the adjoint bundle, \( \chi_M \) is the Euler number of \( M \) and \( \sigma_M \) is the signature of \( M \), see \parencite[p.~446]{AtiyahHitchinSinger1978} and \parencite[Lemma~6.5.5]{RudolphSchmidt2014} for a detailed proof.
\end{enumerate}
Hence, by \cref{prop:normalFormEquivariantMap:elliptic}, the map \( F^+ \) can be brought into an equivariant normal form at every ASD connection \( A \in \ConnSpace(P) \).
Moreover, as a consequence of \cref{prop:quotientsLevelSets:kuranishiByNormalForm,rem:quotientsLevelSets:fredholmFiniteDimKranishiStructure}, we obtain the following description of the local geometry of the moduli space of anti-self-dual connections.

\begin{thm}
	\label{prop:yangMillsASD:kuranishi}
	Let \( P \) be a principal \( G \)-bundle with a compact, semisimple structure group \( G \) over a \( 4 \)-dimensional compact Riemannian manifold \( M \).
	Then, the moduli space \( \SectionSpaceAbb{A} \) of anti-self-dual connections on \( P \) admits a finite-dimensional Kuranishi chart at every point.
	Moreover, the virtual dimension of \( \SectionSpaceAbb{A} \) is given by
	\begin{equation}
			2 \pontryaginClass_1 (\AdBundle P) - \frac{1}{2} (\chi_M - \sigma_M) \cdot \dim G.
			\qedhere
		\end{equation}
\end{thm}
Let us describe the constructed Kuranishi charts on \( \SectionSpaceAbb{A} \) in more detail.
For this purpose, let \( A \in \ConnSpace(P) \) be an ASD connection.
According to \cref{rem:quotientsLevelSets:fredholmFiniteDimKranishiStructure}, the linear spaces occurring in definition of a Kuranishi chart at \( \equivClass{A} \in \SectionSpaceAbb{A} \) are given by the cohomology groups
\begin{equation}\begin{split}
	\SectionSpaceAbb{E} &= \deRCohomology^{1, +}_A(M, \AdBundle P) \equiv \ker \dif^+_A \slash \img \dif_A,
	\\
	\SectionSpaceAbb{F} &= \deRCohomology^{2, +}_A(M, \AdBundle P) \equiv \DiffFormSpace^2_{+}(M, \AdBundle P) \slash \img \dif^+_A.
\end{split}\end{equation}
These spaces are finite-dimensional, because the Yang--Mills complex is elliptic.
Moreover, they are endowed with a natural linear action of the compact stabilizer subgroup \( \GauGroup_A(P) \) of \( A \).
Thus, the obstruction map is a \( \GauGroup_A(P) \)-equivariant map
\begin{equation}
	\SectionMapAbb{s}: \deRCohomology^{1, +}_A(M, \AdBundle P) \supseteq \SectionSpaceAbb{V} \to \deRCohomology^{2, +}_A(M, \AdBundle P),
\end{equation}
where \( \SectionSpaceAbb{V} \) is a \( \GauGroup_A(P) \)-invariant, open neighborhood of \( 0 \) in \( \deRCohomology^{1, +}_A(M, \AdBundle P) \).
Finally, the moduli space \( \SectionSpaceAbb{A} \) in a neighborhood of \( \equivClass{A} \) is modeled on the quotient \( \SectionMapAbb{s}^{-1}(0) \slash \GauGroup_A(P) \).
In this way, we recover in our framework the well-known result \parencite[Proposition~4.2.23]{DonaldsonKronheimer1997} concerning the local structure of \( \SectionSpaceAbb{A} \).

For completeness, we include some remarks on how to get further insights into the geometry of \( \SectionSpaceAbb{A} \).
This requires a more precise control of the obstruction map \( \SectionMapAbb{s} \), which is rather difficult to obtain in full generality.
However, in concrete examples, one can often find conditions which ensure that \( \SectionMapAbb{s} \) vanishes.
For example, if \( M \) is a compact self-dual Riemannian manifold with positive scalar curvature, then it can be shown using the Weitzenböck formula that \( \deRCohomology^{2, +}_A(M, \AdBundle P) \) is trivial for every irreducible ASD connection \( A \).
Thus, in this case, the moduli space of irreducible anti-self-dual connections is either empty or a smooth manifold.
This important result was originally obtained by \textcite[Theorem~6.1]{AtiyahHitchinSinger1978}.

For the remainder, we specialize to \( G = \SUGroup(2) \).
This setting was used by \textcite{Donaldson1983} to gain astounding insights into the topology and geometry of \( 4 \)-manifolds.
For \( G = \SUGroup(2) \), we have \( \pontryaginClass_1(\AdBundle P) = 4k \), where \( k \) is the so-called instanton number.
Let us restrict our attention to the case \( k = 1 \).
Moreover, let us assume that \( M \) is simply connected and that its intersection form is positive definite.
These assumptions imply \( \chi_M - \sigma_M = 1 \) so that the virtual dimension of \( \SectionSpaceAbb{A} \) is \( 8 - 3 = 5 \). 
For \( G = \SUGroup(2) \), a reducible connection \( A \) has a stabilizer group conjugate to \( \UGroup(1) \) or to \( \Z_2 \).
Connections with a discrete stabilizer subgroup are flat as a consequence of the Ambrose--Singer Theorem and thus only connections with a stabilizer subgroup conjugate to \( \UGroup(1) \) are of interest for the geometry of \( \SectionSpaceAbb{A} \).
One can show that there are only finitely many gauge-equivalence classes of ASD connections which are reducible to \( \UGroup(1) \).
The Kuranishi charts constructed above determine the structure of \( \SectionSpaceAbb{A} \) in a neighborhood of a singular point.
Let \( A \) be an ASD connection that is reducible to \( \UGroup(1) \).
In the present setting, a straightforward application of the representation theory of \( \UGroup(1) \) yields the following isomorphisms of \( \GauGroup_A(P) \)-representation spaces
\begin{equation}
	\deRCohomology^{1, +}_A(M, \AdBundle P) \isomorph \C^p,
	\qquad
	\deRCohomology^{2, +}_A(M, \AdBundle P) \isomorph \C^q
\end{equation}
for some \( p \) and \( q \) satisfying \( p + q = 3 \), where \( \UGroup(1) \) acts in the usual way on \( \C^p \) and \( \C^q \), see \parencite[Proposition~4.9]{FreedUhlenbeck1984}.
Thus, if \( \deRCohomology^{2, +}_A(M, \AdBundle P) \) is trivial, then, in a neighborhood of the singular point \( \equivClass{A} \), the moduli space \( \SectionSpaceAbb{A} \) is identified with the cone \( \C^3 \slash \UGroup(1) \) over \( \CP^2 \).
Thus, combining \cref{prop:yangMillsASD:kuranishi} with these additional insights, one recovers the important result of \textcite{Donaldson1983}.
The case of a non-trivial cohomology group \( \deRCohomology^{2, +}_A(M, \AdBundle P) \) is more complicated and a perturbation of the metric on \( M \) is required, see \parencite[Theorem~4.19]{FreedUhlenbeck1984}.

The concrete description of the singularities shows that every singular connection in \( \SectionSpaceAbb{A} \) can be approximated by a sequence of irreducible ASD connections.
This implies that \( \SectionSpaceAbb{A} \) is stratified by orbit types. 
In fact, for all examples known to us, the moduli space of anti-self-dual connections turns out to be stratified by orbit types.
We do not know whether this holds true in general.

\subsection{Application: Seiberg--Witten moduli space}
\label{sec:seibergWitten}

This section is concerned with the Seiberg--Witten moduli space first studied in \parencite{SeibergWitten1994,SeibergWitten1994a}.
We show how our general framework recovers the known local structure of this moduli space.
Our presentation follows the textbook \parencite{Nicolaescu2000} in conventions and notation. 

For the convenience of the reader, let us start by recalling basic notions from spin geometry.
For \( n \geq 3 \) even, denote the connected double cover of \( \SOGroup(n) \) by \( \SpinGroup(n) \), which is a principal bundle \( \pr: \SpinGroup(n) \to \SOGroup(n) \) with structure group \( \fundamentalGroup_1 \SOGroup(n) \isomorph \Z_2 \).
We can thus form \( \SpinCGroup(n) = \SpinGroup(n) \times_{\Z_2} \UGroup(1) \), where \( \Z_2 \) acts by multiplication with \( -1 \) on \( \UGroup(1) \).
The irreducible half-spin representations\footnotemark{} \( \gamma_\pm: \SpinGroup(n) \to \UGroup (\SpinRep^\pm_n)  \) extend to representations \( \gamma_{\pm}: \SpinCGroup(n) \to \UGroup(\SpinRep^\pm_n) \) of \( \SpinCGroup(n) \) defined by \( \gamma_\pm(\equivClass{g, z}) (\psi) = z \cdot \gamma(g) (\psi) \).
\footnotetext{For \( n = 4 \), we have \( \SpinRep^\pm_4 \isomorph \C^2 \) and the spin representation \( \gamma_\pm \) is given by
\begin{equation}
	\gamma_\pm: \SpinGroup(4) \to \UGroup(2), \qquad (h_+, h_-) \mapsto h_\pm
\end{equation}
under the isomorphism \( \SpinGroup(4) \isomorph \SUGroup(2) \times \SUGroup(2) \).
}
Moreover, the group homomorphism \( \lambda: \SpinCGroup(n) \to \UGroup(1) \) given by \( \lambda(\equivClass{g, z}) = z^2 \) fits into the following commutative diagram
\begin{equationcd}[label={eq:seibergWitten:spinCDiag}]
	& \UGroup(1)
		\to[d, "\iota_2"]
		\to[dr, "z \mapsto z^2"]
	&
	\\
	  \SpinGroup(n)
	  	\to[r, "\iota_1"]
	  	\to[dr, "\pr", swap]
	& \SpinCGroup(n)
		\to[d, "\pr"]
		\to[r, "\lambda"]
	& \UGroup(1).
	\\
	& \SOGroup(n)
	&
\end{equationcd}

A \( \SpinCGroup \)-structure on an oriented Riemannian manifold \( (M, g) \) is a principal \( \SpinCGroup(n) \)-bundle \( \SpinCBundle M \to M \) together with a vertical principal bundle morphism to the \( \SOGroup(n) \)-frame bundle \( \FrameBundle M \to M \):
\begin{equationcd}[tikz={column sep=large}]
	\SpinCBundle M
		\to[d, "\SpinCGroup(n)"{name=U}]
		\to[r]
	&
	\FrameBundle M
		\to[d, "\SOGroup(n)"{name=D, pos=0.58}, swap]
	\\
	M
		\to[r, "\id_M", swap]
	&
	M.
	\arrow[rightarrow, from=U, to=D]
\end{equationcd}
The unitary spin representations \( \gamma_\pm \) give rise to Hermitian vector bundles
\begin{equation}
	\SpinorBundle^\pm M = \SpinCBundle M \times_{\gamma_\pm} \SpinRep_n^\pm,
\end{equation}
whose sections are called spinors.
The Clifford multiplication \( \R^n \times \SpinRep_n^\pm \to \SpinRep_n^\mp \) is equivariant and thus yields a bundle map \( \cl: \TBundle M \times_M \SpinorBundle^\pm M \to \SpinorBundle^\mp M \).
With a slight abuse of notation, we consider \( \cl \) to be defined on the cotangent bundle as well under the isomorphism \( g^\sharp: \CotBundle M \to \TBundle M \).
Moreover, let
\begin{equation}
	P = \SpinCBundle M \times_{\lambda} \UGroup(1)
\end{equation}
be the principal \( \UGroup(1) \)-bundle associated with the group homomorphism \( \lambda \).
As the maps \( (\pr, \lambda): \SpinCGroup(n) \to \SOGroup(n) \times \UGroup(1) \) induce an isomorphism on the level of Lie algebras, every connection \( A \) on \( P \) lifts together with the Levi-Civita connection (seen as a connection on \( \FrameBundle M \)) to a connection on \( \SpinCBundle M \).
The induced covariant derivative on \( \SpinorBundle^\pm M \) is denoted by \( \nabla_A \).
Given a connection \( A \in \ConnSpace(P) \), the Dirac operator \( \Dif_A \) is the composition
\begin{equationcd}[tikz={column sep=scriptsize}]
	\sSectionSpace(\SpinorBundle^+ M)
		\to[r, "\nabla_A"] 
	& \sSectionSpace(\CotBundle M \tensorProd \SpinorBundle^+ M) 
		\to[r, "\cl"]
	& \sSectionSpace(\SpinorBundle^- M).
\end{equationcd}

In the following, \( M \) is a closed oriented Riemannian manifold of dimension \( 4 \) endowed with a \( \SpinCGroup \)-structure.
The Seiberg--Witten equations for a connection \( A \in \ConnSpace(P) \) and a spinor \( \psi \in \sSectionSpace(\SpinorBundle^+ M) \) are
\begin{subequations}
\label{eq:seibergWitten:seibergWitten}
\begin{align}
	\Dif_A \psi &= 0,
	\\
	F_A^+ &= q(\psi) + \mu
\end{align}
\end{subequations}
where \( F_A^+ \) denotes the self-dual component of the curvature \( F_A \in \DiffFormSpace^2(M, \I \R) \) of \( A \) and \( \mu \in \DiffFormSpace^2_+(M, \I \R) \) is a given self-dual \( 2 \)-form.
Moreover, the bundle map \( q: \SpinorBundle^+ M \to \I \ExtBundle^2_+ \CotBundle M \) is induced by the quadratic form \( q: \SpinRep_4^+ \to \I \ExtBundle^2_+ \R^4 \) dual\footnotemark{} to the extended Clifford multiplication \( \cl_+: \ExtBundle^2_+ \R^4 \tensorProd \C \to \End \SpinRep_4^+ \), \ie \( q(\psi) = \cl_+^\T (\psi \tensorProd \psi^\cT) \).
\footnotetext{Alternatively, \( q \) is the momentum map for the natural \( \SUGroup(2) \)-action on \( \SpinRep_4^+ \isomorph \C^2 \) under the identification \( \ExtBundle^2_+ \R^4 \isomorph \R^3 \).}
The commutative diagram~\eqref{eq:seibergWitten:spinCDiag} gives rise to the following group homomorphisms of infinite dimensional Lie groups:
\begin{equationcd}
	  \sFunctionSpace(M, \UGroup(1))
		\to[d, "\iota_2"]
		\to[dr, "\chi \mapsto \chi^2"]
	&
	\\
	  \GauGroup(\SpinCBundle M)
		\to[d]
		\to[r, "\lambda"]
	& \GauGroup(P),
	\\
	  \GauGroup(\FrameBundle M)
	&
\end{equationcd}
where \( \GauGroup(P) \) is the group of gauge transformations on \( P \), \cf \cref{sec:yangMillsASD} for its Fréchet Lie group structure. 
In particular, the current group \( \SectionSpaceAbb{G} \defeq \sFunctionSpace(M, \UGroup(1)) \) acts via gauge transformations on \( \ConnSpace(P) \times \sSectionSpace(\SpinorBundle^+ M) \):
\begin{equation}
	\label{eq:seibergWitten:actionDomain}
	\chi \cdot (A, \psi) = (A - 2 \chi^{-1} \dif \chi, \chi \, \psi).
\end{equation}
To formulate the Seiberg--Witten equations in the general setting discussed in the previous sections, define the map
\begin{equation}\begin{split}
 	\SectionMapAbb{SW}: \ConnSpace(P) \times \sSectionSpace(\SpinorBundle^+ M) &\to \DiffFormSpace^2_+(M, \I \R) \times \sSectionSpace(\SpinorBundle^- M),
 	\\
 	(A, \psi) &\mapsto \bigl(F_A^+ - q(\psi), \Dif_A \psi\bigr).
\end{split}\end{equation}
A straightforward calculation (see, \eg, \parencite[Proposition~2.1.9]{Nicolaescu2000}) shows that \( \SectionMapAbb{SW} \) is equivariant with respect to the following \( \SectionSpaceAbb{G} \)-action on \( \DiffFormSpace^2_+(M, \I \R) \times \sSectionSpace(\SpinorBundle^- M) \):
\begin{equation}
	\label{eq:seibergWitten:actionTarget}
	\chi \cdot (\alpha, \varphi) \mapsto (\alpha, \chi \, \varphi).
\end{equation}
For \( \mu \in \DiffFormSpace^2_+(M, \I \R) \), the Seiberg--Witten moduli space \( \SectionSpaceAbb{M}_\mu \) is defined as the space of solutions \( (A, \psi) \) of the Seiberg--Witten equations~\eqref{eq:seibergWitten:seibergWitten} modulo \( \SectionSpaceAbb{G} \), \ie
\begin{equation}
	\SectionSpaceAbb{M}_\mu = \SectionMapAbb{SW}^{-1}(\mu, 0) \slash \SectionSpaceAbb{G}.
\end{equation}

Let us verify that the assumptions of \cref{prop:normalFormEquivariantMap:elliptic} are met for the model under consideration:
\begin{enumerate}
	\item 
		The stabilizer of \( (\mu, 0) \in \DiffFormSpace^2_+(M, \I \R) \times \sSectionSpace(\SpinorBundle^- M) \) under the action~\eqref{eq:seibergWitten:actionTarget} is the whole group \( \SectionSpaceAbb{G} \), which is a geometric tame Fréchet Lie group with Lie algebra \( \sFunctionSpace(M, \I \R) \).
	\item
		The natural action of \( \GauGroup(P) \) on \( \ConnSpace(P) \) is proper and admits a slice at every point as discussed in \cref{sec:yangMillsASD}.
		Thus, also the action of \( \SectionSpaceAbb{G} \) on \( \ConnSpace(P) \) admits a slice \( \SectionSpaceAbb{S}_A \) at every \( A \in \ConnSpace(P) \).
		The stabilizer \( \SectionSpaceAbb{G}_A \) of \( A \) is \( \UGroup(1) \), viewed as constant maps.
		Moreover, the linear continuous action \( (z, \psi) \mapsto z \, \psi \) of the compact group \( \SectionSpaceAbb{G}_A \isomorph \UGroup(1) \) on the Fréchet space \( \sSectionSpace(\SpinorBundle^+ M) \) admits a slice \( \SectionSpaceAbb{S}_\psi \) at every \( \psi \in \sSectionSpace(\SpinorBundle^+ M) \) according to \parencite[Theorem~3.15]{DiezSlice}.
		By \parencite[Proposition~3.29]{DiezSlice}, the \( \SectionSpaceAbb{G} \)-action~\eqref{eq:seibergWitten:actionDomain} on \( \ConnSpace(P) \times \sSectionSpace(\SpinorBundle^+ M) \) has a slice \( \SectionSpaceAbb{S}_{A, \psi} = \SectionSpaceAbb{S}_A \times \SectionSpaceAbb{S}_\psi \) at \( (A, \psi) \).
		The existence of slices in the Banach setting is well known, see \eg \parencite[Proposition~2.2.7]{Nicolaescu2000}.
	\item
		The action~\eqref{eq:seibergWitten:actionTarget} of \( \SectionSpaceAbb{G} \) on \( \sSectionSpace(\SpinorBundle^- M) \times \DiffFormSpace^2_+(M, \I \R) \) is clearly linear.
	\item
		The linearization of \( \SectionMapAbb{SW} \) at a point \( (A, \psi) \) is given by
		\begin{equation}\begin{split}
			\tangent_{(A, \psi)} \SectionMapAbb{SW}: &\DiffFormSpace^1(M, \I \R) \times \sSectionSpace(\SpinorBundle^+ M) \to \DiffFormSpace^2_+(M, \I \R) \times \sSectionSpace(\SpinorBundle^- M),
			\\
			&(\alpha, \varphi) \mapsto \bigl(\dif^+ \alpha - \dot{q}(\psi, \varphi), \Dif_A \varphi + \frac{1}{2} \cl(\alpha, \psi)\bigr),
		\end{split}\end{equation}
		where \( \dot{q}(\psi, \varphi) = \cl_+^\T (\psi \tensorProd \varphi^\cT + \varphi \tensorProd \psi^\cT) \) and \( \dif^+ \) denotes the self-dual part of the exterior derivative, see \parencite[p.~125]{Nicolaescu2000}.
		Moreover, the action of the Lie algebra \( \sFunctionSpace(M, \R) \) induced by the action~\eqref{eq:seibergWitten:actionDomain} is given by
		\begin{equation}
			\xi \ldot (A, \psi) = (- 2 \dif \xi, \xi \, \psi) \equiv \tau_{(A, \psi)} (\xi),
		\end{equation}
		for \( \xi \in \sFunctionSpace(M, \I \R) \).

		For a solution \( (A_0, \psi_0) \in \SectionMapAbb{SW}^{-1}(\mu, 0) \) of the Seiberg--Witten equations, the chain~\eqref{eq:normalFormEquivariantMap:tame:chain} here takes the form
		\begin{equationcd}[label=eq:seibergWitten:deformationChain, tikz={column sep=2.5em}]
			\DiffFormSpace^0(M, \I \R) \to[r, "{\tau_{(A, \psi)}}"]
				& \DiffFormSpace^1(M, \I \R) \times \sSectionSpace(\SpinorBundle^+ M) \to[r, "\tangent_{(A, \psi)} \SectionMapAbb{SW}"]
				& \DiffFormSpace^2_+(M, \I \R) \times \sSectionSpace(\SpinorBundle^- M),
		\end{equationcd}
		where \( A \in \SectionSpaceAbb{S}_{A_0} \) and \( \psi \in \SectionSpaceAbb{S}_{\psi_0} \).
		This chain is clearly a chain of linear differential operators tamely parametrized by \( (A, \psi) \).
		At \(  A = A_0 \), the chain~\eqref{eq:seibergWitten:deformationChain} is a complex that, after ignoring the zeroth order contributions (which are compact operators and thus do not change ellipticity or the index), is given by
		\begin{equationcd}[label=eq:seibergWitten:deformationComplexReduced, tikz={column sep=2.8em}]
			\DiffFormSpace^0(M, \I \R) \to[r, "{(\dif, 0)}"]
				& \DiffFormSpace^1(M, \I \R) \times \sSectionSpace(\SpinorBundle^+ M) \to[r, "{(\dif^+, \Dif_{A_0}\,) }"]
				& \DiffFormSpace^2_+(M, \I \R) \times \sSectionSpace(\SpinorBundle^- M).
		\end{equationcd}
		It is known (\eg, \parencite[Lemma~2.2.10]{Nicolaescu2000}) that this complex is elliptic and that its Euler characteristic is given by
		\begin{equation}
			\frac{1}{4} \bigl(2 \, \chi_M + 3 \, \sigma_M - \chernClass_1(P)^2\bigr), 
		\end{equation}
		where \( \chi_M \) is the Euler number and \( \sigma_M \) is the signature of \( M \).
\end{enumerate}
Hence, by \cref{prop:normalFormEquivariantMap:elliptic}, the map \( \SectionMapAbb{SW} \) can be brought into an equivariant normal form at every solution \( (A, \psi) \) of the Seiberg--Witten equations.
Moreover, as a consequence of \cref{prop:quotientsLevelSets:kuranishiByNormalForm,rem:quotientsLevelSets:fredholmFiniteDimKranishiStructure}, we obtain the following description of the local geometry of the moduli space \( \SectionSpaceAbb{M}_\mu \).

\begin{thm}
	Let \( M \) be a closed oriented Riemannian manifold of dimension \( 4 \) endowed with a \( \SpinCGroup \)-structure.
	For every \( \mu \in \DiffFormSpace^2_+(M, \I \R) \), the Seiberg--Witten moduli space \( \SectionSpaceAbb{M}_\mu \) admits a finite-dimensional Kuranishi chart at every point.
	Moreover, the virtual dimension of \( \SectionSpaceAbb{M}_\mu \) is given by
	\begin{equation}
			\frac{1}{4} \bigl(\chernClass_1(P)^2 - 2 \, \chi_M - 3 \, \sigma_M\bigr).
			\qedhere
		\end{equation}
\end{thm}
Let us describe the constructed Kuranishi charts on \( \SectionSpaceAbb{M}_\mu \) in more detail.
For this purpose, let \( (A, \psi) \) be a solution of the Seiberg--Witten equations.
According to \cref{rem:quotientsLevelSets:fredholmFiniteDimKranishiStructure}, the linear spaces occurring in the definition of a Kuranishi chart at \( \equivClass{A, \psi} \in \SectionSpaceAbb{M}_\mu \) are given by the finite-dimensional cohomology groups
\begin{equation}\begin{split}
	\SectionSpaceAbb{E} &= \deRCohomology^{1}_{A, \psi}(M) \equiv \ker \tangent_{(A, \psi)} \SectionMapAbb{SW} \slash \img \tau_{(A, \psi)},
	\\
	\SectionSpaceAbb{F} &= \deRCohomology^{2}_{A, \psi}(M) \equiv \DiffFormSpace^2_+(M, \I \R) \times \sSectionSpace(\SpinorBundle^- M) \slash \img \tangent_{(A, \psi)} \SectionMapAbb{SW}.
\end{split}\end{equation}
The obstruction map is a \( \SectionSpaceAbb{G}_{A, \psi} \)-equivariant map
\begin{equation}
	\SectionMapAbb{s}: \deRCohomology^{1}_{A, \psi}(M) \supseteq \SectionSpaceAbb{V} \to \deRCohomology^{2}_{A, \psi}(M),
\end{equation}
where \( \SectionSpaceAbb{V} \) is a \( \SectionSpaceAbb{G}_{A, \psi} \)-invariant, open neighborhood of \( 0 \) in \( \deRCohomology^{1}_{A, \psi}(M) \).
Finally, the moduli space \( \SectionSpaceAbb{M}_\mu \) in a neighborhood of \( \equivClass{A, \psi} \) is modeled on the quotient \( \SectionMapAbb{s}^{-1}(0) \slash \SectionSpaceAbb{G}_{A, \psi} \).
In this way, we recover the well-known result concerning the local structure of \( \SectionSpaceAbb{M}_\mu \), see \eg \parencite[Proposition~2.2.16]{Nicolaescu2000}.

\subsection{Application: Pseudoholomorphic immersions}
\label{sec:pseudoholomorphicCurves}

In this section, we discuss the moduli space of pseudoholomorphic curves.
For the sake of simplicity, we consider only the simplest case of the universal Gromov–-Witten moduli space without marked points. 

Let \( (M, \omega) \) be a finite-dimensional symplectic manifold endowed with a compatible almost complex structure \( J \).
For every compact Riemann surface \( (\Sigma, j) \) with complex structure \( j \), consider the Cauchy--Riemann operator
\begin{equation}
	\difpBar_{j, J} \, u = \frac{1}{2} (\tangent u + J \circ \tangent u \circ j),
\end{equation}
where \( u: \Sigma \to M \) is a smooth map.
In the following, we restrict attention to the case when \( u \) is an immersion.
This is to ensure the existence of slices.
Denote the space of immersions of \( \Sigma \) into \( M \) by \( \ImmersionSpace(\Sigma, M) \).
It is an open and dense subset of the Fréchet manifold \( \sFunctionSpace(\Sigma, M) \), see \parencite[Corollary~6.13]{Michor1980}.
Let \( \beta \in \sHomology_2(M, \Z) \) be a homology class, and consider the open subset \( \ImmersionSpace_\beta(\Sigma, M) \) of \( \ImmersionSpace(\Sigma, M) \) consisting of immersions \( u: \Sigma \to M \) with \( u_* \equivClass{\Sigma} = \beta \). 
Denote the space of complex structures on \( \Sigma \) by \( \SectionSpaceAbb{J}(\Sigma) \).
As every almost complex structure on a surface is integrable, \( \SectionSpaceAbb{J}(\Sigma) \) can be identified with the space of reductions of the frame bundle \( \FrameBundle \Sigma \) to \( \UGroup(1) \), \ie, \( \SectionSpaceAbb{J}(\Sigma) = \sSectionSpace(\FrameBundle \Sigma \times_{\GLGroup(2, \R)} \GLGroup(2, \R) \slash \GLGroup(1, \C)) \), and thus it carries a natural Fréchet manifold structure.
In order to realize \( \difpBar_{j, J} \) as a map between infinite-dimensional manifolds, let us introduce the Fréchet vector bundle \( \SectionSpaceAbb{E} \to \SectionSpaceAbb{J}(\Sigma) \times \ImmersionSpace(\Sigma, M) \) whose fiber over \( (j, u) \in \SectionSpaceAbb{J}(\Sigma) \times \ImmersionSpace(\Sigma, M) \) is the space \( \DiffFormSpace^{0,1}(\Sigma, u^* \TBundle M) \) of smooth \( (j, J) \)-antilinear \( 1 \)-forms on \( \Sigma \) with values in \( u^* \TBundle M \).
The Cauchy--Riemann operator yields a smooth section
\begin{equation}
	\SectionMapAbb{CR}: \SectionSpaceAbb{J}(\Sigma) \times \ImmersionSpace_\beta(\Sigma, M) \to \SectionSpaceAbb{E}, \qquad (j, u) \mapsto \difpBar_{j, J} u .
\end{equation}
Note that \( \SectionMapAbb{CR} \) is equivariant with respect to the natural reparametrization actions of the Fréchet Lie group \( \DiffGroup (\Sigma) \) of diffeomorphisms of \( \Sigma \).
The associated moduli space
\begin{equation}
	\SectionSpaceAbb{M}_{\beta, J} = \SectionMapAbb{CR}^{-1}(0) \slash \DiffGroup(\Sigma)
\end{equation}
is the universal moduli space of unparameterized pseudoholomorphic immersions representing \( \beta \in \sHomology_2(M, \Z) \).
Here \( 0 \) denotes the zero section in \( \SectionSpaceAbb{E} \).

Let us verify that the assumptions of \cref{prop:normalFormEquivariantMap:elliptic} are met for the model under consideration.
To be precise, a slight variation of \cref{prop:normalFormEquivariantMap:elliptic} is necessary as \( \SectionSpaceAbb{M}_{\beta, J} \) is defined via the preimage of a submanifold instead of a single point.
We thus have to establish the existence of a normal form relative to the submanifold as in \cref{prop:normalFormMap:banachRelative}.
Since we are interested in the preimage of the zero section in \( \SectionSpaceAbb{E} \), this simply amounts to replacing the derivative of the section in~\eqref{eq:normalFormEquivariantMap:elliptic:chain} by its vertical derivative.
\begin{enumerate}
	\item 
		The zero section of \( \SectionSpaceAbb{E} \) is invariant under the action of  \( \DiffGroup (\Sigma) \), which is a geometric tame Fréchet Lie group with Lie algebra \( \VectorFieldSpace(\Sigma) \).
	\item
		The reparametrization action of \( \DiffGroup (\Sigma) \) on \( \ImmersionSpace(\Sigma, M) \) admits a slice \( \SectionSpaceAbb{S}_u \) at every immersion \( u \in \ImmersionSpace(\Sigma, M) \), see \parencite{CerveraMascaroEtAl1991}.
		According to \parencite[Lemma~3.1]{CerveraMascaroEtAl1991}, the stabilizer \( \DiffGroup_u(\Sigma) \) of \( u \) is a finite subgroup of \( \DiffGroup(\Sigma) \).
		Since every invariant open set is a slice for the action of a finite group, the induced action of \( \DiffGroup_u(\Sigma) \) on \( \SectionSpaceAbb{J}(\Sigma) \) admits a slice \( \SectionSpaceAbb{S}_j \) at every \( j \in \SectionSpaceAbb{J}(\Sigma) \).
		By \parencite[Proposition~3.29]{DiezSlice}, the \( \DiffGroup (\Sigma) \)-action on \( \SectionSpaceAbb{J}(\Sigma) \times \ImmersionSpace(\Sigma, M) \) has a slice \( \SectionSpaceAbb{S}_{j, u} = \SectionSpaceAbb{S}_j \times \SectionSpaceAbb{S}_u \) at \( (j, u) \).
	\item
		The action of \( \DiffGroup(\Sigma) \) on \( \SectionSpaceAbb{E} \) is fiberwise linear.
	\item
		The tangent space of \( \SectionSpaceAbb{J}(\Sigma) \) at \( j \) is the space of sections of the bundle \( \EndBundle_j (\TBundle \Sigma) \) whose fiber at \( x \in \Sigma \) is the space of linear maps \( \xi: \TBundle_x \Sigma \to \TBundle_x \Sigma \) satisfying \( j \circ \xi + \xi \circ j = 0 \).
		The vertical tangent map of \( \SectionMapAbb{CR} \) at a point \( (j, u) \in \SectionMapAbb{CR}^{-1}(0) \) is given by
		\begin{equation}\begin{split}
			\vtangent_{(j, u)} \SectionMapAbb{CR}: &\EndBundle_j (\TBundle \Sigma) \times \sSectionSpace(u^* \TBundle M) \to \DiffFormSpace^{0,1}(\Sigma, u^* \TBundle M),
			\\
			&(\xi, X) \mapsto \Bigl(\Dif_u \xi + \frac{1}{2} J \circ \tangent u \circ \xi\Bigr),
		\end{split}\end{equation}
		where \( \Dif_u: \sSectionSpace(u^* \TBundle M) \to \DiffFormSpace^{0,1}(\Sigma, u^* \TBundle M) \) is the usual linearized Cauchy--Riemann operator, \cf \parencite[Proposition~3.1.1]{McDuffSalamon2012}.
		As discussed above, we have to use the vertical derivative in place of the normal derivative in the chain~\eqref{eq:normalFormEquivariantMap:tame:chain}.
		It here takes the form
		\begin{equationcd}[label=eq:pseudoholomorphic:deformationChain]
			\VectorFieldSpace(\Sigma) \to[r, "{(\difLie l, - \tangent v)}"]
				& \EndBundle_j (\TBundle \Sigma) \times \sSectionSpace(u^* \TBundle M) \to[r, "\vtangent_{(l, v)} \SectionMapAbb{CR} \,"]
				& \DiffFormSpace^{0,1}(\Sigma, u^* \TBundle M),
		\end{equationcd}
		where \( l \in \SectionSpaceAbb{S}_j \) and \( v \in \SectionSpaceAbb{S}_u \).
		This chain is a chain of linear differential operators tamely parametrized by \( (l, v) \).
		At \(  (l,v) = (j,u) \), the chain~\eqref{eq:pseudoholomorphic:deformationChain} is a complex that, after ignoring the zeroth-order contributions (which are compact operators and thus do not change ellipticity or the index), is given by
		\begin{equationcd}[label=eq:pseudoholomorphic:deformationComplexReduced]
			\VectorFieldSpace(\Sigma) \to[r, "{(\difLie j, 0)}"]
				& \EndBundle_j (\TBundle \Sigma) \times \sSectionSpace(u^* \TBundle M) \to[r, "\Dif_u"]
				& \DiffFormSpace^{0,1}(\Sigma, u^* \TBundle M).
		\end{equationcd}
		This complex is elliptic because \( \difLie j: \VectorFieldSpace(\Sigma) \to \EndBundle_j (\TBundle \Sigma) \) is elliptic with index (\cf \parencites[Proposition~9.4.4]{Oh2015}[Theorem~C.1.10]{McDuffSalamon2012})
		\begin{equation}
			\ind \difLie j = (2 - 2g) + 2 \, \dualPair{\chernClass_1 (\TBundle^{1,0} \Sigma)}{\equivClass{\Sigma}} = 3 \, (2 - 2g)
		\end{equation}
		and \( \Dif_u: \sSectionSpace(u^* \TBundle M) \to \DiffFormSpace^{0,1}(\Sigma, u^* \TBundle M) \) is elliptic with index (\cf \parencite[Theorem~C.1.10]{McDuffSalamon2012})
		\begin{equation}
			\ind \Dif_u = n \, (2 - 2g) + 2 \, \dualPair{\chernClass_1 (u^* \TBundle M)}{\equivClass{\Sigma}},
		\end{equation}
		where \( 2n = \dim_\R M \).
		Thus, the Euler characteristic of~\eqref{eq:pseudoholomorphic:deformationComplexReduced} is given by
		\begin{equation}
			\ind \difLie j - \ind \Dif_u = (3 - n) \, (2 - 2g) - 2 \, \dualPair{\chernClass_1 (u^* \TBundle M)}{\equivClass{\Sigma}}. 
		\end{equation}
\end{enumerate}
Hence, by \cref{prop:normalFormEquivariantMap:elliptic} and the tame Fréchet version of \cref{prop:normalFormMap:banachRelative}, the map \( \SectionMapAbb{CR} \) can be brought into an equivariant normal form relative to the zero section at every pseudoholomorphic curve \( (j, u) \).
Moreover, as a consequence of \cref{prop:quotientsLevelSets:kuranishiByNormalForm,rem:quotientsLevelSets:fredholmFiniteDimKranishiStructure}, we obtain the following description of the local geometry of the moduli space \( \SectionSpaceAbb{M}_{\beta, J} \).

\begin{thm}
	Let \( \Sigma \) be a closed, oriented surface and let \( (M, \omega, J) \) be a symplectic manifold of dimension \( 2n \) endowed with a compatible almost complex structure \( J \).
	For every \( \beta \in \sHomology_2(M, \Z) \), the moduli space \( \SectionSpaceAbb{M}_{\beta, J} \) of unparameterized pseudoholomorphic immersions admits a finite-dimensional Kuranishi chart at every point \( \equivClass{j, u} \in \SectionSpaceAbb{M}_{\beta, J} \).
	Moreover, the virtual dimension of \( \SectionSpaceAbb{M}_{\beta , J} \) is given by
	\begin{equation}
		(n - 3) \, (2 - 2g) + 2 \, \dualPair{\chernClass_1 (M)}{\beta}.
		\qedhere
	\end{equation}
\end{thm}

In the study of pseudoholomorphic curves, one is usually not only interested in immersions but also in more general smooth or nodal curves.
We have restricted attention to the open subset \( \ImmersionSpace(\Sigma, M) \subseteq \sFunctionSpace(\Sigma, M) \) of immersions to ensure that the reparametrization action of \( \DiffGroup (\Sigma) \) admits slices.
The discussion in \parencite{CerveraMascaroEtAl1991} shows that, in general, one cannot expect the existence of slices for the reparametrization action on the space \( \sFunctionSpace(\Sigma, M) \) of all smooth maps.
In particular, the slice property \iref{i::slice:SliceDefOnlyStabNotMoveSlice} is problematic and it is not even clear whether the quotient \( \sFunctionSpace(\Sigma, M) \slash \DiffGroup (\Sigma) \) be Hausdorff.
However, as we have discussed in \cref{rem:quotientsLevelSets:kuranishiNotNeedsAllSliceProperties}, for the Kuranishi structure of the moduli space one only needs \iref{i::slice:SliceDefOnlyStabNotMoveSlice} to hold for solutions \( (j, u) \) of the Cauchy--Riemann equation \( \SectionMapAbb{CR}(j, u) = 0 \).
Such a weaker version of \iref{i::slice:SliceDefOnlyStabNotMoveSlice} can be established using Gromov's Compactness Theorem as in \parencites[p.~999]{FukayaOno1999}[Lemma~20.15]{FukayaOhOhtaEtAl2012}.
Using this observation, one can then also construct Kuranishi charts on the moduli space of unparameterized pseudoholomorphic (stable) curves.

The existence of Kuranishi charts on \( \SectionSpaceAbb{M}_{\beta, J} \) is well known, see, \eg, \parencite[Theorem~12.9]{FukayaOno1999} and references therein.
However, the traditional approach using Banach spaces faces serious technical problems due to the fact that the reparametrization action is not smooth with respect to maps of a given Sobolev class, see, for example, \parencite[Section~3]{McDuffWehrheim2015} for a detailed discussion.
These difficult issues have lead to considerable discussion about the correctness of the foundations of the theory of pseudoholomorphic curves, see \parencite{FukayaOhOhtaEtAl2012,McDuffWehrheim2015,FukayaOhOhtaEtAl2017,McDuffWehrheim2015a} and references therein.
Moreover, these issues inspired the development of the so-called scale calculus that provides the analytic backbone of the polyfold framework \parencite{HoferWysockiZehnder2017}.
Our construction using the Nash--Moser theorem has the important advantage that we circumvent such technical problems completely, since the reparametrization action of the diffeomorphism group on the space of smooth maps is a tame smooth action.
Thus, phrasing the problem in the well-established framework of tame Fréchet manifolds allows us to concentrate on the geometric constructions rather than the analytic details.

\section{Outlook}
In this paper, local normal form theorems for smooth equivariant maps between infinite-dimensional manifolds are established in various analytic settings.
As we have seen, these equivariant normal form theorems are a powerful tool to study the local structure of moduli spaces and to show that these moduli spaces carry the structure of a Kuranishi space, \ie, they are locally modeled on the quotient by a compact group of the zero set of a smooth map.
Using this general framework, we were able to give short and unified proofs that the moduli space of anti-self-dual instantons, the Seiberg–Witten
moduli space and the moduli space of pseudoholomorphic curves admit Kuranishi charts.

The normal forms developed in this paper are flexible enough to respect further geometric data.
In forthcoming work \parencite{DiezSingularReduction} (see also \parencite{DiezThesis}), we refine and adapt the techniques developed in this paper to Hamiltonian actions on infinite-dimensional symplectic manifolds.
We construct a normal form for equivariant momentum maps in the spirit of the classical Marle---Guillemin---Sternberg normal form.
The fundamental idea is to use \cref{prop:normalFormEquivariantMap:abstract} to bring the momentum map into a normal form and then gain control over the behavior of the singular part of the momentum map using the symplectic form.
With the help of this normal form, we then prove a singular symplectic reduction theorem including the analysis of the stratification into symplectic manifolds. 
This strategy is completely different to the traditional proof of the finite-dimensional Marle---Guillemin---Sternberg Theorem which uses an equivariant version of Darboux's Theorem.
A simple counterexample by \textcite{Marsden1972} shows that the Darboux Theorem fails spectacularly already for weakly symplectic Banach manifolds, so this approach is not possible in infinite dimensions.
In physics, one is mainly interested in the case where the symplectic manifold is a cotangent bundle.
A normal form of the momentum map for a lifted action on an infinite-dimensional cotangent bundle is established in \parencite{DiezRudolphReduction}.

Our results concerning equivariant normal forms in infinite dimensions and the techniques developed in the previous chapters open many exciting avenues for further research.
We list some relevant open problems and fundamental issues:
\begin{enumerate}
	\item 
		It would be very interesting to extend the discussion of the normal form of a smooth map to higher orders.
		One would expect that the knowledge of higher-order terms of the Taylor expansion of a smooth map \( f \) yields further control over the behavior of its singular part \( f_\singularPart \).
		In this way, one would gain deeper insight into the singular structure of the level sets of \( f \).
		This also opens the path toward an infinite-dimensional Morse theory.
	\item
		Studying the local structure of moduli spaces is the first step in an elaborated program to define powerful geometric invariants.
		Usually, one then passes to a compactification and constructs a virtual fundamental cycle for the compactified moduli space.
		Thus, it is desirable to have a general scheme in the Fréchet framework for the issues of compactification and perturbation, similarly to the aim of the polyfold theory of \textcite{HoferWysockiZehnder2017}.
	\item
		In \cref{sec:moduliSpaces}, we studied moduli spaces for which the deformation complex is elliptic.
		The Nash--Moser techniques also allow to study problems that are not accessible to elliptic methods.
		Thus, it would be interesting to apply our general framework to, say, KAM theory, to deformations of fibrations or to normal forms in Poisson geometry.   
\end{enumerate}

\appendix
\section{Inverse Function Theorems}
\label{sec:inverseFunctionTheorem}
\label{sec:tameFrechet}

In this section, we give a brief overview of different generalizations of the classical Inverse Function Theorem to the infinite-dimensional setting.
The primary focus is on Glöckner's Inverse Function Theorem for maps between Banach spaces with parameters in a locally convex space and on the Nash--Moser theorem in the tame Fréchet category.

As a reference point, let us recall the classical version of the Inverse Function Theorem in the Banach setting.
\begin{thm}[\textnormal{Banach version, \parencite[Theorem~I.5.2]{Lang1999}}]
	\label{prop:inverseFunctionTheorem:banach}
	Let \( X, Y \) be Banach spaces and let \( f: X \supseteq U \to Y \) be a smooth map defined on an open neighborhood \( U \) of \( 0 \) in \( X \).
	If \( \tangent_0 f: X \to Y \) is an isomorphism of Banach spaces, then \( f \) is a local diffeomorphism at \( 0 \).   
\end{thm}
\Textcite{Gloeckner2006a,Gloeckner2005a} has established the following generalization of the Banach Inverse Function Theorem to smooth maps depending on parameters in a locally convex space.
Similar results have been obtained in \parencite{Hiltunen1999} (using a slightly stronger notion of differentiability) and in \parencite{Teichmann2001} (using the so-called convenient calculus). 
\begin{thm}[\textnormal{Banach version with parameters, \parencite[Theorem~2.3]{Gloeckner2006a}}]
	\label{prop:inverseFunctionTheorem:banachWithParameters}
	Let \( P \subseteq E \) be an open neighborhood of \( 0 \) in the locally convex space \( E \).
	Let \( X, Y \) be Banach spaces, let \( U \) be an open neighborhood of \( 0 \) in \( X \) and let \( f: E \times X \supseteq P \times U \to Y \) be a smooth map.
	If the partial derivative \( \tangent^2_{(0, 0)} f: X \to Y \) of \( f \) at \( (0, 0) \) with respect to the second variable is an isomorphism of Banach spaces, then the map
	\begin{equation}
		E \times X \supseteq P \times U \to E \times Y,
		\quad
		(p, x) \mapsto \bigl(p, f(p, x)\bigr)
	\end{equation}
	is a local diffeomorphism at \( (0, 0) \).
\end{thm}

We now recall the main notions of the tame Fréchet category and the Nash--Moser Inverse Function Theorem, \cf \parencite{Hamilton1982}.
A Fréchet space \( X \) is called graded if it carries a distinguished increasing fundamental system of seminorms \( \normDot_k \).
A graded Fréchet space is called \emphDef{tame} if the seminorms satisfy an additional interpolation property, which formalizes the idea that \( X \) admits smoothing operators, see \parencite[Definition~II.1.3.2]{Hamilton1982} for the exact statement.
Let \( X \) and \( Y \) be tame Fréchet spaces. 
A continuous (possibly nonlinear) map \( f: X \supseteq U \to Y \) defined on an open subset \( U \subseteq X \) is \emphDef{\( r \)-tame} if it satisfies a local estimate of the form
\begin{equation}
	\norm{f(x)}_k \leq C (1 + \norm{x}_{k+r}).
\end{equation}
Roughly speaking, this means that \( f \) has a maximal loss of \( r \) derivatives.
Moreover, a smooth map \( f \) is called \emphDef{\( r \)-tame smooth} if \( f \) and all its derivatives \( \dif^{(j)} f: U \times X^j \to Y \) are \( r \)-tame.
Let \( Z \) be a tame Fréchet space and assume that \( Z \) is the topological direct sum of closed subspaces \( X \) and \( Y \).
We say that the sum \( Z = X \oplus Y \) is tame if the map \( X \times Y \to Z \) given by \( (x, y) \mapsto x + y \) is a tame isomorphism.

\begin{thm}[\textnormal{Nash--Moser Inverse Function Theorem, \parencite[Section~III.1]{Hamilton1982}}]
	\label{prop:inverseFunctionTheorem:nashMoser}
	Let \( X \) and \( Y \) be tame Fréchet spaces, let \( U \) be an open neighborhood of \( 0 \) in \( X \) and let \( f: X \supseteq U \to Y \) be a tame smooth map. 
	Assume that the derivative \( \tangent f \) has a tame smooth family \( \psi \) of inverses, that is, \( \psi: U \times Y \to X\) is a tame smooth map and \( \psi (x, \cdot): Y \to X \) is inverse to \( \tangent_x f \) for every \( x \in U \).
	Then the map \( f \) is a tame local diffeomorphism at \( 0 \).
\end{thm}
The important point is that the derivative of \( f \) has to be invertible in a \emph{neighborhood} of \( 0 \) and that one requires tame estimates for the inverses.


\section{Slices and orbit type stratification}
\label{sec:calculus:groupActionsSlices}

In this section, we give a brief account of the theory of Lie group actions in the category of locally convex manifolds.
The focus lies on slices for the action and the stratification of the manifold into orbit types.
We refer the reader to \parencite{DiezSlice} for more details.

Let \( M \) be a (locally convex) manifold.
Assume a (locally convex) Lie group \( G \) acts smoothly on \( M \), that is, assume that the action map \( G \times M \to M \) is smooth.
We refer to this setting by saying that \( M \) is a \( G \)-manifold.
The action is often written, using the dot notation, as \( (g, m) \mapsto g \cdot m \).
Similarly, the induced action of the Lie algebra \( \LieA{g} \) of \( G \) is denoted by \( \xi \ldot m \in \TBundle_m M \) for \( \xi \in \LieA{g} \) and \( m \in M \).
Clearly, \( m \mapsto \xi \ldot m \) is the Killing vector field generated by \( \xi \). 
Furthermore, $G \cdot m = \set{g \cdot m \given g \in G} \subseteq M$ is the orbit through \( m \in M \).
The \( G \)-action is called proper if inverse images of compact subsets under the map
\begin{equation}
	G \times M \to M \times M, \qquad (g, m) \mapsto (g \cdot m, m)
\end{equation}
are compact. 

The subgroup $G_m \defeq \set{g \in G \given g \cdot m = m}$ is called the \emphDef{stabilizer subgroup} of \( m \in M \).
It is not known, even for Banach Lie group actions, whether \( G_m \) is always a \emph{Lie} subgroup, see \parencite[Problem~IX.3.b]{Neeb2006}.
However, for proper actions this is the case, see \parencite[Lemma~2.11]{DiezSlice}.
In fact, then \( G_m \) is even a \emphDef{principal Lie subgroup} of \( G \), which means that the natural fibration \( G \to G \slash G_m \) is a principal bundle.
The \( G \)-action is called free if all stabilizer subgroups are trivial.
Two subgroups \( H \) and \( K \) of \( G \) are said to be conjugate if there exists \( a \in G \) such that \( a H a^{-1} = K \); we write \( H \sim K \) is this case.
In view of the equivariance relation \( G_{g \cdot m} = g G_m g^{-1} \), for every \( m \in M \) and \( g \in G \), we can assign to every orbit \( G \cdot m \) the conjugacy class \( (G_m) \), which is called the \emphDef{orbit type} of \( m \).
We put a preorder on the set of orbit types by declaring \( (H) \leq (K) \) for two orbit types, represented by the stabilizer subgroups \( H \) and \( K \), if there exists \( a \in G \) such that \( a H a^{-1} \subseteq K \).
If the action is proper, this preorder is actually a partial order.
For every closed subgroup \( H \subseteq G \), define the following subsets of \( M \):
\begin{align*}
	M_{H} &= \set{m \in M \given G_m = H},
	\\
	M_{(H)} &= \set{m \in M \given (G_m) = (H)}.
\end{align*}
The subset \( M_{H} \) is called the \emphDef{isotropy type subset} and \( M_{(H)} \) is the \emphDef{subset of orbit type \( (H) \)}.
Analogous definitions hold for every subset \( N \subseteq M \), so, for example, \( N_H = N \cap M_H \).

As in finite dimensions, the local structure of the orbit type subsets is studied with the help of slices, \cf \cref{prop:slice:orbitTypeSubsetIsSubmanifold} below.
Slices also play a fundamental role in the construction of the normal form of an equivariant map in \cref{sec:normalFormEquivariantMap}.   
\begin{defn}
	\label{defn:slice:slice}
	Let \( M \) be a \( G \)-manifold.
	A \emphDef{slice} at \( m \in M \) is a submanifold \( S \subseteq M \) containing \( m \) with the following properties:
	\begin{thmenumerate}[label=(SL\arabic*), ref=(SL\arabic*), leftmargin=*] 
		\item \label{i::slice:SliceDefSliceInvariantUnderStab}
			The submanifold \( S \) is invariant under the induced action of the stabilizer subgroup \( G_m \), that is \( G_m \cdot S \subseteq S \).

		\item \label{i::slice:SliceDefOnlyStabNotMoveSlice}
			Any \( g \in G \) with \( (g \cdot S) \cap S \neq \emptyset \) is necessarily an element of \( G_m \). 

		\item \label{i::slice:SliceDefLocallyProduct}
			The stabilizer \( G_m \) is a principal Lie subgroup of \( G \) and the principal bundle \( G \to G \slash G_m \) admits a local section \( \chi: G \slash G_m \supseteq W \to G \) defined on an open neighborhood \( W \) of the identity coset \( \equivClass{e} \) in such a way that the map
			\begin{equation}
				\chi^S: W \times S \to M, \qquad (\equivClass{g}, s) \mapsto \chi(\equivClass{g}) \cdot s
			\end{equation}
			is a diffeomorphism onto an open neighborhood of \( m \), which is called a \emphDef{slice neighborhood} of \( m \).

		\item \label{i::slice:SliceDefPartialSliceSubmanifold}
			The partial slice \( S_{(G_m)} = \set{s \in S \given G_s \text{ is conjugate to } G_m} \) is a closed submanifold of \( S \).

		\item
			\label{i:slice:linearSlice}
			There exist a smooth representation of \( G_m \) on a locally convex space \( X \) and a \( G_m \)-equivariant diffeomorphism \( \iota_S \) from a \( G_m \)-invariant open neighborhood of \( 0 \) in \( X \) onto \( S \) such that \( \iota_S(0) = m \).
			\qedhere
	\end{thmenumerate}
\end{defn}

The notion of a slice is closely related to the concept of a tubular neighborhood.
\begin{prop}[{\parencite[Proposition~2.6.2]{DiezSlice}}]
	\label{prop:slice:sliceToTube}
	Let \( M \) be a \( G \)-manifold.
	For every slice \( S \) at \( m \in M \), the tube map
	\begin{equation}
		\chi^\tube: G \times_{G_m} S \to M, \qquad \equivClass{g, s} \mapsto g \cdot s
	\end{equation}
	is a \( G \)-equivariant diffeomorphism onto an open, \( G \)-invariant neighborhood of \( G \cdot m \) in \( M \).
\end{prop}

In the finite-dimensional context, the existence of slices for proper actions is ensured by Palais' slice theorem \parencite{Palais1961}.
Passing to the infinite-dimensional case, this may no longer be true and additional hypotheses have to be made.
We refer the reader to \parencite{DiezSlice,Subramaniam1986} for general slice theorems in infinite dimensions and to \parencite{AbbatiCirelliEtAl1989,Ebin1970,CerveraMascaroEtAl1991} for constructions of slices for concrete examples.

As in the finite-dimensional case, the existence of slices implies many nice properties of the orbit space.
For example, we have the following.
\begin{prop}[{\parencite[Propositions~4.1 and~4.5]{DiezSlice}}]
	\label{prop:slice:orbitTypeSubsetIsSubmanifold}
	Let \( M \) be a \( G \)-manifold with proper \( G \)-action.
	If the \( G \)-action admits a slice at every point of \( M \), then \( M_{(H)} \) is a submanifold of \( M \).
	Moreover, \( \check{M}_{(H)} = M_{(H)} \slash G \) carries a smooth manifold structure such that the natural projection \( \pi_{(H)}: M_{(H)} \to \check{M}_{(H)} \) is a smooth submersion.  
\end{prop}
If, in addition, a certain approximation property is satisfied, then the orbit type manifolds fit together nicely and the orbit space is a stratified space, see \parencite[Theorem~4.2]{DiezSlice}.
More generally, we have the following stratification result for subsets of \( M \).
\begin{prop}[{\parencite[Proposition~4.7]{DiezSlice}}]
	\label{prop:slice:orbitTypeStratificationSubset}
	Let \( M \) be a \( G \)-manifold with proper \( G \)-action and let \( P \) be a closed \( G \)-invariant subset of \( M \).
	Assume that the \( G \)-action on \( M \) admits a slice \( S \) at every point \( m \in P \) such that the following holds:
	\begin{enumerate}
		\item 
			\( P \intersect S_{(G_m)} \) is a closed submanifold of \( S_{(G_m)} \).
		\item
			For every orbit type \( (K) \leq (G_m) \), the point \( m \) lies in the closure of \( P \intersect S_{(K)} \) in \( S \).
	\end{enumerate}
	Then, the induced partition of \( P \) into the orbit type subsets \( P_{(H)} = P \intersect M_{(H)} \) is a stratification.
	Moreover, under these assumptions, the decomposition of \( \check{P} = P \slash G \) into \( \check{P}_{(H)} = P_{(H)} \slash G \) is a stratification, too.
\end{prop}

For completeness, we include our definition of a stratification here and refer the reader to \parencite{DiezSlice} for further details and comparison with other notions of stratified spaces in the literature.
\begin{defn}
	Let \( X \) be Hausdorff topological space. 
	A partition \( \stratification{Z} \) of \( X \) into subsets \( X_\sigma \) indexed by \( \sigma \in \Sigma \) is called a \emphDef{stratification} of \( X \) if the following conditions are satisfied:
	\begin{thmenumerate}[label=(DS\arabic*), ref=(DS\arabic*), leftmargin=*]
		\item \label{i::stratification:stratumIsManifold} 
			Every piece \( X_\sigma \) is a locally closed, smooth manifold (whose manifold topology coincides with the relative topology).
			We will call \( X_\sigma \) a \emphDef{stratum} of \( X \).
 
		\item \label{i::stratification:frontierCondition} (frontier condition)
			Every pair of disjoint strata \( X_\varsigma \) and \( X_\sigma \) with \( X_\varsigma \cap \closureSet{X_\sigma} \neq \emptyset \) satisfies:
			\begin{thmenumerate}[label=\alph*), ref=(DS2\alph*)]
				\item \label{i:stratification:frontierConditionBoundary}
					\( X_\varsigma \) is contained in the frontier \( \closureSet{X_\sigma} \setminus X_\sigma \) of \( X_\sigma \),
				\item \label{i:stratification:frontierConditionIntersection}
					\( X_\sigma \) does not intersect \( \closureSet{X_\varsigma} \).
					\qedhere
			\end{thmenumerate}
	\end{thmenumerate}
\end{defn}

\begin{refcontext}[sorting=nyt]{}
	\printbibliography
\end{refcontext}

\end{document}